\pdfoutput=1
\documentclass[reqno,11pt,final]{amsart}
\usepackage{glaudo_pkg_en}

\title[Stability of Sobolev inequality with bubbling]
{On the sharp stability of critical points\\ of the Sobolev inequality}
\author{A. Figalli}
\address{ETH, Rämistrasse 101, 8092 Zürich, Switzerland}
\email{alessio.figalli@math.ethz.ch}
\author{F. Glaudo}
\address{ETH, Rämistrasse 101, 8092 Zürich, Switzerland}
\email{federico.glaudo@math.ethz.ch}

\begin{document}

\begin{abstract}
Given $n\geq 3$, consider the critical elliptic equation $\lapl u + u^{2^*-1}=0$ in $\R^n$ with $u > 0$.
This equation corresponds to the Euler-Lagrange 
  equation induced by the Sobolev embedding $H^1(\R^n)\hookrightarrow L^{2^*}(\R^n)$,
  and it is well-known that the solutions are uniquely characterized and are given by the so-called ``Talenti bubbles''. 
  In addition, thanks to a fundamental result by Struwe \cite{struwe1984}, this statement is ``stable up to bubbling'': if $u:\R^n\to\oo0\infty$ \emph{almost} solves
  $\lapl u + u^{2^*-1}=0$ then $u$ is (nonquantitatively) close in the $H^1(\R^n)$-norm to a sum of 
  weakly-interacting Talenti bubbles.
  More precisely, if $\dist(u)$ denotes the $H^1(\R^n)$-distance 
  of $u$ from the manifold of sums of Talenti bubbles, Struwe proved that $\delta(u)\to 0$
  as $\norm{\lapl u + u^{2^*-1}}_{H^{-1}}\to 0$.
  
  In this paper we investigate the validity of a sharp quantitative version of the stability for critical points: more precisely, we ask whether 
  under a bound on the energy $\norm{\nabla u}_{L^2}$ (that controls the number of bubbles) it holds
  \begin{equation*}
    \dist(u) \lesssim \norm{\lapl u + u^{2^*-1}}_{H^{-1}} \fullstop
  \end{equation*}
  A recent paper by the first author together with Ciraolo and Maggi \cite{cirfigmag2018} shows that the above result is true if $u$ is close to only one bubble. Here we prove, to our surprise, that whenever there are at least two bubbles then the estimate above is true for $3\le n\le 5$ while it is 
  false for $n\ge 6$.
  To our knowledge, this is the first situation where quantitative stability estimates depend so strikingly on the dimension of the space, changing completely behavior for some particular value of the dimension $n$.
\end{abstract}

\maketitle
\vspace{-0.8cm}

\tableofcontents

\section{Introduction}
The Sobolev inequality with exponent $2$ states that, for any $n\ge 3$ and any
$u\in H^1(\R^n)$, it holds
\begin{equation}\label{eq:sobolev_ineq_intro}
  S\norm{u}_{L^{2^*}} \le \norm{\nabla u}_{L^2} \comma
\end{equation}
where $2^*=\frac{2n}{n-2}$ and $S=S(n)$ is a dimensional constant. In this paper we denote by $H^1(\R^n)$ the closure
of $C^{\infty}_c(\R^n)$ with respect to the norm $\norm{\nabla u}_{L^2}$.
Also, whenever a norm is computed on the whole $\R^n$, we do not specify the domain (so, for instance, $\norm{\cdot}_{L^2}=\norm{\cdot}_{L^2(\R^n)}$).

The optimal value of the constant $S$ is known and so are the optimizers of the Sobolev inequality (see 
\cite{aubin1976,talenti1976}): the functions that satisfy the equality in \cref{eq:sobolev_ineq_intro} 
have the form
\begin{equation*}
  \frac{c}{(1+\lambda^2\abs{x-z}^2)^{\frac {n-2}2}} \comma
\end{equation*}
where $c\in \R$, $\lambda\in\oo{0}{\infty}$, and $z\in\R^n$ can be chosen arbitrarily.
Let us define a subclass of all the optimizers, that is the parametrized family of functions $U[z,\lambda]$, 
with $z\in\R^n$ and $\lambda>0$, defined as
\begin{equation}\label{eq:talenti_intro}
  U[z,\lambda](x) \defeq 
  (n(n-2))^{\frac{n-2}4}\lambda^{\frac{n-2}2}\frac{1}{(1+\lambda^2\abs{x-z}^2)^{\frac{n-2}2}} \fullstop
\end{equation}
We will call such functions \emph{Talenti bubbles}. Later it will be clear why we want to put a specific dimensional
constant in the definition of Talenti bubbles.

Once \cref{eq:sobolev_ineq_intro} is established, it is natural to look for a quantitative version.
Informally, we wonder if \emph{almost} satisfying the equality in 
\cref{eq:sobolev_ineq_intro} implies being \emph{almost} a Talenti bubble up to scaling.

One of the most natural ways to state this question is to ask if the discrepancy 
$\norm{\nabla u}_{L^2}^2-S^2\norm{u}_{L^{2^*}}^2$ of a function $u\in H^1(\R^n)$ can bound
the distance of $u$ from a rescaled Talenti bubble.
The answer is positive as shown in \cite{bianchi1991}, where the authors prove that for any 
$u\in H^1(\R^n)$ it holds
\begin{equation*}
  \inf_{\substack{z\in\R^n, \lambda>0 \\ c\in\R}}\norm*{\nabla(u-c U[z,\lambda])}_{L^2}^2
  \le C(n) \left(\norm{\nabla u}_{L^2}^2-S^2\norm{u}_{L^{2^*}}^2\right) \comma
\end{equation*}
where $C(n)$ is a dimensional constant.

A different (and more challenging) way to approach the question is to consider the Euler-Lagrange equation associated
to the inequality \cref{eq:sobolev_ineq_intro}. It is well-known that the Euler-Lagrange equation
is, up to a suitable scaling, given by
\begin{equation}\label{eq:euler_lagrange}
  \lapl u + u\abs{u}^{2^*-2} = 0 \fullstop
\end{equation}
Notice that our definition of Talenti bubbles \cref{eq:talenti_intro} is such that every Talenti bubble
solves exactly \cref{eq:euler_lagrange}.
Passing from the inequality to the Euler-Lagrange equation is analogous to passing from minimizers to general
critical points.

In na\"ive terms, the topic of the current paper is to investigate whether a function $u$ that \emph{almost} 
solves \cref{eq:euler_lagrange} must be quantitatively close to a Talenti bubble (scaling is 
not necessary since the equation is nonlinear).
There is a number of fundamental obstructions to consider.

First and foremost, Talenti bubbles do not constitute all the solutions of \cref{eq:euler_lagrange}. Indeed,
as was shown in \cite{ding1986}, there are many other \emph{sign-changing} solutions on $\R^n$.
However, if we restrict to nonnegative functions, then, according to \cite{gidas1979}, the family of Talenti bubbles
are the only solutions.

There is another major obstruction to take care of. If we set 
$u\defeq U_1 + U_2$, where $U_1$ and $U_1$ are two \emph{weakly-interacting} Talenti bubbles (for instance
$U_1=U[-Re_1,1]$ and $U_2=U[Re_1,1]$ with $R\ggg 1$), then $u$ will approximately solve 
\cref{eq:euler_lagrange} in any reasonable sense. 
At the same time $u$ is not close to a \emph{single} Talenti bubble.
Hence we have to accept that even if $u$ almost solves \cref{eq:euler_lagrange} it might be close to a sum
of weakly-interacting bubbles.

In fact this is always the case, as proven in the seminal work \cite{struwe1984}. 
Let us recall the mentioned theorem in the form we will need:
\begin{theorem}[Struwe, 1984]\label{thm:struwe_intro}
  Let $n\ge 3$ and $\nu\ge 1$ be positive integers.
  Let $(u_k)_{k\in\N}\subseteq H^1(\R^n)$ be a sequence of nonnegative functions such that 
  $(\nu-\frac12)S^n\le \int_{\R^n}\abs{\nabla u_k}^2 \le (\nu+\frac12)S^n$ with $S=S(n)$ as in \cref{eq:sobolev_ineq_intro}, and assume that
  \begin{equation*}
    \norm{\lapl u_k + u_k^{2^*-1}}_{H^{-1}} \to 0 \quad\text{as $k\to\infty$}\fullstop
  \end{equation*}
  Then there exist a sequence $(z^{(k)}_1,\dots,z^{(k)}_\nu)_{k\in\N}$ of $\nu$-tuples of points in $\R^n$ 
  and a sequence $(\lambda^{(k)}_1,\dots,\lambda^{(k)}_\nu)_{k\in\N}$ of $\nu$-tuples of positive real numbers 
  such that
  \begin{equation*}
    \norm*{\nabla\left(u_k-\sum_{i=1}^\nu U[z_i^{(k)},\lambda_i^{(k)}]\right)}_{L^2}\to 0 
    \quad\text{as $k\to\infty$}\fullstop
  \end{equation*}
\end{theorem}
Let us remark that the assumptions on the sequence $(u_k)_{k\in\N}$ required in \cref{thm:struwe_intro} are 
equivalent to saying that $(u_k)_{k\in\N}$ is a Palais-Smale sequence for the functional
\begin{equation*}
  J(u) \defeq \frac12\int_{\R^n}\abs{\nabla u}^2 - \frac1{2^*}\int_{\R^n}u^{2^*} \fullstop
\end{equation*}
Hence, a different way to see the mentioned result is: all \emph{critical points at infinity} 
of the functional $J$ are induced by limits of sums of Talenti bubbles (at least if we consider only 
nonnegative functions).

\subsection{Main results}
It is now natural (and useful for applications) to look for a quantitative version of \cref{thm:struwe_intro}.
Considering $J$ as an energy, in analogy with the finite-dimensional setting, we expect that
$\norm{\de J(u)}_{H^{-1}}=\norm{\lapl u + u\abs{u}^{2^*-2}}_{H^{-1}}$ bounds the distance between $u$ 
and the manifold of approximate critical points (namely the sums of weakly-interacting Talenti bubbles).  In addition, a series of results both on this problem and to analogous stability questions for critical points (see for instance \cite{cirfigmag2018,ciraolo2018}) suggests that the control should be linear.
Let us state clearly the problem we want to investigate.
\begin{problem}\label{prob:intro}
  Let $n\ge 3$ and $\nu\ge 1$ be positive integers. 
  Let $(z_i,\lambda_i)_{1\le i\le \nu}\subseteq \R^n\times\oo0\infty$ be a $\nu$-tuple such that for any
  $i\not=j$ it holds
  \begin{equation}\label{eq:condition:intro}
    \min\left(\frac{\lambda_i}{\lambda_j}, 
	      \frac{\lambda_j}{\lambda_i}, 
	      \frac{1}{\lambda_i\lambda_j\abs{z_i-z_j}^2}\right) 
    \le \delta \fullstop
  \end{equation}
  Setting $\sigma\defeq\sum_{i=1}^{\nu}U[z_i,\lambda_i]$, let $u\in H^1(\R^n)$ satisfy 
  $\norm{\nabla u-\nabla\sigma}_{L^2} \le \delta$ for some $\delta=\delta(n,\nu)>0$ small enough.
  Does it exist a constant $C=C(n,\nu)>0$ such that the bound
  \begin{equation}\label{eq:prob:intro}
    \inf_{\substack{ (z'_i)_{1\le i\le\nu}\subseteq\R^n\\
		     (\lambda_i')_{1\le i\le\nu}\subseteq \oo0\infty}}
      \norm*{\nabla \left(u - \sum_{i=1}^{\nu}U[z_i',\lambda_i']\right)}_{L^2} 
      \le C \norm{\lapl u + u\abs{u}^{2^*-2}}_{H^{-1}}
  \end{equation}
  holds true?
\end{problem}
Let us remark that the condition \cref{eq:condition:intro} has to be understood as a requirement of 
weak-interaction between the Talenti bubbles $(U[z_i,\lambda_i])_{1\le i\le\nu}$.

We have shifted our attention from the set of nonnegative functions (recall that nonnegativity is necessary to 
classify exact solutions) to the set of functions in the 
neighborhood of a sum of weakly-interacting Talenti bubbles. 
The latter is more in line with the spirit of the problem. 
In fact our investigation is mainly local, as we want to understand if the quantity 
$\norm{\lapl u + u\abs{u}^{2^*-2}}_{H^{-1}}$ grows linearly in the distance from the manifold of sums 
of weakly-interacting Talenti bubbles.
Moreover it is easy to recover the result for nonnegative functions from the local result (as
we will do in the proof of \cref{cor:main_pos}) and to construct a nonnegative counterexample from a local one
(as we will do in the proof of \cref{thm:counterexample_pos}). 

As shown in the recent paper \cite{cirfigmag2018}, \cref{prob:intro} has a positive answer in any dimension when $\nu=1$ (i.e. only one bubble is present). 
Hence, it is natural to conjecture that a positive answer should hold also in the general case $\nu \geq 2$.

The main results of the paper show that \cref{prob:intro} has a positive answer if the dimension satisfies 
$3\le n\le 5$ (this is \cref{thm:main_close} and \cref{cor:main_pos}), whereas it is false if $\nu \geq 2$ and $n\ge 6$ 
(as shown in \cref{thm:counterexample} and \cref{thm:counterexample_pos}).

To show an application of our stability result, in \cref{sec:fdi} we obtain
a quantitative rate of convergence to equilibrium for a critical fast diffusion equation related to the Yamabe flow.
This result already appeared in \cite{cirfigmag2018} but the proof there contains a gap that we fix here.

\subsection{Comments and remarks}
We postpone a thorough description of the strategy of the proofs to the introductions of 
\cref{sec:pos,sec:counterexample}.
Here we gather a handful of general comments and remarks.
\begin{itemize}
 \item In order to show that \cref{prob:intro} is false in high dimension, we build a 
 family of functions $u_R$ that satisfy the assumption for an arbitrary small $\delta$ but do
 not satisfy the inequality \cref{eq:prob:intro} for any fixed $C$.
 The functions in the family are constructed starting from the solutions of a partial differential equation
 that is a linearization of $\lapl u + u\abs{u}^{2^*-2}=0$ near a sum of weakly-interacting
 Talenti bubbles.
 It is remarkable how counterintuitive and implicit this family of counterexamples is. 
 It is counterintuitive, since one would expect that if the function $u$ is \emph{incredibly} close to a 
 sum of \emph{incredibly} weakly-interacting bubbles, then inequality \cref{eq:prob:intro} might be recovered
 from the same inequality when only one bubble is involved (as a consequence of the ``independence'' 
 among the bubbles).
 It is implicit, as the mentioned partial differential equation (that is \cref{eq:def_rho}) cannot be solved
 explicitly and moreover the datum of the equation (that is $\tilde{f}$ in \cref{eq:def_rho}) is itself 
 defined implicitly.
 \item We are not in a position to claim \emph{why} \cref{prob:intro} fails in high dimension. Our proofs
 seems to indicate that the numerological reason of the failure is the fact that in dimension $n \le 5$ it 
 holds $2^*-2>1$, whereas in dimension $n\ge6$ it holds $2^*-2\le 1$.
 \item
As a consequence of the strategy we employed for building the counterexample, we needed to establish a number 
of properties on eigenfunctions of operators of the form $\frac{-\lapl}{\w}$, where $\w\in L^{\frac n2}(\R^n)$ is a 
positive weight. These results are stated and proven in \cref{app:spectrum}. 
The theory developed in the appendix contains several new results that might be of independent interest.
 \item Although the techniques developed in this paper do not provide any positive result in dimension $n \geq 6$, we believe that in dimension $n=6$ a weaker 
 version of \cref{prob:intro} might hold, where the right-hand side of 
 \cref{eq:prob:intro} is replaced by
 \begin{equation*}
    \norm{\lapl u + u\abs{u}^{2^*-2}}_{H^{-1}}
    \abs*{\log\left(\norm{\lapl u + u\abs{u}^{2^*-2}}_{H^{-1}}\right)} \comma
     \end{equation*}
 while for $n \geq 7$ one may replace it with
 \begin{equation*}
  \norm{\lapl u + u\abs{u}^{2^*-2}}_{H^{-1}}^\gamma\quad \text{for some }\gamma=\gamma(n)<1\fullstop
  \end{equation*}
However, we do not address this question here.
 \item There is a more geometrical perspective on \cref{prob:intro} described in the introduction of
 \cite{cirfigmag2018}. 
 We give only a sketch of this point of view. Let $(\S^n,g_0)$ be the $n$-dimensional sphere endowed
 with its standard Riemannian structure. Let $v:\S^n\to\oo0\infty$ be a conformal factor and let 
 $g=v^{2^*-2}g_0$ be the induced metric.
 The equation satisfied by the scalar curvature $R:\S^n\to\R$ of the metric $g$ is
 \begin{equation}\label{eq:intro:yamabe}
    -\lapl_{g_0}v + \frac{n(n-2)}4v=\frac{n-2}{n-1}Rv^{2^*-1} \fullstop
 \end{equation}
 Let us consider $u:\R^n\to\oo0\infty$ defined as 
 \begin{equation*}
    u(x) = \left(\frac{2}{1+\abs{x}^2}\right)^{\frac{n-2}2}v(F(x)) \comma
 \end{equation*}
 where $F(x)\defeq \left(\frac{2x}{1+\abs{x}^2}, \frac{\abs{x}^2-1}{1+\abs{x}^2}\right)$ is the 
 stereographic projection. In this new coordinates, \cref{eq:intro:yamabe} becomes
 \begin{equation*}
    \lapl u + R(F(x))u^{2^*-1} = 0 \fullstop
 \end{equation*}
 Hence \cref{prob:intro} can be interpreted also as a statement on the metrics on the sphere, conformal to 
 the standard one, that have \emph{almost constant} scalar curvature.
\end{itemize}

\subsection{Analogies with the isoperimetric inequality and Alexandrov's Theorem}\label{subsec:isoperimetric}
The Euclidean isoperimetric inequality states that for any $E\subseteq\R^n$ in a suitable family of sets 
(i.e. open sets with smooth boundary or finite perimeter sets) it holds
\begin{equation*}
  \abs{E}^\frac{n-1}{n} \le C_n\Per(E) \comma
\end{equation*}
where $C_n$ is a dimensional constant.
There is a strong parallel between the Sobolev inequality and the isoperimetric inequality, as the latter
is an instance of the former with exponent $1$.
It makes perfect sense, and in fact there is a rich literature on the topic, to study quantitative
versions of the isoperimetric inequality analogous to the ones we described for the Sobolev inequality.
Let us briefly recall some of the known results.

It was first proven by De Giorgi in \cite{degiorgi1958} that all minimizers of the isoperimetric inequality
are balls (analogous to the fact that Talenti bubbles are the only minimizers for the Sobolev inequality).
The next step is of course to understand whether a set $E$ with rescaled isoperimetric ratio 
\begin{equation*}
  \frac{\abs{E}^{\frac{n-1}{n}}}{\Per(E)} 
  \cdot \left(
  \frac{\abs{B(0,1)}^{\frac{n-1}{n}}}{\Per(B(0,1))}
  \right)^{-1}
\end{equation*}
very close to $1$ must be close to a ball (analogous to the result by \cite{bianchi1991} for the Sobolev inequality).
This quantitative stability of the isoperimetric inequality in a sharp form and in arbitrary dimension
has been first established in 
\cite{fusco2008}, and then obtained again in \cite{figalli2010} with optimal transportation methods
and by \cite{cicalese2012} with a penalization approach. 
See the survey \cite{maggi2008} for a more detailed history of the problem.

Then we move to the Euler-Lagrange equation induced by the isoperimetric inequality: the mean-curvature 
 of the boundary of $E$ must be constant.
As in the functional setting we asked whether the Talenti bubbles are the only solutions of 
$\lapl u + u\abs{u}^{2^*-2}=0$, in the geometrical setting we ask whether the spheres are the only 
(closed, compact, connected) hypersurfaces with constant mean-curvature. 
Remarkably the answer is negative in both cases without further assumptions. 
In the functional setting we require the nonnegativity of $u$, whereas in the
geometrical setting we need to ask that the hypersurface is embedded (otherwise Wente's torus is a
counterexample \cite{wente1986}). With this additional assumption the desired statement is the celebrated 
Alexandrov's Theorem (see \cite{alexandrov1962} for the original proof, and  \cite{delgadino2019} for 
the statement in the class of finite perimeter sets).

With all these results in our toolbox, we can now approach the stability problem:
if the boundary of $E$ has \emph{almost} constant mean curvature, is $E$ close to a ball? 
Exactly as in the functional setting, this is not the case (on the contrary, the answer is positive for
the analogue of this problem for the nonlocal perimeter \cite{ciraolo2018}). 
In fact, it is possible to build a \emph{chain} of balls (see for example \cite{butscher2011}) 
such that the mean curvature is uniformly close to a constant. This fact is absolutely analogous to the fact 
that if $\lapl u + u\abs{u}^{2^*-2}$ is very small, it might be that $u$ is close to a sum of multiple 
Talenti bubbles. 
As shown recently in \cite[Theorem 1.1]{ciraolo2017}, this is the only case: 
if $E\subseteq\R^n$ has isoperimetric ratio bounded by $L\in\N$, then there exists 
a union $G$ of at most $L$ balls such that $\frac{|E\bigtriangleup G|}{|E|}$ is bounded by a power 
of the $L^\infty$-oscillation of the mean curvature of $\partial E$.
Let us emphasize that the spirit of this statement is exactly the same of \cref{prob:intro}.
The only shortcoming of this result is its lack of sharpness:
\begin{itemize}
 \item The norms considered are not the most natural ones, as the natural norm would be the $L^2$-oscillation of the mean curvature.
 Let us remark that in 
 \cite[Theorem 1.1]{delgadino2018} the authors obtain a stability estimate with the  $L^2$-oscillation,
but the result is nonquantitative (in analogy with Struwe's result \cite{struwe1984}).
 \item The power of the oscillation of the mean curvature that controls 
 $\frac{|E\bigtriangleup G|}{|E|}$ is arguably not the sharp one.
\end{itemize}
Our results (i.e. the positive answer to \cref{prob:intro} for $n\le 5$, and the negative answer for $n\ge 6$ and $\nu \geq 2$)
makes one wonder whether a sharp version of \cite[Theorem 1.1]{ciraolo2017} with the \emph{natural} exponent (i.e. 1) and the natural norm (i.e. the $L^2$-oscillation) might fail in high dimension.

\subsection{Structure of the paper}
After a section of notation and preliminaries, in \cref{sec:pos} we give a positive answer to \cref{prob:intro} in low dimensions $3\le n\le 5$, and we obtain a couple of easy corollaries in 
\cref{subsec:corollaries}. 
Then, in \cref{sec:counterexample} we show that the conjecture
cannot hold if $n\ge 6$ and $\nu \geq 2$.
Finally, in \cref{sec:fdi} we prove the result concerning the fast diffusion equation.

This work contains also two appendices. 
The first one, \cref{app:spectrum}, is devoted to the investigation of the spectral properties of 
the operator $\left(\frac{-\lapl}{\w}\right)^{-1}$ where $\w\in L^{\frac n2}(\R^n)$ 
is a positive weight. The properties shown are of fundamental importance in the construction of the 
counterexample, and we believe that several of the results have  their own interest.
Finally, in \cref{app:computations} we collect a couple of statements useful to estimate and approximate
various type of integrals involving the Talenti bubbles.

\section{Notation and preliminaries}
We begin by setting the notation and the definitions that we will use throughout the paper.

We denote by $n\in\N$ the dimension of the ambient space. 
Since we are interested in  the Sobolev embedding with exponent $2$, we will always assume $n\ge 3$.

We recall that the Sobolev exponent is given by $2^*=\frac{2n}{n-2}$, and we define $p\defeq 2^*-1=\frac{n+2}{n-2}$. Given $q \in [1,\infty]$, we denote by $q'=\frac{q}{q-1}$ the H\"older conjugate of $q$.
The following identities will be useful: 
\begin{align*}
  &(2^*)' = \frac{2n}{n+2}=\frac{2^*}p \comma  &p' = \frac{n+2}4 \fullstop
\end{align*}
For any $z\in\R^n$ and $\lambda>0$, the Talenti bubble $U[z,\lambda]$ is defined as in \cref{eq:talenti_intro}.
Let us recall that, according to \cite{aubin1976,talenti1976}, this family of functions constitutes 
(up to scaling) the set of all minimizers of the Sobolev inequality.

Let $S>0$ be the sharp Sobolev constant in $\R^n$, that is
\begin{equation*}
  S \defeq \inf \left\{\frac{\norm{\nabla u}_{L^2}}{\norm{u}_{L^{2^*}}}\ :\ u\in H^1(\R^n)\setminus \{0\}\right\} \fullstop
\end{equation*}
Setting $U=U[z,\lambda]$, as a consequence of the dimensional constant we have chosen in the definition of
the Talenti bubbles, it holds
\begin{equation*}
  \int_{\R^n} U^{2^*} = \int_{\R^n} \abs{\nabla U}^2 = S^n \fullstop
\end{equation*}
Moreover, the Talenti bubble and its derivatives satisfy
\begin{equation} \label{eq:u_solves_yamabe}
  -\lapl U = U^p \comma 
  \qquad -\lapl (\partial_\lambda U) = pU^{p-1}\partial_\lambda U
  \comma
  \qquad -\lapl (\nabla_z U) = pU^{p-1}\nabla_z U
  \fullstop
\end{equation}
Finally, we have the following expression for the $\lambda$-derivative of the Talenti bubble:
\begin{equation}\label{eq:U_derlambda}
    \partial_\lambda U(x) 
    = \frac{n-2}{2\lambda} U(x) \left(\frac{1-\lambda^2\abs{x-z}^2}{1+\lambda^2\abs{x-z}^2}\right) \fullstop
\end{equation}

Let us recall the definitions of homogeneous Sobolev space and of weighted Lebesgue space.
\begin{definition}[Homogeneous Sobolev space]
  For any $1\le p < \infty$, the homogeneous Sobolev space $\dot W^{1,p}(\R^n)$ is the 
  closure of $C^{\infty}_c(\R^n)$ with respect to the norm 
  \begin{equation*}
    \norm{u}_{\dot W^{1,p}} \defeq \norm{\nabla u}_{L^p} \fullstop
  \end{equation*}
  The space $\dot W^{1,2}(\R^n)$ will be called $H^1(\R^n)$.
\end{definition}
Usually the notation $\dot H^1(\R^n)$ is adopted to denote $\dot W^{1,2}$, we decided to drop the dot as 
we will never use the \emph{standard} $W^{1,2}$.
\begin{definition}[Weighted Lebesgue space]
  Let $E\subseteq \R^n$ be a Borel set and let $\w\in L^1_{loc}(E)$ be a positive function.
  For any $1\le p < \infty$, the weighted Lebesgue space $L^p_\w(E)$ is the space of measurable 
  functions $f:E\to\R$ such that
  \begin{equation*}
    \int_E \abs{f}^p \w < \infty \fullstop
  \end{equation*}
  The norm on $L^p_\w(E)$ is
  \begin{equation*}
    f \mapsto \norm{f}_{L^p_{\w}(E)}=\left(\int_E f^p\w\right)^{\frac1p} \fullstop
  \end{equation*}
\end{definition}
The reason why weighted spaces happen to play a role in our treatment will be evident in 
\cref{subsec:spectral_intro}.

\subsection{Symmetries of the problem}\label{subsec:symmetries}
Given $\lambda>0$ and $z\in\R^n$, let $T_{z,\lambda}:C^\infty_c(\R^n)\to C^\infty_c(\R^n)$ be the operator 
defined as
\begin{equation*}
  T_{z,\lambda}(\varphi)(x) \defeq \lambda^{\frac{n-2}2}\varphi(\lambda(x-z)) \fullstop
\end{equation*}

The operator $T_{z,\lambda}$ satisfies a multitude of properties.
\begin{itemize}
 \item For any couple of functions $\varphi,\psi\in C^{\infty}_c(\R^n)$, it holds
 \begin{equation*}
    T_{z,\lambda}(\varphi\cdot\psi)(x) = T_{z,\lambda}(\varphi)(x)\cdot\psi(\lambda(x-z)) \fullstop
 \end{equation*}
 \item Given $k\in\N$, for any choice of positive exponents $(e_i)_{1\leq i \leq k}$ with $e_1+\cdots+e_k=2^*$, and for any choice of nonnegative
 functions $\varphi_1,\dots,\varphi_k\in C^\infty_c(\R^n)$, it holds
 \begin{equation*}
  \int_{\R^n} T_{z,\lambda}(\varphi_1)^{e_1}\cdots T_{z,\lambda}(\varphi_k)^{e_k} = 
  \int_{\R^n}\varphi_1^{e_1}\cdots\varphi_k^{e_k}
 \end{equation*}
 and in particular
 \begin{equation*}
  \int_{\R^n} T_{z,\lambda}(\varphi)^{2^*} = \int_{\R^n} \varphi^{2^*} 
 \end{equation*}
 for any $\varphi\in C^\infty_c(\R^n)$.
 \item For any pair of functions $\varphi,\psi\in C^\infty_c(\R^n)$ it holds
 \begin{equation*}
  \int_{\R^n} \nabla T_{z,\lambda}(\varphi)\cdot \nabla T_{z,\lambda}(\psi) 
  = \int_{\R^n} \nabla\varphi\cdot\nabla\psi
 \end{equation*}
 and in particular
 \begin{equation*}
  \int_{\R^n} \abs{\nabla T_{z,\lambda}(\varphi)}^2 = \int_{\R^n} \abs{\nabla\varphi}^2 \fullstop
 \end{equation*}
 \item As a consequence of their definition, the Talenti bubbles satisfy
 \begin{equation*}
    U[z,\lambda] = T_{z,\lambda}(U[0, 1]) \quad\text{and}\quad
    \partial_\lambda U[z,\lambda] = \frac1\lambda T_{z,\lambda}(\partial_\lambda U[0, 1]) \fullstop
 \end{equation*}
\end{itemize}
Obviously all the mentioned properties hold also if the functions are not smooth with compact support,
provided that the involved integrals are finite.

The transformations $T_{z,\lambda}$ play a central role in the study of the Sobolev inequality as they
do not change the two quantities $\norm{\varphi}_{L^{2^*}}$ and $\norm{\nabla\varphi}_{L^2}$.
In particular we will often use this symmetries to reduce ourselves to the situation where, instead of considering a generic 
Talenti bubble, we can take the bubble $U[0,1]$.

\subsection{Properties and spectrum of \texorpdfstring{$\left(\frac{-\lapl}{w}\right)^{-1}$}{the weighted Laplacian}}
\label{subsec:spectral_intro}
Both in \cref{sec:pos} and in the construction of the counterexample
(\cref{sec:counterexample}), a fundamental role will be played by the spectrum of 
$\left(\frac{-\lapl}{U^{p-1}}\right)^{-1}$ where $U$ is a Talenti bubble, and more in general by the spectrum of
the operator $\left(\frac{-\lapl}{\w}\right)^{-1}$ where $\w\in L^{\frac n2}(\R^n)$ is a suitable positive 
weight.

We note that operator $\left(\frac{-\lapl}{\w}\right)^{-1}$ is well-defined, compact, and self-adjoint from 
$L^2_\w(\R^n)$ into $L^2_\w(\R^n)$, therefore it has a discrete spectrum. 
This fundamental fact, together with many more properties of the spectrum, is contained in \cref{app:spectrum}. 
We will always consider the the eigenvalues of $\frac{-\lapl}{\w}$ instead of those of the inverse operator. 
We adopt this convention as it is more natural to write $-\lapl\psi = \lambda w\psi$ compared to 
$-\lambda\lapl\psi = w\psi$.

The properties of the spectrum of $\left(\frac{-\lapl}{U^{p-1}}\right)^{-1}$, where $U$ is a Talenti bubble, have already been investigated in \cite[Appendix]{bianchi1991}.

\section{Sharp stability in dimension \texorpdfstring{$3\le n\le 5$}{below 6}}\label{sec:pos}
In this whole section we consider the dimension $n$ and the number of bubbles $\nu$ as fixed.
Therefore constants that depend only on $n$ and $\nu$ can be hidden in the notation $\lesssim$ and 
$\approx$. More precisely, we write that $a \lesssim b$ (resp. $a\gtrsim b$) if $a \leq Cb$ (resp. $Ca \geq b$) where $C$ is a constant depending only on the dimension $n$ and on the number of bubbles $\nu$. Also, we say that $a\approx b$ if $a\lesssim b$ and $a \gtrsim b$.

We will deal with \emph{weakly-interacting} family of Talenti bubbles. A family $\{U[z_i, \lambda_i]\}_{1\leq i \leq \nu}$ is weakly-interacting
if either the \emph{centers} $z_i$ of the Talenti bubbles are very far one from the other, or
their scaling factors $\lambda_i$ have different magnitude. It is useful to give a quantitative definition of
the \emph{amount of interaction} that a certain family of Talenti bubbles has.
\begin{definition}[Interaction of Talenti bubbles]
  Let $U_1=U[z_1, \lambda_1],\dots, U_\nu=U[z_\nu, \lambda_\nu]$ be a family of Talenti bubbles.
  We say that the family is $\delta$-interacting for some $\delta>0$ if
  \begin{equation}\label{eq:def_delta_interaction}
    \min\left(\frac{\lambda_i}{\lambda_j}, \frac{\lambda_j}{\lambda_i}, 
    \frac{1}{\lambda_i\lambda_j\abs{z_i-z_j}^2}\right)
    \le \delta \fullstop
  \end{equation}
  If together with the family we have also some positive coefficients $\alpha_1,\dots,\alpha_\nu\in\R$, we say 
  that the family together with the coefficients is $\delta$-interacting if \cref{eq:def_delta_interaction} 
  holds and moreover
  \begin{equation*}
    \max_{1\le i\le \nu} \abs{\alpha_i-1} \le \delta \fullstop
  \end{equation*}
\end{definition}
\begin{remark}
  Our definition of $\delta$-interaction between bubbles is tightly linked to the $H^1$-interaction.
  Indeed, if $U_1=U[z_1,\lambda_1],U_2=U[z_2,\lambda_2]$ are two bubbles, thanks to 
  \cref{prop:interaction_approx} it holds (recall that $-\Delta U_1=U_1^p$)
  \begin{equation*}
    \int_{\R^n}\nabla U_1\cdot\nabla U_2 = \int_{\R^n}U_1^pU_2 
    \approx \min\left(\frac{\lambda_1}{\lambda_2}, \frac{\lambda_2}{\lambda_1}, 
    \frac{1}{\lambda_1\lambda_2\abs{z_1-z_2}^2}\right)^{\frac{n-2}2} \fullstop
  \end{equation*}
  In particular, if $U_1$ and $U_2$ belong to a $\delta$-interacting family then their $H^1$-scalar product is
  bounded by $\delta^{\frac{n-2}2}$.
\end{remark}

\subsection{Main Theorem}
We are ready to state and prove our main theorem in low dimension ($3\le n\le 5$).
We want to show that, in a neighborhood of a weakly-interacting family of Talenti bubbles, the quantity
$\norm{\lapl u + u\abs{u}^{p-1}}_{H^1}$ controls the $H^1$-distance of $u$ from the manifold of sums of Talenti
bubbles.

Let us briefly describe the structure of the proof.
First we consider the sum of Talenti bubbles $\sigma$ that minimizes the distance from $u$.
Then, setting $u=\sigma+\rho$, we test $\lapl u + u\abs{u}^{p-1}$ against $\rho$. Doing so we obtain
an estimate on $\norm{\nabla\rho}_{L^2}$ (namely \cref{eq:pos_init}).
From there, we estimate the right-hand side of \cref{eq:pos_init} with \cref{eq:fundamental_pos_ineq3} to reduce the statement to the validity of the two nontrivial inequalities 
\cref{eq:spectral_estimate,eq:interaction_estimate}.
The proofs of the two mentioned inequalities are postponed to the subsequent sections.

We note that first part of the strategy follows the approach used in \cite{cirfigmag2018} to deal with the
simpler case of a single bubble ($\nu=1$).

\begin{theorem}\label{thm:main_close}
  For any dimension $3\le n\le 5$ and $\nu\in\N$, there exist a small constant $\delta=\delta(n,\nu)>0$ and
  a large constant $C=C(n,\nu)>0$ such that the following statement holds.
  Let $u\in H^1(\R^n)$ be a function such that
  \begin{equation*}
    \norm*{\nabla u - \sum_{i=1}^{\nu}\nabla \tilde U_i}_{L^2} \le \delta \comma
  \end{equation*}
  where $(\tilde U_i)_{1\le i\le \nu}$ is a $\delta$-interacting family of Talenti bubbles.
  Then there exist $\nu$ Talenti bubbles $U_1,U_2,\dots, U_\nu$ such that
  \begin{equation*}
    \norm*{\nabla u - \sum_{i=1}^{\nu}\nabla U_i}_{L^2} \le C\norm{\lapl u + u\abs{u}^{p-1}}_{H^{-1}} 
    \fullstop
  \end{equation*}
  Furthermore, for any $i\not= j$, the interaction between the bubbles can be estimated as
  \begin{equation}\label{eq:interaction_estimate_statement}
    \int_{\R^n}U_i^pU_j \le C\norm{\lapl u + u\abs{u}^{p-1}}_{H^{-1}} \fullstop
  \end{equation}
\end{theorem}
\begin{proof}
In our approach, first we approximate $u$ not only with \emph{sums} of Talenti bubbles, but even
with \emph{linear combinations} of them. A posteriori we show that we can recover the result for 
sums. Adding the degree of freedom of choosing the coefficients of the linear combination gives us the 
fundamental information \cref{eq:orthogonality_U1}, but at the same time it compels us to prove that in the optimal choice
the coefficients are (approximately) $1$ (see \cref{prop:interaction_and_coef}).

Let $\sigma=\sum_{i=1}^\nu \alpha_i U[z_i,\lambda_i]$ be the linear combination of Talenti bubbles that is 
closest to $u$ in the $H^1$-norm, that is
\begin{equation*}
  \norm{\nabla u-\nabla\sigma}_{L^2} 
  = \min_{\substack{
  \tilde\alpha_1,\dots,\tilde\alpha_\nu\in\R
  \\
  \tilde z_1,\dots,\dots,\tilde z_\nu\in\R^n
  \\
  \tilde \lambda_1,\dots,\tilde \lambda_\nu
  }} \norm*{\nabla u-\nabla\left(
  \sum_{i=1}^\nu \tilde\alpha_i U[\tilde z_i, \tilde \lambda_i]
  \right)}_{L^2} \fullstop
\end{equation*}
Let $\rho\defeq u-\sigma$ be the difference between the original function and the best 
approximation. Moreover, let us denote $U_i\defeq U[z_i,\lambda_i]$.

From the fact that the $H^1$-distance of $u$ from $\sum_{i=1}^{\nu}\tilde U_i$ is less than $\delta$, it 
follows directly that $\norm{\nabla\rho}_{L^2}\le \delta$. Furthermore, since the bubbles $\tilde U_i$ are
$\delta$-interacting, the family $(\alpha_i,U_i)_{1\le i\le \nu}$ is $\delta'$-interacting for some $\delta'$ that
goes to zero as $\delta$ goes to $0$.

Summing up, we can say \emph{qualitatively} that $\sigma$ is a sum of weakly-interacting 
Talenti bubbles and that $\norm{\nabla\rho}_{L^2}$ is small.

Since $\sigma$ minimizes the $H^1$-distance from $u$, $\rho$ is $H^1$-orthogonal to the manifold
composed of linear combinations of $\nu$ Talenti bubbles. Hence, for any $1\le i\le \nu$, the
following $n+2$ orthogonality conditions hold:
\begin{align}
  &\int_{\R^n}\nabla\rho\cdot\nabla U_i = 0 \comma \label{eq:orthogonality_H1}\\
  &\int_{\R^n}\nabla\rho\cdot\nabla \partial_\lambda U_i = 0 \comma \label{eq:orthogonality_H2}\\
  &\int_{\R^n}\nabla\rho\cdot\nabla \partial_{z_j}U_i = 0 
  \quad\text{for any $1\le j\le n$.} \label{eq:orthogonality_H3}
\end{align}
Since the functions $U_i, \partial_\lambda U_i, \partial_{z_j}U_i$ are eigenfunctions for 
$\frac{-\lapl}{U_i^{p-1}}$, the mentioned orthogonality conditions are equivalent to
\begin{align}
  &\int_{\R^n}\rho \,U_i^p = 0 \comma \label{eq:orthogonality_U1}\\
  &\int_{\R^n}\rho \,\partial_\lambda U_i\, U_i^{p-1}= 0 \comma \label{eq:orthogonality_U2}\\
  &\int_{\R^n}\rho\,\partial_{z_j}U_i\,U_i^{p-1}= 0 \quad\text {for any $1\le j\le n$.} \label{eq:orthogonality_U3}
\end{align}
Our goal is to show that $\norm{\nabla\rho}_{L^2}$ is controlled by $\norm{\lapl u + u\abs{u}^{p-1}}_{H^{-1}}$.
To achieve this, let us start by testing $\lapl u + u\abs{u}^{p-1}$ against $\rho$: 
exploiting the orthogonality condition \cref{eq:orthogonality_H1} yields
\begin{equation}\label{eq:pos_init}
\begin{aligned}
  \int_{\R^n}\abs{\nabla\rho}^2 = \int_{\R^n} \nabla u\cdot\nabla\rho
  &=\int_{\R^n} u\abs{u}^{p-1}\rho  
  -\int \rho(\lapl u + u\abs{u}^{p-1})\\
&  \le \int_{\R^n} u\abs{u}^{p-1}\rho  
  + \norm{\nabla\rho}_{L^2}\norm{\lapl u + u\abs{u}^{p-1}}_{H^{-1}} \fullstop
  \end{aligned}
\end{equation}
To control the first term, we use the elementary estimates
\begin{align}
  \abs*{(a+b)\abs{a+b}^{p-1} - a\abs{a}^{p-1}}
  \le p\abs{a}^{p-1}\abs{b} + C_n\left(\abs{a}^{p-2}\abs{b}^2 + \abs{b}^p\right)\comma 
  \label{eq:tmp_main_pos1}
  \\
  \abs*{\left(\sum_{i=1}^{\nu}a_i\right)\abs*{\sum_{i=1}^{\nu}a_i}^{p-1}
  -\sum_{i=1}^\nu a_i\abs{a_i}^{p-1}} 
  \lesssim 
  \sum_{1\le i\not=j \le \nu} \abs{a_i}^{p-1}\abs{a_j} \label{eq:tmp_main_pos2}\comma
\end{align}
that hold for any $a,b\in\R$ and for any $a_1,\dots,a_\nu\in\R$.
Applying \cref{eq:tmp_main_pos1} with $a=\sigma$ and $b=\rho,$ and \cref{eq:tmp_main_pos2}
with $a_i=\alpha_i U_i$, we deduce
\begin{equation*}
  \abs*{u\abs{u}^{p-1}-\sum_{i=1}^\nu\alpha_i\abs{\alpha_i}^{p-1}U_i^p}
  \le 
  p\sigma^{p-1}\abs{\rho} + C_{n,\nu}\Biggl(
  \sigma^{p-2}\abs{\rho}^2 + \abs{\rho}^p + \sum_{1\le i\not= j\le \nu} U_i^{p-1}U_j\Biggr)
\end{equation*}
and therefore, recalling \cref{eq:orthogonality_U1}, we find
\begin{equation}\label{eq:fundamental_pos_ineq}\begin{aligned}
  \int_{\R^n}u\abs{u}^{p-1}\rho \le p\int_{\R^n}\sigma^{p-1}\rho^2 + 
  C_{n,\nu}\Biggl(\int_{\R^n}\sigma^{p-2}\abs{\rho}^3 + \int_{\R^n}\abs{\rho}^{2^*} 
  + \sum_{1\le i\not=j\le \nu} \int_{\R^n}\abs{\rho}U_i^{p-1}U_j\Biggr) \fullstop
\end{aligned}\end{equation}
Some of the terms in the right-hand side can be controlled easily. Applying  H\"older
and Sobolev inequalities we get
\begin{align*}
  &\int_{\R^n}\sigma^{p-2}\abs{\rho}^3 \le \norm{\sigma}_{L^{2^*}}^{p-2}\norm{\rho}_{L^{2^*}}^3
  \lesssim \norm{\nabla\rho}_{L^2}^3 \comma \\
  &\int_{\R^n}\abs{\rho}^{2^*} \lesssim \norm{\nabla\rho}_{L^2}^{2^*} \comma \\
  &\int_{\R^n}\abs{\rho}U_i^{p-1}U_j 
  \le \norm{\rho}_{L^{2^*}}\norm{U_i^{p-1}U_j}_{L^{(2^*)'}}
  \lesssim \norm{\nabla\rho}_{L^2}\norm{U_i^{p-1}U_j}_{L^{(2^*)'}} \fullstop
\end{align*}
Substituting these estimates into \cref{eq:fundamental_pos_ineq} gives us
\begin{equation}\label{eq:fundamental_pos_ineq2}\begin{aligned}
  \int_{\R^n}u\abs{u}^{p-1}\rho \le p\int_{\R^n}\sigma^{p-1}\rho^2 
  + 
  C_{n,\nu}\Biggl(\norm{\nabla\rho}_{L^2}^3 + \norm{\nabla\rho}_{L^2}^{2^*} + 
  \sum_{1\le i\not=j \le\nu}\norm{\nabla\rho}_{L^2}\norm{U_i^{p-1}U_j}_{L^{(2^*)'}}\Biggr) \fullstop
\end{aligned}\end{equation}
In order to proceed further let us notice that, thanks to \cref{prop:interaction_approx}, for any $i\not=j$ it 
holds\footnote{This is the only point of the whole proof where the condition on the dimension plays a crucial role. Indeed,
  when $n\ge 7$ only the weaker estimate
  \begin{equation*}
    \norm{U_i^{p-1}U_j}_{L^{(2^*)'}} \approx \left(\int_{\R^n}U_i^pU_j\right)^{p-1} 
  \gg \int_{\R^n}U_i^pU_j
  \end{equation*}
  holds, while for $n=6$ we have
   \begin{equation*}
    \norm{U_i^{p-1}U_j}_{L^{(2^*)'}} \approx \left(\int_{\R^n}U_i^pU_j\right)\left|\log\left(\int_{\R^n}U_i^pU_j\right)\right|^{2/3}
    \gg \int_{\R^n}U_i^pU_j\comma
  \end{equation*}
  and none of these estimates suffices to conclude the proof.}
\begin{equation}\label{eq:le_dimensioni_contano}
  \norm{U_i^{p-1}U_j}_{L^{(2^*)'}} = \left(\int_{\R^n}U_i^{(p-1)(2^*)'}U_j^{(2^*)'}\right)^{\frac{1}{(2^*)'}}
  \approx \left(\int_{\R^n}U_i^p U_j\right)^{\frac{(2^*)'}{(2^*)'}}
  = \int_{\R^n}U_i^p U_j \fullstop
\end{equation}
Hence \cref{eq:fundamental_pos_ineq2} becomes
\begin{equation}\label{eq:fundamental_pos_ineq3}
  \int_{\R^n}u\abs{u}^{p-1}\rho \le p\int_{\R^n}\sigma^{p-1}\rho^2
  + 
  C_{n,\nu}\Biggl(\norm{\nabla\rho}_{L^2}^3 + \norm{\nabla\rho}_{L^2}^{2^*} + 
  \sum_{1\le i\not=j \le\nu}\norm{\nabla\rho}_{L^2}\int_{\R^n}U_i^pU_j\Biggr) \fullstop
\end{equation}
It remains to estimate $\int_{\R^n} \sigma^{p-1}\rho^2$ and $\int_{\R^n} U_i^pU_j$. 
While the control on the first one is by now rather standard, some new ideas are needed to control the second term.
We state here the two inequalities that we need to conclude, and we postpone their proofs to \cref{subsec:missing,subsec:localization,subsec:spectral,subsec:interaction}:
\begin{itemize}
 \item Provided $\delta'$ is sufficiently small, it holds
  \begin{equation}\label{eq:spectral_estimate}
    \int_{\R^n} \sigma^{p-1}\rho^2 \le \frac{\tilde c(n,\nu)}{p}\int_{\R^n}\abs{\nabla\rho}^2
  \end{equation}
  for some constant $\tilde c(n,\nu)<1$.
 \item Given $\hat \eps>0$, if $\delta'$ is sufficiently small then 
  \begin{equation}\label{eq:interaction_estimate}
    \int_{\R^n}U_i^pU_j 
    \lesssim \hat \eps\norm{\nabla\rho}_{L^2} 
    + \norm{\lapl u + u\abs{u}^{p-1}}_{H^{-1}} 
    + \norm{\nabla\rho}_{L^2}^2
  \end{equation}
  for any $i \neq j$.
\end{itemize}
With these inequalities at our disposal, we can easily conclude the proof.
Indeed, choose $\hat \eps>0$ such that $\nu^2\hat \eps C_{n,\nu}\tilde C + \tilde c(n,\nu)<1$, where $C_{n,\nu}$ is the 
constant that appears in \cref{eq:fundamental_pos_ineq3} and $\tilde C$ is the constant hidden in the 
$\lesssim$-notation in the inequality \cref{eq:interaction_estimate}.  
Combining \cref{eq:spectral_estimate,eq:interaction_estimate} into \cref{eq:fundamental_pos_ineq3} yields
\begin{align*}
  \int_{\R^n}u\abs{u}^{p-1}\rho &\le \left(\tilde c(n,\nu)+\nu^2\hat \eps C_{n,\nu}\tilde C\right)\norm{\nabla\rho}_{L^2}^2 \\
  &+C'_{n,\nu}\left(\norm{\nabla\rho}_{L^2}^3 + \norm{\nabla\rho}_{L^2}^{2^*} + 
  \norm{\nabla\rho}_{L^2}\norm{\lapl u + u\abs{u}^{p-1}}_{H^{-1}}\right) \fullstop
\end{align*}
Hence, recalling \cref{eq:pos_init}, we deduce
\begin{equation*}
  \left(1-\tilde c(n,\nu)-\nu^2\hat \eps C_{n,\nu}\tilde C\right)\norm{\nabla\rho}_{L^2}^2 \lesssim 
  \norm{\nabla\rho}_{L^2}\norm{\lapl u + u\abs{u}^{p-1}}_{H^{-1}}
  +\norm{\nabla\rho}_{L^2}^3 + \norm{\nabla\rho}_{L^2}^{2^*} \fullstop
\end{equation*}
Since we can assume that $\norm{\nabla\rho}_{L^2} \ll 1$, it is easy to see that
this last inequality implies the desired estimate
\begin{equation}\label{eq:pos_final_deduction}
  \norm{\nabla\rho}_{L^2} \lesssim \norm{\lapl u + u\abs{u}^{p-1}}_{H^{-1}} \fullstop
\end{equation}
Hence, apart from proving the two inequalities \cref{eq:spectral_estimate,eq:interaction_estimate} (that for now we have taken
for granted), in order to finish the proof we have to check:
\begin{itemize}
  \item that the value of all the $\alpha_i$ can be replaced with $1$;
  \item that \cref{eq:interaction_estimate_statement} holds.
\end{itemize}
Note that, thanks to \cref{eq:pos_final_deduction}, both facts are direct byproducts either of
\cref{eq:interaction_estimate} or of the full statement of the proposition that proves 
\cref{eq:interaction_estimate} (that is \cref{prop:interaction_and_coef}). 
Indeed,  by the latter proposition and \cref{eq:pos_final_deduction} we know $\abs{\alpha_i-1} \lesssim \norm{\lapl u + u\abs{u}^{p-1}}_{H^{-1}}$,
so it suffices to  consider $\sigma'=\sum_{i=1}^{\nu}U_i$ to get that $\sigma'$ satisfies all the desired conditions.
\end{proof}

\subsection{Consequences of the main theorem}\label{subsec:corollaries}
As a direct consequence of \cref{thm:main_close} (and of well-known results in literature) we can show the
following corollary:

\begin{corollary}\label{cor:main_pos}
  For any dimension $3\le n\le 5$ and $\nu\in\N$, there exists a constant $C=C(n,\nu)$ such that the following
  statement holds.
  For any nonnegative function $u\in H^1(\R^n)$ such that
  \begin{equation*}
    \left(\nu-\frac12\right) S^n \le \int_{\R^n}\abs{\nabla u}^2 \le \left(\nu+\frac12\right) S^n \comma
  \end{equation*}
  there exist $\nu$ Talenti bubbles $U_1,U_2,\dots, U_\nu$ such that
  \begin{equation*}
    \norm*{\nabla u - \sum_{i=1}^{\nu}\nabla U_i}_{L^2} \le C\norm{\lapl u + u^p}_{H^{-1}} 
    \fullstop
  \end{equation*}
  Furthermore, for any $i\not= j$, the interaction between the bubbles can be estimated as
  \begin{equation*}
    \int_{\R^n}U_i^pU_j \le C\norm{\lapl u + u^p}_{H^{-1}} \fullstop
  \end{equation*}
\end{corollary}
\begin{proof}
Our strategy is to apply \cref{thm:main_close}.

Up to enlarging the constant $C$ in the statement, we can assume that
$\norm{\lapl u + u\abs{u}^{p-1}}_{H^{-1}}$ is smaller than a fixed $\eps>0$.
Applying \cite[Chapter III, Theorem 3.1 and Remarks 3.2]{struwe2008} (or directly
the original papers \cite{struwe1984,gidas1979,bahri1988,obata1962}), we know that for any $\delta>0$ 
we can find an $\eps>0$ such that if $\norm{\lapl u + u\abs{u}^{p-1}}_{H^{-1}} \le \eps$ then
\begin{equation*}
  \norm*{\nabla u-\sum_{i=1}^{\nu}\nabla U_i}_{L^2} \le \delta \comma
\end{equation*}
where $(U_i)_{1\le i\le \nu}$ is a $\delta$-interacting family.
This is exactly the hypothesis necessary to apply \cref{thm:main_close} and conclude.
\end{proof}
\begin{remark}
  Let us emphasize that \cref{cor:main_pos} would become false if we drop the assumption of nonnegativity of 
  the function $u$. Indeed, as shown in \cite{ding1986}, there exist 
  sign-changing solutions of $-\lapl u + u\abs{u}^{p-1} = 0$ with finite energy on $\R^n$ that are not Talenti 
  bubbles.
\end{remark}

\begin{remark}
\label{rmk:every n}
Our proof works in any dimension if $\nu=1$. Indeed, there are only two points in the proof where the assumption $n \leq 5$ is used:
\begin{enumerate}
\item For $n >6$ the exponent $p$ is less than 2, and therefore the inequality
$$
\int_{\R^n}\sigma^{p-2}|\rho|^3 \leq \|\sigma\|_{L^{2^*}}^{p-2}\|\rho\|_{L^{2^*}}^3
$$
is false. However, one can note that for $p<2$ the inequality \cref{eq:tmp_main_pos1} holds also without the term $\abs{a}^{p-2}\abs{b}^2$, so for $n>6$ the term $\int_{\R^n}\sigma^{p-2}|\rho|^3$ is not present.
\item
As observed in the footnote before \cref{eq:le_dimensioni_contano},
the assumption $n\le 5$ is crucial for estimating the interaction
integrals between bubbles. However, if $\nu=1$ then there are no interaction integrals and thus everything
works also in higher dimension. 
\end{enumerate}
\end{remark}

Thence, we can prove the result for nonnegative functions in arbitrary dimension when only one bubble is 
allowed.
The following statement, with some minor differences, is the main result of \cite{cirfigmag2018}.
We state it here in this slightly different form since it will be convenient later in \cref{sec:fdi}.
\begin{corollary}\label{cor:single_bubble}
  For any dimension $n\ge 3$, there exists a constant $C=C(n)$ such that the following statement holds.
  For any nonnegative function $u\in H^1(\R^n)$ such that
  \begin{equation*}
    \frac12 S^n \le \int_{\R^n}\abs{\nabla u}^2 \le \frac32 S^n \comma
  \end{equation*}
  there exists a Talenti bubble $U$ such that
  \begin{equation*}
    \norm*{\nabla u - \nabla U}_{L^2} \le C\norm{\lapl u + u^p}_{H^{-1}} 
    \fullstop
  \end{equation*}
\end{corollary}
\begin{proof}
  The proof of the statement is identical to the proof of \cref{cor:main_pos}, with the 
  only difference that (thanks to \cref{rmk:every n}) we can apply \cref{thm:main_close} for any dimension $n\ge 3$ since only one bubble is present.
\end{proof}

\subsection{The two missing estimates}
\label{subsec:missing}
It remains to prove \cref{eq:spectral_estimate} and \cref{eq:interaction_estimate}.
Before proving them, let us shed some light on the reasons why these two inequalities should hold.

The first of the two inequalities is a strengthened Poincaré whose validity follows from the spectral
properties of $\rho$. Indeed, in the simple setting with a single bubble $\sigma=U[0,1]$, 
\cref{eq:spectral_estimate} is equivalent to
\begin{equation*}
  \int_{\R^n} U[0,1]^{p-1} \rho^2 \le \frac{\tilde c}p \int_{\R^n}\abs{\nabla\rho}^2 \fullstop
\end{equation*}
This latter inequality follows from the orthogonality conditions \cref{eq:orthogonality_U1,eq:orthogonality_U2,eq:orthogonality_U3}
since $U,\partial_{\lambda}U,\partial_{z_i}U$ are the eigenfunctions of $\frac{-\lapl}{U^{p-1}}$ with
eigenvalue greater or equal to $\frac1p$. For a justification of the last statement, see for instance 
\cite[Appendix]{bianchi1991}.
To handle the fact that in our setting $\sigma$ can be a linear combination of multiple bubbles, we make use of
a localization argument via partitions of unity that allows us to treat each bubble independently.
Although this argument is rather standard and \cref{eq:spectral_estimate} is already known (see for instance 
\cite[Proposition 3.1]{bahri1989}), we prefer to write the proof both for the convenience of the reader and also because some of the localization arguments will be useful later.

The estimate \cref{eq:interaction_estimate} is of course empty (and thus trivial) if there is a single bubble,
hence also the difficulty of this estimate depends heavily on the presence of multiple bubbles. Exploiting
the  localization argument used to prove \cref{eq:spectral_estimate}, we manage to prove \cref{eq:interaction_estimate} by testing 
$-\lapl u + u\abs{u}^{p-1}$ against suitably localized versions of $U$ and $\partial_\lambda U$. 

Showing \cref{eq:interaction_estimate} is the less intuitive and most involved part of the whole proof.

\subsection{Localization of a family of bubbles}
\label{subsec:localization}
Given a family of Talenti bubbles $\sigma=\sum_{i=1}^\nu \alpha_i U_i$, we want to build some bump functions 
$\Phi_1,\dots,\Phi_\nu$ in such a way that, in some appropriate sense, $\sigma\Phi_i \cong \alpha_i U_i$.
The existence of these bump functions is tightly linked to the fact that the family is $\delta$-interacting
for a small $\delta$. Indeed if, for example, $\sigma = U_1 + U_2$ and $U_1=U_2$ it is clearly impossible
to find any region where, in any meaningful sense, $\sigma\cong U_1$.

A similar localization argument is present in \cite[Proposition 3.1 and Lemma 3.2]{bahri1989}, where the author proves \cref{prop:spectral_ineq}. However, as mentioned before, 
since in any case we need some further properties of the localization in order to prove \cref{prop:interaction_and_coef},
we decided to include full proofs both of the localization argument and of 
\cref{prop:spectral_ineq}.

\begin{lemma}\label{lem:nice_cutoff}
Let $n\geq 1$. Given a point $\bar x\in\R^n$ and two radii $0<r<R$, there exists a 
  Lipschitz bump function $\varphi=\varphi_{\bar x, r, R}:\R^n\to\cc01$ such that $\varphi\equiv 1$ in 
  $B(\bar x, r)$, $\varphi\equiv 0$ in $B(\bar x, R)^\complement$, and
  \begin{equation*}
    \int_{\R^n} \abs{\nabla\varphi}^n \lesssim \log\left(\frac{R}{r}\right)^{1-n} \fullstop
  \end{equation*}
\end{lemma}
\begin{proof}
  Without loss of generality we can assume $\bar x = 0$. We define $\varphi$ as
  \begin{equation*}
    \varphi(x) \defeq
    \begin{cases}
      1 &\quad\quad\text{if $\abs{x}\le r$}, \\
      \frac{\log(R)-\log(\abs{x})}{\log(R)-\log(r)} &\quad\quad\text{if $r<\abs{x}<R$},\\
      0 &\quad\quad\text{if $R<\abs{x}$}.
    \end{cases}
  \end{equation*}
  By definition $\varphi\equiv 1$ in $B(0, r)$ and $\varphi\equiv 0$ in $B(0,R)^\complement$. 
  The norm of the gradient of $\varphi$ is $0$ outside $B(0,R)\setminus B(0,r)$, whereas inside that annulus it
  satisfies
  \begin{equation*}
    \abs{\nabla\varphi}(x) = \log\left(\frac Rr\right)^{-1}\frac 1{\abs{x}} \fullstop
  \end{equation*}
  Thus, it holds
  \begin{equation*}
    \int_{\R^n} \abs{\nabla\varphi}^n 
    = \log\left(\frac Rr\right)^{-n}\int_{B(0, R)\setminus B(0, r)} \frac{1}{\abs{x}^n}\de x
    \approx \log\left(\frac Rr\right)^{1-n}
  \end{equation*}
  as desired.
\end{proof}

\begin{lemma}\label{lem:localization}
  For any $n\ge 3$, $\nu\in\N$, and $\eps>0$, there exists $\delta=\delta(n,\nu,\eps)>0$ such that if 
  $U_1=U[z_1,\lambda_1],\dots,U_\nu=U[z_\nu,\lambda_\nu]$ is a $\delta$-interacting family of $\nu$ Talenti 
  bubbles, then for any $1\le i\le \nu$ there exists a Lipschitz bump function
  $\Phi_i:\R^n\to\cc01$ such that the following hold:
  \begin{enumerate}[label=\textit{(\arabic*)}]
   \item \label{it:localization1}Almost all mass of $U_i^{2^*}$ is in the region $\{\Phi_i=1\}$, that is 
   \begin{equation*}\int_{\{\Phi_i=1\}} U_i^{2^*} \ge (1-\eps)S^n \fullstop\end{equation*}
   \item \label{it:localization2}In the region $\{\Phi_i>0\}$ it holds $\eps U_i > U_j$ for any $j\not=i$.
   \item \label{it:localization3}The $L^n$-norm of the gradient is small, that is
   \begin{equation*} \norm{\nabla\Phi_i}_{L^n} \le \eps \fullstop\end{equation*}
   \item \label{it:localization4}For any $j\not=i$ such that $\lambda_j\le \lambda_i$, it holds
   \begin{equation*}
      \frac{\sup_{\{\Phi_i>0\}}U_j}{\inf_{\{\Phi_i>0\}}U_j} \le 1 + \eps \fullstop
   \end{equation*}
  \end{enumerate}
\end{lemma}
\begin{proof}
  \newcommand{\uffa}{\epsilon}
  Without loss of generality we can show the statement only for one of the indices, say $i=\nu$, and we can also assume that
  $U_\nu=U[0, 1]$. This second assumption is justified by the observations contained in 
  \cref{subsec:symmetries}, since the $L^n$-norm of the gradient of a function is invariant under scaling.
  For notational simplicity we denote $U\defeq U_\nu$.
  
  We fix a small number $\uffa>0$ (that will be fixed at the end of the proof, depending on the parameter $\eps$ appearing in the statement) and a large parameter $R>1$ ($R$ will be fixed later, depending on $\uffa$).
  
  If $\delta$ is sufficiently small (depending on $\uffa$ and $R$), it follows that $\uffa U>U_j$ in $B(0,R)$
  for any $1\le j < \nu$ such that $\lambda_j < 1$ or $\abs{z_j}>2R$. In other words we are saying that, if the bubbles are 
  sufficiently weakly-interacting, in an arbitrarily large ball $U$ is much larger than every other bubble which is either less 
  concentrated  or sufficiently far. Moreover, if $\delta$ is sufficiently small, since 
  $U_j(x) \sim (\lambda_j^{-1}+\abs{x-z_j})^{2-n}$ it also holds
  \begin{equation*}
    \frac{\sup_{B(0,R)}U_j}{\inf_{B(0,R)}U_j} \le 1 + \uffa
  \end{equation*}
  for any $1\le j < \nu$ such that $\lambda_j\le 1$.
 Hence, it remains to control the bubbles that are not far and that are more concentrated than $U$.
  
  Let $I\subset \{1,\ldots,\nu-1\}$ be the set of indices $j$ such that $\lambda_j>1$ and $\abs{z_j}<2R$.
  For any $j\in I$, let $R_j\in\oo{0}{\infty}$ be the only positive real number such that
  \begin{equation*}
      \uffa\left(\frac1{1+R^2}\right)^{\frac{n-2}2} 
      = \left(\frac{\lambda_j}{1+\lambda_j^2\abs{R_j}^2}\right)^{\frac{n-2}2} \fullstop
  \end{equation*}
 Note that if $\delta$ is sufficiently small then $R_j\le \uffa^2$ for any $j\in I$.
  Thus, as a consequence of the definition of $R_j$, for any $j\in I$ it holds $\uffa U \ge U_j$ in 
  $B(0, R)\setminus B(0, R_j)$.
  
  We are now in position to define the function $\Phi=\Phi_\nu$. 
  With the notation $\varphi_{\bar x, r, R}$ introduced in \cref{lem:nice_cutoff}, we define $\Phi$ as
  \begin{equation*}
    \Phi \defeq \varphi_{0,\uffa R, R} \prod_{j\in I} (1-\varphi_{z_j, R_j, \uffa^{-1}R_j}) \fullstop
  \end{equation*}
  Let us check that if $R$ is chosen sufficiently large, then all requirements are satisfied.
  
Since $R_j\le \uffa^2$, it holds
  \begin{align*}
    \int_{\{\Phi < 1\}}U^{2^*} 
    \le \int_{B(0,R)^{\complement}} U^{2^*} + \sum_{j\in I}\int_{B(z_j, \uffa^{-1}R_j)} U^{2^*}
    \le \int_{B(0,R)^{\complement}} U^{2^*} + \sum_{j\in I}C_n\left(\uffa^{-1}R_j\right)^{n} 
    \le \int_{B(0,R)^{\complement}} U^{2^*} + C_n\nu\uffa^n \fullstop
  \end{align*}
  Hence, if $\uffa$ is sufficiently small, 
  choosing $R=R(\uffa)$ large enough we obtain
  \begin{equation*}
    \int_{\{\Phi<1\}}U^{2^*} \le \uffa S^n
  \end{equation*}
  and therefore \cref{it:localization1} holds.
  
 Noticing that $\{\Phi>0\}$ is contained inside
  \begin{equation*}
    B(0,R)\setminus \bigcup_{j\in I} B(z_j, R_j) \comma
  \end{equation*}
since inside such region  $\uffa U\ge U_j$ for 
  any $1\le j<\nu$ (by the observations above), also \cref{it:localization2} holds.
  
  Similarly, since $\{\Phi>0\}$ is contained into $B(0,R)$, also property \cref{it:localization4} is satisfied.
  
 Finally, since $\varphi_{\bar x, r_1, r_2}(x)\in \cc01$ for any choice of the 
  parameters $\bar x, x, r_1, r_2$, 
  \begin{equation*}
    \abs{\nabla \Phi(x)} \le \abs{\nabla\varphi_{0,\uffa R, R}(x)} 
    + \sum_{j\in I}\abs{\nabla\varphi_{0,R_j, \uffa^{-1}R_j}(x)} \quad \text{for any $x\in \R^n$}\fullstop
  \end{equation*}
 Thus, taking into account \cref{lem:nice_cutoff}, we deduce
  \begin{equation*}
    \norm{\nabla \Phi(x)}_{L^n} 
    \le \norm{\nabla\varphi_{0,\uffa R, R}(x)}_{L^n} 
    + \sum_{j\in I}
    \norm{\nabla\varphi_{0,R_j, \uffa^{-1}R_j}(x)}_{L^n}
    \le C(n)\nu \log(\uffa^{-1})^{\frac1n-1} \comma
  \end{equation*}
for some dimensional constant $C(n)$.
Hence, given $\eps>0$, it suffices to choose $\uffa$ small enough to ensure that 
$C(n)\nu\log(\uffa^{-1})^{\frac1n-1} \leq \eps$ and \cref{it:localization3}. In this way, since $\eps\gg\uffa$, also all the
  other properties hold with $\uffa$ replaced by $\eps$.
\end{proof}

\subsection{Spectral Inequality}
\label{subsec:spectral}
Using the localization devised in \cref{lem:localization}, the proof of \cref{prop:spectral_ineq} follows
a very natural path: we localize, apply the spectral inequality for a single bubble, and then sum all the
terms to obtain the full estimate. 

\begin{proposition}\label{prop:spectral_ineq}
Let $n\ge 3$ and $\nu\in\N$. There exists a positive constant $\delta=\delta(n,\nu)>0$ such that if 
  $\sigma=\sum_{i=1}^{\nu}\alpha_i U[z_i,\lambda_i]$ is a linear combination of $\delta$-interacting Talenti 
  bubbles and $\rho\in H^1(\R^n)$ satisfies \cref{eq:orthogonality_U1,eq:orthogonality_U2,eq:orthogonality_U3} 
  with $U_i=U[z_i,\lambda_i]$, then
  \begin{equation*}
    \int_{\R^n}\sigma^{p-1}\rho^2 \le \frac{\tilde c}{p} \int_{\R^n}\abs{\nabla\rho}^2
  \end{equation*}
  where $\tilde c=\tilde c(n,\nu)$ is a constant strictly less than $1$.
\end{proposition}
\begin{proof}
  In this proof we will denote with $\smallo(1)$ any quantity that goes to zero when $\delta$ goes to zero.
  Let $\Phi_1,\dots,\Phi_\nu$ be the localization functions built in \cref{lem:localization} for a certain 
  $\eps$ that depends on $\delta$. It is clear that we can choose $\eps=\smallo(1)$.
  
  Thanks to \cref{lem:localization}-\cref{it:localization2}, it holds
  \begin{equation}\label{eq:spectral_tmp1}
    \int_{\R^n}\sigma^{p-1}\rho^2 \le (1+\smallo(1))\sum_{i=1}^{\nu} \int_{\R^n} \Phi_i^2\rho^2 U_i^{p-1}
    + \int_{\{\sum \Phi_i < 1\}} \sigma^{p-1}\rho^2 \fullstop
  \end{equation}
Then, by \cref{lem:localization}-\cref{it:localization1}, using H\"older and Sobolev inequalities we
  find
  \begin{equation}\label{eq:spectral_tmp2}
    \int_{\{\sum \Phi_i < 1\}} \sigma^{p-1}\rho^2
    \le \biggl(\int_{\{\sum \Phi_i < 1\}} \sigma^{2^*}\biggr)^{\frac{p-1}{2^*}}\norm{\rho}_{L^{2^*}}^2
    \le \smallo(1)\norm{\nabla\rho}_{L^2}^2 \fullstop
  \end{equation}
We now claim that
  \begin{equation}\label{eq:spectral_tmp3}
    \int_{\R^n}(\rho\Phi_i)^2 U_i^{p-1} 
    \le \frac1{\Lambda}\int_{\R^n}\abs{\nabla (\rho\Phi_i)}^2 + \smallo(1)\norm{\nabla\rho}_{L^2}^2\comma
  \end{equation}
  where $\frac1\Lambda$ is the largest eigenvalue of $\frac{-\lapl}{U_i^{p-1}}$ that is strictly smaller 
  than $\frac1p$. Let us remark that the value of $\Lambda$ does not depend on $i$.
  
  This inequality is crucial and is the only one that exploits the orthogonality conditions we are 
  assuming on $\rho$. Its proof relies on the fact that $\rho\Phi_i$ \emph{almost} satisfies the orthogonality 
  conditions, and hence the spectral properties of the operator $\frac{-\lapl}{U_i^{p-1}}$ will give us
  \cref{eq:spectral_tmp3}.
  
  Let $\psi:\R^n\to\R$ be, up to scaling, one of the functions $U_i,\partial_\lambda U_i,\partial_{z_j}U_i$, with the scaling chosen so that
$\int_{\R^n}\psi^2 U_i^{p-1}=1$. 
  Hence, thanks to the orthogonality conditions \cref{eq:orthogonality_U1,eq:orthogonality_U2,eq:orthogonality_U3}, it holds
  \begin{align*}
    \scalprod{\rho\Phi_i}{\psi}_{L^2_{U_i^{p-1}}} &=
    \abs*{\int_{\R^n}(\rho\Phi_i)\psi U_i^{p-1}} = \abs*{\int_{\R^n}\rho\psi U_i^{p-1} (1-\Phi_i)}
    \le \abs[\bigg]{\int_{\{\Phi_i<1\}} \rho\psi U_i^{p-1}} \\
    &\le \norm{\rho}_{L^{2^*}}\left(\int_{\R^n}\psi^2 U_i^{p-1}\right)^{\frac12}
    \biggl(\int_{\{\Phi_i<1\}}U_i^{2^*}\biggr)^{\frac1n}
    \le \smallo(1)\norm{\nabla\rho}_{L^2}\fullstop
  \end{align*}
  where in the last inequality we applied \cref{lem:localization}-\cref{it:localization1}.
  
  This proves that $\rho\Phi_i$ is almost orthogonal to $\psi$.
 Hence, since the functions $U_i,\partial_\lambda U_i,\partial_{z_j}U_i$ form an orthogonal basis for the space of eigenfunctions
  of $\frac{-\lapl}{U_i^{p-1}}$ with eigenvalue greater or equal than $\frac1p$ (see 
  \cite[Appendix]{bianchi1991}), it holds
  \begin{equation*}
    \int_{\R^n}(\rho\Phi_i)^2 U_i^{p-1} \le \frac{1}{\Lambda}\int_{\R^n}\abs{\nabla(\rho\Phi_i)}^2
    +\smallo(1)\int_{\R^n}\abs{\nabla\rho}^2
  \end{equation*}
  that is exactly \cref{eq:spectral_tmp3}.
  
  We now estimate the right-hand side of \cref{eq:spectral_tmp3}. Note that 
  \begin{equation}\label{eq:spectral_tmp4}
    \int_{\R^n}\abs{\nabla (\rho\Phi_i)}^2 = 
    \int_{\R^n}\abs{\nabla \rho}^2\Phi_i^2
    +\int_{\R^n}\rho^2\abs{\nabla \Phi_i}^2
    +2\int_{\R^n}\rho\Phi_i\nabla\rho\cdot\nabla\Phi_i \comma
  \end{equation}
and that the last two terms above  can be bounded as follows:
for the first one, using H\"older and Sobolev inequalities, we have the estimate
  \begin{equation*}	
    \int_{\R^n}\rho^2\abs{\nabla \Phi_i}^2
    \le 
    \norm{\rho}_{L^{2^*}}^2\norm{\nabla\Phi_i}_{L^n}^2 \le \smallo(1)\norm{\nabla\rho}_{L^2}^2
  \end{equation*}
  while for the second one, since $\frac1{2^*}+\frac1{\infty}+\frac12+\frac1n=1$, we find
  \begin{equation*}
    \int_{\R^n}\rho\Phi_i\nabla\rho\cdot\nabla\Phi_i \le 
    \norm{\rho}_{L^{2^*}}\norm{\Phi_i}_{L^{\infty}}\norm{\nabla\rho}_{L^2}\norm{\nabla\Phi_i}_{L^n}
    \le \smallo(1)\norm{\nabla\rho}_{L^2}^2 \fullstop
  \end{equation*}
  Hence \cref{eq:spectral_tmp4} becomes
  \begin{equation}\label{eq:spectral_tmp5}
    \int_{\R^n}\abs{\nabla (\rho\Phi_i)}^2 
    \le \int_{\R^n}\abs{\nabla \rho}^2\Phi_i^2 + \smallo(1)\norm{\nabla\rho}_{L^2}^2 \fullstop
  \end{equation}
  Note that, as a consequence of \cref{lem:localization}-\cref{it:localization2}, we know that the various bump 
  functions $\Phi_i$ have disjoint supports, therefore
  \begin{equation}\label{eq:spectral_tmp6}
    \sum_{i=1}^{\nu}\int_{\R^n}\abs{\nabla\rho}^2\Phi_i^2 \le \int_{\R^n}\abs{\nabla\rho}^2 \fullstop
  \end{equation}
Thus, combining \cref{eq:spectral_tmp1,eq:spectral_tmp2,eq:spectral_tmp3,eq:spectral_tmp5,eq:spectral_tmp6} we 
  achieve
  \begin{align*}
    \int_{\R^n}\sigma^{p-1}\rho^2 
    &\le (1+\smallo(1))\sum_{i=1}^{\nu} \int_{\R^n} \Phi_i^2\rho^2 U_i^{p-1}
    +\smallo(1)\norm{\nabla\rho}_{L^2}^2 \\
    &\le \left(\frac1\Lambda+\smallo(1)\right)\sum_{i=1}^{\nu} \int_{\R^n} \abs{\nabla (\rho\Phi_i)}^2
    +\smallo(1)\norm{\nabla\rho}_{L^2}^2
    \le \left(\frac1\Lambda+\smallo(1)\right)\int_{\R^n}\abs{\nabla\rho}^2
  \end{align*}
  that implies the statement because $\Lambda>p$.
\end{proof}

\subsection{Interaction integral estimate}
\label{subsec:interaction}
\begin{proposition}\label{prop:interaction_and_coef}
Let $n\ge 3$ and $\nu\in\N$. For any $\hat \eps>0$ there exists $\delta=\delta(n,\nu,\hat \eps)>0$ 
  such that the following statement holds.
  Let $u=\sum_{i=1}^\nu\alpha_i U_i + \rho$, where the family $(\alpha_i,U_i)_{1\le i\le \nu}$ 
  is $\delta$-interacting, and $\rho$ satisfies both the orthogonality conditions 
  \cref{eq:orthogonality_U1,eq:orthogonality_U2,eq:orthogonality_U3} and the bound $\norm{\nabla\rho}_{L^2}\le 1$.
  Then, for any $1\le i\le \nu$, it holds
  \begin{equation}\label{eq:interaction_and_coef:coef}
    \abs{\alpha_i-1} \lesssim \hat \eps\norm{\nabla\rho}_{L^2} 
    + \norm{\lapl u + u\abs{u}^{p-1}}_{H^{-1}} + \norm{\nabla\rho}_{L^2}^{\min(2,p)}\comma
  \end{equation}
  and for any pair of indices $i\not=j$ it holds
  \begin{equation}\label{eq:interaction_and_coef:interaction}
    \int_{\R^n} U_i^pU_j \lesssim \hat \eps\norm{\nabla\rho}_{L^2} 
    + \norm{\lapl u + u\abs{u}^{p-1}}_{H^{-1}} + \norm{\nabla\rho}_{L^2}^{\min(2,p)} \fullstop
  \end{equation}
\end{proposition}
\begin{proof}
  In order to handle the cases $n\le 6$ and $n>6$ at the same time, in this proof we will highlight as 
  $\lowdim{E}$ the terms $E$ that appear in our estimates only when $n\le 6$ (that is when 
  $p\ge 2$).

  We consider $\delta$ as a parameter and we denote with $\smallo(1)$ any expression that goes to zero 
  when the parameter $\delta$ goes to zero. Similarly $\smallo(E)$ denotes any expression that, when divided
  by $E$, is $\smallo(1)$.
  
  Let $\lambda_1,\dots,\lambda_\nu > 0$ and $z_1,\dots z_\nu\in\R^n$ be the parameters such that 
  $U_i=U[z_i,\lambda_i]$ for any $1\le i\le\nu$. Without loss of generality we can assume $\lambda_i$ to be
  decreasing (i.e. $U_1$ is the most concentrated bubble).
  We prove the statement by induction on the index $i=1,\ldots,\nu$ (starting from the most concentrated bubble).
  
  Let us fix $1\le i\le\nu$ and let us assume to know the result for all smaller values of the index.
  For notational simplicity we denote $U=U_i$, $\alpha=\alpha_i$, and $V=\sum_{j\not=i}\alpha_j U_j$.
  Let $\Phi=\Phi_i$ be the bump function built in \cref{lem:localization} for a fixed $\eps>0$ that depends
  on $\delta$ (and is $\smallo(1)$ by definition). 
  
  Without loss of generality we can assume $U_i=U=U[0,1]$. Indeed, all the quantities involved ($\alpha-1$, 
  $\int_{\R^n} U_i^pU_j$, $\norm{\nabla\rho}_{L^2}$, $\norm{\lapl u + u\abs{u}^{p-1}}_{H^{-1}}$) are invariant
  under the action of the symmetries described in \cref{subsec:symmetries}.
  
We begin from the identity
  \begin{equation}\begin{aligned}\label{eq:color_identity}
    ({\alpha}{-\alpha^p})U^p {- p\alpha^{p-1}U^{p-1}V} 
    = &\,{\lapl\rho + (-\lapl u - u\abs{u}^{p-1})} {-\sum \alpha_i U_i^p} {+ p(\alpha U)^{p-1}\rho}\\
      &+ \left[{(\sigma +\rho)\abs{\sigma+\rho}^{p-1}} { - \sigma^p} {- p\sigma^{p-1}\rho}\right] \\
      &+ \left[{p\sigma^{p-1}\rho} { - p(\alpha U)^{p-1}\rho}\right] \\
      &+ \left[{(\alpha U+V)^p} {-(\alpha U)^p} {- p(\alpha U)^{p-1}V}\right]
    \fullstop
  \end{aligned}\end{equation}
  Exploiting \cref{lem:localization}-\cref{it:localization2} we can show that, in the region 
  $\{\Phi>0\}$, it holds
  \begin{align*}
    &\sum_{j\not=i} \alpha_j U_j^p = \smallo\left(U^{p-1}V\right) \comma \\
    &\abs*{(\sigma+\rho)\abs{\sigma +\rho}^{p-1} - \sigma^p - p\sigma^{p-1}\rho} 
    \lesssim \abs{\rho}^p + \lowdim{U^{p-2}\abs{\rho}^2} \comma \\
    &\abs*{p\sigma^{p-1}\rho-p(\alpha U)^{p-1}\rho} = \smallo(U^{p-1}\abs{\rho}) \comma \\
    &\abs*{(\alpha U+V)^p -(\alpha U)^p - p(\alpha U)^{p-1}V} = \smallo\left(U^{p-1}V\right) \fullstop
  \end{align*}
  Thus, applying these  estimates in \cref{eq:color_identity}, 
  we deduce that inside the region $\{\Phi>0\}$ one has
  \begin{equation}\label{eq:main_before_testing}
   \begin{split}
    \abs*{(\alpha-\alpha^p)U^p - (p\alpha^{p-1}+\smallo(1))U^{p-1}V 
    - \lapl\rho - (-\lapl u-u\abs{u}^{p-1}) - p(\alpha U)^{p-1}\rho} \\ 
    \lesssim \abs{\rho}^p + \lowdim{U^{p-2}\abs{\rho}^2} + \smallo(U^{p-1}\abs{\rho}) \fullstop
   \end{split}
  \end{equation}
  
  In the remaining part of this proof, all integrals are computed on the whole $\R^n$ and therefore, for notational convenience, 
  we do not write explicitly the domain of integration.
  
  Let $\xi$ be either $U$ or $\partial_\lambda U$. What follows holds for both choices.
  
First of all, thanks to the orthogonality conditions \cref{eq:orthogonality_H1,eq:orthogonality_H2,eq:orthogonality_U1,eq:orthogonality_U2}, recalling \cref{eq:u_solves_yamabe}, we know that
  \begin{equation}\label{eq:ortho}
    \int U^{p-1}\xi\rho = \int \nabla\xi\cdot\nabla\rho = 0\fullstop
  \end{equation}

  Let us test \cref{eq:main_before_testing} against $\xi\Phi$. We get
  \begin{equation}\label{eq:generic_testing}\begin{aligned}
    \phantom{x}&\abs*{\int \left[(\alpha-\alpha^p)U^p-(p\alpha^{p-1}+\smallo(1))U^{p-1}V\right]\xi\Phi} \\
    &\quad\quad\quad\lesssim \abs*{\int \nabla\rho\cdot\nabla(\xi\Phi)} 
	  + \abs*{\int (-\lapl u - u\abs{u}^{p-1})\xi\Phi} 
	  + \abs*{\int U^{p-1}\xi\rho\Phi} \\
	  &\quad\quad\quad + \int \abs{\rho}^p\abs{\xi}\Phi
	  + \lowdim{\int U^{p-2}\abs{\xi}\abs{\rho}^2\Phi}
	  + \smallo\left(\int U^{p-1}\abs{\xi}\abs{\rho}\Phi\right) \fullstop
  \end{aligned}\end{equation}
We now exploit \cref{eq:ortho} to bound all the terms appearing in the right-hand side of \cref{eq:generic_testing}:
  \begin{equation}\label{eq:interaction_tmp1}\begin{aligned}
    \phantom{x}&\abs*{\int \nabla\rho\cdot\nabla(\xi\Phi)}
	= \abs*{\int \nabla\rho\cdot\nabla(\xi(\Phi-1))}
        \le \norm{\nabla\rho}_{L^2}\norm{\nabla(\xi(\Phi-1))}_{L^2} \comma \\
    &\abs*{\int (-\lapl u - u\abs{u}^{p-1})\xi\Phi} 
	\le \norm{\lapl u + u\abs{u}^{p-1}}_{H^{-1}}\norm{\nabla(\xi\Phi)}_{L^2}\comma\\ 
    &\abs*{\int U^{p-1}\xi\rho\Phi} = \abs*{\int U^{p-1}\xi\rho(\Phi-1)} 
      \lesssim 
      \norm{\nabla\rho}_{L^2}
      \biggl(\int_{\{\Phi<1\}}\bigl(U^{p-1}\abs{\xi}\bigr)^{\frac{2^*}p}\biggr)^{\frac{p}{2^*}} \comma\\
    &\int \abs{\rho}^p\abs{\xi}\Phi \le \int\abs{\rho}^p \abs{\xi} 
    \lesssim \norm{\nabla\rho}^p_{L^2}\norm{\xi}_{L^{2^*}} \comma \\
    &\lowdim{\int U^{p-2}\abs{\xi}\abs{\rho}^2\Phi \le \int U^{p-2}\abs{\xi}\abs{\rho}^2
    \lesssim \norm{\nabla\rho}^2_{L^2}\norm{U^{p-2}\xi}_{L^\frac{2^*}{p-1}}} \comma \\
    &\int U^{p-1}\abs{\xi}\abs{\rho}\Phi \le \int U^{p-1}\abs{\xi}\abs{\rho}
      \lesssim \norm{\nabla\rho}_{L^2}\norm{U^{p-1}\xi}_{L^\frac{2^*}{p}}\fullstop
  \end{aligned}\end{equation}
  Moreover, since $\xi$ is equal either to $U$ or to $\partial_\lambda U$, we have that $\abs{\xi}\lesssim U$ pointwise (see \cref{eq:U_derlambda}).
  Hence, recalling \cref{lem:localization}-\cref{it:localization1} and 
  \cref{lem:localization}-\cref{it:localization3} we obtain
  \begin{equation}\label{eq:interaction_tmp2}\begin{aligned}
    &\norm{\nabla(\xi(\Phi-1))}_{L^2} = \smallo(1) \comma
    &&\norm{\nabla(\xi\Phi)}_{L^2} \lesssim 1 \comma
    &&\int_{\{\Phi<1\}}\bigl(U^{p-1}\abs{\xi}\bigr)^{\frac{2^*}p} = \smallo(1) \comma\\
    &\norm{\xi}_{L^{2^*}} \lesssim 1 \comma
    &&\norm{U^{p-2}\xi}_{L^\frac{2^*}{p-1}} \lesssim 1 \comma
    &&\norm{U^{p-1}\xi}_{L^\frac{2^*}{p}} \lesssim 1 \fullstop
  \end{aligned}\end{equation}
  Using the set of inequalities \cref{eq:interaction_tmp1,eq:interaction_tmp2}, it follows by
  \cref{eq:generic_testing} that
  \begin{equation}\label{eq:testing_with_orthogonality}\begin{aligned}
    \phantom{x}&\abs*{\int \left[(\alpha-\alpha^p)U^p-(p\alpha^{p-1}+\smallo(1))U^{p-1}V\right]\xi\Phi} \\
    &\quad\quad\quad\lesssim 
    \smallo(1)\norm{\nabla\rho}_{L^2} 
    + \norm{\lapl u + u\abs{u}^{p-1}}_{H^{-1}} 
    + \norm{\nabla\rho}_{L^2}^{\min(2,p)}
    \fullstop
  \end{aligned}\end{equation}
  Let us now split $V=V_1+V_2$ where $V_1\defeq\sum_{j<i}\alpha_j U_j$ and $V_2\defeq\sum_{j>i}\alpha_j U_j$.
  Since by induction we can assume that the statement of the proposition holds for all $j<i$, and recalling that
  $\int U_i^pU_j = \int \nabla U_i\cdot\nabla U_j = \int U_j^pU_i$ 
  and $\abs{\xi}\lesssim U$, we see
  \begin{equation}\label{eq:inter_estimate_tmp10}
    \int U^{p-1}V_1\abs{\xi}\Phi \lesssim \int U^{p}V_1 \lesssim \smallo(1)\norm{\nabla\rho}_{L^2} 
    + \norm{\lapl u + u\abs{u}^{p-1}}_{H^{-1}} 
    + \norm{\nabla\rho}_{L^2}^{\min(2,p)} \fullstop
  \end{equation}
  On the other hand, thanks to \cref{lem:localization}-\cref{it:localization4}, we know that
  \begin{equation}\label{eq:inter_estimate_tmp11}
    V_2(x)\Phi(x) = (1+\smallo(1))V_2(0)\Phi(x)
  \end{equation}
  for any $x\in\R^n$.

We now prove \cref{eq:interaction_and_coef:coef}.  
  If $\alpha=1$ there is nothing to prove, so we can assume that $\alpha \neq 1$
  and we define 
  $\theta\defeq\frac{p\alpha^{p-1}V_2(0)}{\alpha-\alpha^p}$. 
  Recalling \cref{eq:inter_estimate_tmp10,eq:inter_estimate_tmp11}, it follows by \cref{eq:testing_with_orthogonality}
  that
  \begin{equation}\label{eq:clean_testing}\begin{aligned}
 \abs{\alpha-\alpha^p}\abs*{\int (U^p-(1+\smallo(1))\theta U^{p-1})\xi\Phi} \lesssim 
    \smallo(1)\norm{\nabla\rho}_{L^2} 
    + \norm{\lapl u + u\abs{u}^{p-1}}_{H^{-1}} 
    + \norm{\nabla\rho}_{L^2}^{\min(2,p)}
    \fullstop
  \end{aligned}\end{equation}
  This latter inequality is very strong since it holds both with $\xi=U$ and $\xi=\partial_\lambda U$, with the constant $\theta$ independent of this choice. 
  Since $\Phi$ is identically $1$ on a large ball centered at $0$ where $U$ has almost all the mass (see \cref{lem:localization}-\cref{it:localization1}), we have
  \begin{equation}\label{eq:inter_estim_tmp100}
    \int \bigl(U^p-(1+\smallo(1))\theta U^{p-1}\bigr)\xi\Phi = \int U^p\xi-\theta\int U^{p-1}\xi + \smallo(1) \fullstop
  \end{equation}
  Our goal is to show that the right-hand side cannot be very small both when $\xi=U$ and 
  when $\xi=\partial_\lambda U$. In order to achieve this, it suffices to check that
  \begin{equation}\label{eq:inter_estim_tmp101}
    \frac{\int U^{2^*}}{\int U^p} \not
    = \frac{\int U^p\partial_\lambda U}{\int U^{p-1}\partial_\lambda U} \fullstop
  \end{equation}
  Note that the left-hand side is clearly positive, while the right-hand side is equal to zero. Indeed $\int U^p\partial_\lambda U$ is the derivative with respect to $\lambda$ of
  $\frac1{2^*}\int U[0,\lambda]^{2^*}$ which is independent of $\lambda$ (see  \cref{subsec:symmetries}), hence $\int U^p\partial_\lambda U=0$,
  while
  \begin{equation*}
    p\int U^{p-1}\partial_\lambda U 
    = \frac{\de}{\de\lambda}\Bigr|_{\lambda=1}\int U[0,\lambda]^p =
    \frac{\de}{\de\lambda}\Bigr|_{\lambda=1}\biggl(\lambda^{\frac{2-n}2}\int U[0, 1]^p \biggr)
    = 
    \frac{2-n}{2}\int U^p \not = 0 \fullstop
  \end{equation*}
Thus, as a consequence of \cref{eq:inter_estim_tmp100,eq:inter_estim_tmp101} we deduce that
  \begin{equation*}
    \max_{\xi\in\{U,\partial_\lambda U\}}
    \left(\abs*{\int (U^p-(1+\smallo(1))\theta U^{p-1})\xi\Phi}\right)
    \gtrsim 1
  \end{equation*}
  and therefore, choosing $\xi$ so that the maximum above is attained, \cref{eq:clean_testing} implies 
  \cref{eq:interaction_and_coef:coef}.
  
  Now that we have proven \cref{eq:interaction_and_coef:coef}, choosing $\xi=U$ in 
  \cref{eq:testing_with_orthogonality} we obtain
  \begin{equation*}
    \abs*{\int U^pV\Phi} 
    \lesssim 
    \smallo(1)\norm{\nabla\rho}_{L^2} 
    + \norm{\lapl u + u\abs{u}^{p-1}}_{H^{-1}} 
    + \norm{\nabla\rho}_{L^2}^{\min(2,p)}
  \end{equation*}
  and in particular
  \begin{equation*}
    \abs*{\int_{B(0,1)} U^pU_j} 
    \lesssim 
    \smallo(1)\norm{\nabla\rho}_{L^2} 
    + \norm{\lapl u + u\abs{u}^{p-1}}_{H^{-1}} 
    + \norm{\nabla\rho}_{L^2}^{\min(2,p)}
  \end{equation*}
  for any $j \not = i$. Thanks to \cref{cor:interaction_integral_localized}, we deduce 
  \cref{eq:interaction_and_coef:interaction} for all $j > i$. Since for $j<i$ we already know the validity
  of \cref{eq:interaction_and_coef:interaction} by the induction (recall that $\int U^pU_j = \int U_j^pU$),
  this concludes the proof.
\end{proof}

\newcommand\res{\mathcal E} 
\newcommand\nes{\mathcal F} 
\newcommand\rbasis{B_\res} 
\newcommand\nbasis{B_\nes} 
\newcommand\LW{L^2_{(U+V)^{p-1}}} 
\newcommand{\tf}{\ensuremath{\tilde f}}%
\section{Counterexample in dimension \texorpdfstring{$n\ge 6$}{strictly larger than 5}}\label{sec:counterexample}
In this section we show that \cref{thm:main_close} does not hold when the dimension $n$ is strictly above $5$. 
The exact statement we want to prove is the following.
\begin{theorem}\label{thm:counterexample}
  For any dimension $n\ge 6$, there exists a family of functions $u_R\in H^1(\R^n)$ parametrized by a positive
  real number $R>1$ such that the following statement holds.
  
  For any choice of the parameters $\alpha,\beta,\lambda_1,\lambda_2>0$ and $z_1,z_2\in\R^n$, if we denote 
  $\sigma'\defeq \alpha U[z_1,\lambda_1] + \beta U[z_2,\lambda_2]$, then
  \begin{equation*}
    \norm{\lapl u_R + u_R\abs{u_R}^{p-1}}_{L^{(2^*)'}} \lesssim
    \zeta_n(\norm*{\nabla u_R-\nabla\sigma'}_{L^2}) \comma
  \end{equation*}
  where $\zeta_n:\oo0\infty\to\oo0\infty$ is defined as
  \begin{equation*}
    \zeta_n(t) \defeq 
    \begin{cases}
      \frac{t}{\abs{\log(t)}}&\quad\text{if $n=6$,}\\
      t^{\frac{10}9}&\quad\text{if $n=7$,}\\
      t^{\frac65}\abs{\log(t)}&\quad\text{if $n=8$,}\\
      t^{\frac{n+4}{n+2}}&\quad\text{if $n>8$.}\\
    \end{cases}
  \end{equation*}
  Furthermore, when we let $R\to\infty$, the family $u_R$ satisfies
  \begin{equation*}
    \norm{\nabla u_R-\nabla U[-Re_1,1]-\nabla U[Re_1, 1]}_{L^2} \to 0 \comma
  \end{equation*}
  where $e_1=(1,0,\dots,0)$ is the first vector of the canonical basis of $\R^n$.
  
  In particular, \cref{thm:main_close} cannot hold when $n\ge 6$.
\end{theorem}
\begin{remark}
  As can be seen from the statement, we prove more than the failure of \cref{thm:main_close}. Indeed the 
  counterexample is stronger than needed for the following reasons:
  \begin{itemize}
  \item The $H^{-1}$-norm is replaced with the stronger $L^{(2^*)'}$-norm.
  \item When $n\ge 7$, we show the existence of an exponent $\gamma>1$ such that \cref{thm:main_close} remains 
  false even if $\norm{\lapl u+u\abs{u}^{p-1}}_{H^{-1}}$ is raised to the power $\frac{1}{\gamma'}$ with 
  $\gamma'<\gamma$. 
  This shows that the infimum of the exponents that make \cref{thm:main_close} true, if it exists, is 
  greater than $1$ when $n\ge 7$.
  \item We allow the freedom to choose the coefficients in front of the Talenti bubbles.
  \end{itemize}
  Moreover, the counterexample is ``as simple as it can be''. 
  In fact, in the neighborhood of a single bubble \cref{thm:main_close} holds in every dimension, as observed in the paragraph
  before \cref{cor:single_bubble}. 
  Therefore at least two bubbles are required for a counterexample, and indeed our construction uses exactly two bubbles (that is 
  $\nu=2$ in the statement of \cref{thm:main_close}).
\end{remark}

Nonetheless, \cref{thm:counterexample} is not entirely satisfying.
In fact, as mentioned in the introduction, Struwe's \cref{thm:struwe_intro} (and then \cref{cor:main_pos}) deal with nonnegative function.
This shortcoming is solved by the following theorem that shows the existence of a \emph{nonnegative} 
counterexample, which negates
also the validity of \cref{cor:main_pos} when the dimension is strictly larger than $5$.
\begin{theorem}\label{thm:counterexample_pos}
  For any dimension $n\ge 6$, there exists a family of nonnegative functions $u^+_R\in H^1(\R^n)$ parametrized 
  by a positive real number $R>1$ such that the following statement holds.
  
  For any choice of the parameters $\alpha,\beta,\lambda_1,\lambda_2>0$ and $z_1,z_2\in\R^n$, if we denote
  $\sigma'\defeq \alpha U[z_1,\lambda_1] + \beta U[z_2,\lambda_2]$, it holds
  \begin{equation*}
    \norm{\lapl u^+_R + (u^+_R)^p}_{H^{-1}} \lesssim
    \xi_n(\norm*{\nabla u^+_R-\nabla\sigma'}_{L^2}) \comma
  \end{equation*}
  where $\xi_n(t)\defeq \sqrt{t\zeta_n(t)}$, and $\zeta_n$ is the function defined in 
  \cref{thm:counterexample}.
  Furthermore, when we let $R\to\infty$, the family satisfies
  \begin{equation}
  \label{eq:R infty u+}
    \norm{\nabla u^+_R-\nabla U[-Re_1,1]-\nabla U[Re_1, 1]}_{L^2} \to 0 \comma
  \end{equation}
  where $e_1=(1,0,\dots,0)$ is the first vector of the canonical basis of $\R^n$.
  
  In particular, \cref{cor:main_pos} cannot hold when $n\ge 6$.
\end{theorem}
\begin{remark}
  Note that in \cref{thm:counterexample_pos} we use again the $H^{-1}$-norm
  (instead of the $L^{(2^*)'}$-norm used in 
  \cref{thm:main_close}).
  Moreover the dependence given by $\xi_n$ is slightly worse than the dependence given by $\zeta_n$ (that, most
  likely, was already non-optimal).
  
  The notation $u_R^+$ is significative not only of the nonnegativity of the counterexample, but also of 
  the way it is constructed: it is exactly the positive part of the counterexample $u_R$ built in 
  \cref{thm:counterexample}.
  
  Let us remark that, perturbing suitably a family of nonnegative functions that satisfies 
  \cref{thm:counterexample_pos}, we can easily obtain a family of \emph{positive} functions that
  still satisfies \cref{thm:counterexample_pos}. Hence \cref{thm:counterexample_pos} holds even if the 
  functions are required to be strictly positive.
\end{remark}

Let us sketch briefly how the counterexample is built. From now on we will not show explicitly the dependence
of our construction from the real parameter $R>0$ (so we will write $u$ in place of $u_R$). 
Moreover, unless stated otherwise,  the dimension $n$ will always be greater or equal than $6$.

Let us fix two Talenti bubbles $U=U[-Re_1, 1]$ and $V=[Re_1,1]$.
The idea is to linearize the equation $\lapl u + u\abs{u}^{p-1}=0$ when $u$ is close to $U+V$. Hence, let 
$u=U+V+\rho$.
We shall ask $\rho$ to be $H^1$-orthogonal to the manifold of linear combinations of two Talenti bubbles 
(see \cref{eq:orthogonality_H1,eq:orthogonality_H2,eq:orthogonality_H3}). Indeed, under this assumption we shall have that $\norm{\nabla\rho}_{L^2}\approx d(u)$, where $d(u)$ is the $H^1$-distance between
$u$ and the manifold of all linear combinations of two Talenti bubbles.

Thanks to the estimate (see \cref{eq:apriori_error_estimate} below)
\begin{align*}
  \norm{\lapl u + u\abs{u}^{p-1}}_{L^{(2^*)'}} 
  = \norm{\lapl\rho + \left((U+V)^p-U^p-V^p\right) + p(U+V)^{p-1}\rho}_{L^{(2^*)'}} + \bigo(\norm{\nabla \rho}_{L^2}^p) \comma
\end{align*}
if we were able to solve $\lapl\rho + \left((U+V)^p-U^p-V^p\right) + p(U+V)^{p-1}\rho=0$, we would have 
finished. Unfortunately this is not possible because there are some 
nontrivial obstructions (that are a consequence of the orthogonality conditions we are imposing on $\rho$) related to the spectrum of $\frac{-\lapl}{(U+V)^{p-1}}$. 

For this reason, we consider instead a perturbation $\tf$ of $f\defeq (U+V)^p-U^p-V^p$ such that
$\lapl\rho + \tf + p(U+V)^{p-1}\rho=0$ becomes solvable. We build $\tf$ as a suitable projection of $f$
onto a subspace of eigenfunctions of $\frac{-\lapl}{(U+V)^{p-1}}$. This will allow us to prove the desired controls on
$\rho$, from which we will deduce that $u=U+V+\rho$ is the sought counterexample.

To be precise, we should say that the function $\rho$ that we are going to construct does not satisfy exactly the orthogonality conditions 
mentioned above. Nonetheless, we will be able to show (through a series of delicate properties of eigenspaces of close-by operators) that it \emph{almost} satisfies these conditions, and this will be enough to show that the distance 
of $u$ from the manifold of linear combinations of two Talenti bubbles cannot be much less than 
$\norm{\nabla \rho}_{L^2}$.

Let us remark that our construction is not explicit, as it depends on the solution of a partial differential equation.

\subsection{Notation and definitions for the counterexample}\label{sub:notation_counterexample}
The dimension $n \ge 6$ will be considered fixed, and all constants are implicitly allowed to depend on the 
dimension.
Let us fix $\eps=\eps(n)$ such that $(1-\eps)^2=\frac{p}{\Lambda}$, where $\Lambda^{-1}$ is the largest
eigenvalue of $\left(\frac{-\lapl}{U[0,1]^{p-1}}\right)^{-1}$ below $\frac1p$.  Our choice of $\eps$ ensures that
\begin{equation}\label{eq:eigenvalue_separation}
  \frac{p}{p(1-\eps)^{-1}} = 1-\eps \quad \text{ and } \quad \frac{p(1-\eps)^{-1}}{\Lambda} = 1-\eps \fullstop
\end{equation}
Let $U=U[-Re_1, 1]$ and $V=[Re_1,1]$ be two Talenti bubbles.
Our constructions and definitions depend on a real parameter $R\gg1$. 
The dependence from $R$ will not be explicit in our notation (that is, we will not add an index $R$ 
everywhere).
Given two expressions $A$ and $B$ (that depend on $R$), the notation $A=\smallo(B)$ means
that there exists a function $\omega:\oo0\infty\to\oo0\infty$ such that $\omega(R)\to0$ when 
$R\to\infty$ and $A\le \omega(R) B$.

Let us now introduce two subspaces of $\LW(\R^n)$.

Let $\res$ be the subspace generated by all eigenfunctions of 
$\left(\frac{-\lapl}{(U+V)^{p-1}}\right)^{-1}$ with eigenvalue larger than $\frac{1-\eps}{p}$, and fix an orthonormal basis $\rbasis$ of $\res$ made of eigenfunctions of
$\left(\frac{-\lapl}{(U+V)^{p-1}}\right)^{-1}$. Note that, since this basis is made of eigenfunctions, the functions in
$\rbasis$ are orthogonal also with respect to the $H^1$-scalar product.

Let $\nes$ be the subspace generated by all eigenfunctions of 
$\left(\frac{-\lapl}{U^{p-1}}\right)^{-1}$ and $\left(\frac{-\lapl}{V^{p-1}}\right)^{-1}$ with eigenvalue 
larger than $\frac{1-\eps}{p}$. Let us recall (see \cite[Appendix]{bianchi1991}) that the set $\nbasis$ 
composed of the $2(1+1+n)$ functions
\begin{equation*}
  \nbasis
  \defeq
  \bigg\{U, \partial_{\lambda}U, \partial_{z_i}U, V, \partial_{\lambda}V, \partial_{z_i}V    
  \bigg\}
\end{equation*}
is a basis of $\nes$.
Such a basis is not orthonormal (or orthogonal) with respect to neither the $\LW$-scalar product nor 
the $H^1$-scalar product. 
Nonetheless, if we let $R$ go to infinity, the $\LW$-scalar product between two functions in $\nbasis$ 
converges to $0$ and the $\LW$-norms of all those functions converge to some positive values. The same 
properties hold also for the $H^1$-norm. We will refer to these properties saying that $\nbasis$  is 
\emph{asymptotically quasi-orthonormal} with respect to both the $\LW$-scalar product and the $H^1$-scalar 
product.

Let $\pi:\LW(\R^n)\to \LW(\R^n)$ be the projection on the subspace orthogonal to $\res$ (with respect 
to the $\LW$-scalar product), set $f\defeq(U+V)^p-U^p-V^p$, and define
\begin{equation*}
  \tf \defeq (U+V)^{p-1}\cdot\pi\left(\frac{f}{(U+V)^{p-1}}\right) \fullstop
\end{equation*}
\begin{remark}
  Our choice of $\tf$ ensures (by construction) that $\tf$ is orthogonal, in the $L^2$-scalar product, 
  to $\res$. 
  Since this is fundamentally the only property that we need on $\tf$, one could also be tempted to define $\tf$ as 
  the projection, with respect to the $L^2$-scalar product, onto the subspace $L^2$-orthogonal to $\res$. 
  These two definitions are not the same and we have chosen ours because it allows us to prove 
  \cref{lem:counterexample_deep} (while it is not clear to us how to prove \cref{lem:counterexample_deep} with the other definition).
\end{remark}

Let us define the function $\rho\in \LW(\R^n)$ as the unique solution of
\begin{equation}\label{eq:def_rho}
  \lapl\rho+p(U+V)^{p-1}\rho+\tf=0
\end{equation}
such that $\pi(\rho)=\rho$ (namely $\rho$ is orthogonal, in the $\LW$-scalar product, to $\res$).
The existence of $\rho$ can be justified as follows.
If we denote $T=\left(\frac{-\lapl}{(U+V)^{p-1}}\right)^{-1}$, then \cref{eq:def_rho} is equivalent to
\begin{equation*}
  (\id-pT)\rho = T\left(\pi\left(\frac{f}{(U+V)^{p-1}}\right)\right) \fullstop
\end{equation*}
Since $\pi$ is the projection into $\res^{\perp}$ and $\res$ is an eigenspace for the self-adjoint operator 
$T$, the right-hand side belongs to $\res^{\perp}$.
Let $T|_{\res^{\perp}}$ denote the restriction of the
operator $T$ on the subspace $\res^{\perp}$.
Since (by definition of $\res$) $T|_{\res^{\perp}}$ has eigenvalues strictly smaller than $\frac{1-\eps}{p}$, it follows that $\id-pT|_{\res^{\perp}}:\res^{\perp}\to \res^{\perp}$ is an isomorphism.
Thus there exists a unique solution $\rho \in \res^{\perp}$ given by
\begin{equation*}
\rho\defeq \left(\id-pT|_{\res^{\perp}}\right)^{-1}T|_{\res^{\perp}}\left(\pi\left(\frac{f}{(U+V)^{p-1}}\right) \right)\fullstop
\end{equation*}
With this definition, we set $u\defeq U+V+\rho$.

\subsection{First observations}
Before going into the technical details of the proof, let us remark some properties that will be crucial
later on.

First of all, thanks to \cref{lem:H1estimate}, since $\pi(\rho)=\rho$ it holds 
\begin{equation}\label{eq:normrho}
  \norm{\nabla \rho}_{L^2} \approx \norm{\tf}_{H^{-1}} \fullstop
\end{equation}
Also, recalling \cref{eq:def_rho}, the identity
\begin{equation}\label{eq:tmp_first_remarks1}
  \lapl u + u\abs{u}^{p-1}
  = \bigl[f-\tf\bigr] + \left[(U+V+\rho)\abs{U+V+\rho}^{p-1}-(U+V)^p-p(U+V)^{p-1}\rho\right] 
\end{equation}
holds.
Moreover, since $1<p\le 2$, we have the pointwise elementary estimate
\begin{equation}\label{eq:tmp_first_remarks2}
  \abs*{(U+V+\rho)\abs{U+V+\rho}^{p-1}-(U+V)^p-p(U+V)^{p-1}\rho} \lesssim \abs{\rho}^p\fullstop
\end{equation}
In particular, combining \cref{eq:tmp_first_remarks1,eq:tmp_first_remarks2} and applying the 
Sobolev inequality, we deduce
\begin{equation}\label{eq:apriori_error_estimate}
  \norm{\lapl u + u\abs{u}^{p-1}}_{L^{(2^*)'}} \le 
  \norm{f-\tf}_{L^{(2^*)'}} + \norm{\rho^p}_{L^{(2^*)'}}
  \lesssim 
  \norm{f-\tf}_{L^{(2^*)'}} + \norm{\nabla\rho}_{L^2}^p\fullstop
\end{equation}
Hence, the main challenge is to estimate $\norm{f-\tf}_{L^{(2^*)'}}$.
This is done in the next sections, where we prove the crucial estimate in \cref{lem:counterexample_fundamental}.

Finally, as already mentioned before, let us emphasize that our choice of $\rho$
does not satisfy the orthogonality conditions we would desire (namely $\rho$ is not
orthogonal to $\nes$ with respect to the $H^1$-scalar product).
Nonetheless, we will show that it is \emph{almost} orthogonal and this will suffice to deduce that the 
function $u$ provides the desired family of counterexamples.
For this the first important step is to estimate the distance between the subspaces 
 $\res$ and $\nes$. This is the purpose of the next subsection.

\subsection{\texorpdfstring{The subspaces $\res$ and $\nes$}{The two eigenspaces} are very close}
When $R>0$ is large, since $U^{p-1}$ and $V^{p-1}$ are concentrated in regions  distant one from
the other, it is natural to expect that the lowest section of the spectrum of $\frac{-\lapl}{(U+V)^{p-1}}$ is 
approximately the sum of the lowest parts of the spectra of $\frac{-\lapl}{U^{p-1}}$ and 
$\frac{-\lapl}{V^{p-1}}$.
Thus the two subspaces $\res$ and $\nes$ should have the same dimension and
be, in some appropriate sense, very close one to the other.
This subsection is devoted exactly to this: formalizing and proving the mentioned ansatz.

Let us begin introducing a distance between subspaces of a Hilbert space. This definition is classical, 
see \cite[Equation (3)]{morris2010} and the references therein for the proofs of the properties we 
are going to state.
\begin{definition}
  Let $E, F$ be two finite subspaces of a Hilbert space $X$. The distance $d(E,F)$ between them is defined as
  \begin{equation*}
    d(E, F) \defeq d_H\bigl(\{\abs{x}\le 1\}\cap E, \{\abs{x}\le 1\}\cap F \bigr) \comma
  \end{equation*}
  where $d_H$ denotes the Hausdorff distance.
\end{definition}
This notion of distance enjoys a number of nice properties. We list some of them:
\begin{itemize}
 \item For any two subspaces $E$ and $F$ it holds $d(E, F)\le 1$.
 \item If two subspaces $E$ and $F$ have different dimensions, then $d(E, F)=1$.
 \item There are several equivalent definitions of such distance involving orthogonal projections. More precisely,
 given two subspaces $E$ and $F$, let $\pi_E:X\to E$ and $\pi_F:X\to F$ be the orthogonal projections onto $E$ and 
 $F$, respectively. Then the following identity holds:
 \begin{equation*}
    d(E, F) = \norm{\pi_E-\pi_F}_{op} \comma
 \end{equation*}
 where $\norm{\emptyparam}_{op}$ is the operator norm.
 Moreover, it also holds
 \begin{equation*}
    d(E, F) = \max\left(
    \sup_{e\in E\setminus \{0\}} \frac{\abs{e-\pi_F(e)}}{\abs{e}},
    \sup_{f\in F\setminus \{0\}} \frac{\abs{f-\pi_E(f)}}{\abs{f}}
    \right) \fullstop
 \end{equation*}
 If $\dim{E} = \dim{F}$, then the two suprema in the last formula are equal and it 
 holds
 \begin{equation*}
    d(E, F) 
    = \sup_{e\in E\setminus \{0\}} \frac{\abs{e-\pi_F(e)}}{\abs{e}}
    = \sup_{f\in F\setminus \{0\}} \frac{\abs{f-\pi_E(f)}}{\abs{f}}
    \fullstop
 \end{equation*}
 \item The distance does not change if we replace $E,F$ with $E^{\perp}$ and $F^{\perp}$, that is
 \begin{equation*}
    d(E, F) = d(E^{\perp}, F^{\perp}) \fullstop
 \end{equation*}
\end{itemize}

We are ready to prove that $\res$ and $\nes$ are close with respect to the distance defined above. The proof is 
based on a series of technical tools that are postponed to \cref{app:spectrum}.

\begin{proposition}\label{prop:realisnice}
  For any sufficiently large $R$ we have $\dim\res=\dim\nes=2n+4$ and
  \begin{equation*}
    d(\res, \nes) = \smallo(1) \comma
  \end{equation*}
  where the distance on the subspaces is induced by the $\LW$-norm.
\end{proposition}
\begin{proof}
  Let us notice that the subspace $\nes$ is the direct sum of the two subspaces $\nes_U$ and $\nes_V$
  that are generated by the eigenfunctions of $\left(\frac{-\lapl}{U^{p-1}}\right)^{-1}$ and 
  $\left(\frac{-\lapl}{V^{p-1}}\right)^{-1}$ with eigenvalue larger than $\frac{1-\eps}p$.
  The two bases $\rbasis$ and $\nbasis$ are respectively orthonormal and asymptotically quasi-orthonormal with
  respect to the $\LW$-scalar product (see \cref{sub:notation_counterexample}), hence the statement is equivalent to showing the following two facts:
  \begin{enumerate}[label=\textit{(\arabic*)}]
   \item \label{it:hard_close_EF} For any $\psi \in\rbasis$ there exists $\psi'\in\nes$ such that 
   $\norm{\psi-\psi'}_{\LW}=\smallo(1)$.
   \item \label{it:easy_close_EF} For any $\psi \in\nbasis$ there exists $\psi'\in\res$ such that 
   $\norm{\psi-\psi'}_{\LW}=\smallo(1)$.
  \end{enumerate}
 We begin by proving \cref{it:easy_close_EF}.
  Without loss of generality we can assume $\psi\in\nbasis\cap\nes_U$. Therefore it holds
  $-\lapl \psi = \lambda U^{p-1}\psi$ with $\lambda\in\{1, p\}$, and rearranging the terms we get
  \begin{equation*}
    -\lapl\psi - \lambda (U+V)^{p-1}\psi = \lambda\left(U^{p-1}-(U+V)^{p-1}\right)\psi \fullstop
  \end{equation*}
  We can now apply \cref{lem:approximation_eigenfunctions} in conjunction with 
  \cref{cor:rellich2} to obtain
  \begin{equation*}
    \sum_k \alpha_k^2\left(1-\frac{\lambda}{\lambda_k}\right)^2 
    \lesssim \norm{\left(U^{p-1}-(U+V)^{p-1}\right)\psi}_{L^2} = \smallo(1) \comma
  \end{equation*}
  where $\psi = \sum\alpha_k\psi_k$ and $(\lambda_k^{-1},\psi_k)$ is the sequence of eigenvalues and normalized
  eigenfunctions of the operator $\left(\frac{-\lapl}{(U+V)^{p-1}}\right)^{-1}$. Choosing $\psi'\defeq\pi_\res(\psi)$ as the
  projection of $\psi$ onto $\res$ we have
  \begin{equation*}
    \psi' = \sum_{\lambda_k < p(1-\eps)^{-1}}\alpha_k\psi_k\comma
  \end{equation*}
  and thus
  \begin{equation*}
    \norm{\psi-\psi'}_{\LW}^2 = \sum_{\lambda_k \ge p(1-\eps)^{-1}}\alpha_k^2
    \le \frac1{\eps^2}\sum_{\lambda_k \ge p(1-\eps)^{-1}}
    \alpha_k^2 \left(1-\frac{\lambda}{\lambda_k}\right)^2
    = \smallo(1) \comma
  \end{equation*}
  that is exactly the statement of \cref{it:easy_close_EF}.
  
  The proof of \cref{it:hard_close_EF} is similar to the one of \cref{it:easy_close_EF}; the main difference
  being that, instead of \cref{lem:approximation_eigenfunctions}, we will use 
  \cref{prop:restriction_eigenfunction}.
  Let us consider $\psi\in\rbasis$ that solves $-\lapl\psi = \lambda (U+V)^{p-1}\psi$ 
  with $0< \lambda < p(1-\eps)^{-1}$. Define two functions $\varphi_U$ and $\varphi_V$ as
  \begin{align*}
    \varphi_U(x) \defeq \psi(x)\,\eta\left(\frac{x+Re_1}{R/2}\right) \comma\qquad
    \varphi_V(x) \defeq \psi(x)\,\eta\left(\frac{x-Re_1}{R/2}\right) \comma
  \end{align*}
  where $\eta$ is a smooth bump function as described in the statement of 
  \cref{prop:restriction_eigenfunction}.
  If we apply \cref{prop:restriction_eigenfunction} and we follow the same reasoning we have used to prove
  \cref{it:easy_close_EF} (recalling \cref{eq:eigenvalue_separation}), we find two functions 
  $\psi_U\in\nes_U$ and $\psi_V\in\nes_V$ such that
  \begin{align*}
    \norm{\psi_U-\varphi_U}_{L^2_{U^{p-1}}} = \smallo(1) \comma \qquad
    \norm{\psi_V-\varphi_V}_{L^2_{V^{p-1}}} = \smallo(1) \fullstop
  \end{align*}
  Since $\varphi_U$ is supported inside the set $\{U\ge V\}$ and $\varphi_V$ is supported inside $\{V\ge U\}$, the previous estimates 
  can be upgraded to
  \begin{equation}\label{eq:close_subspace_tmp}\begin{aligned}
    &\norm{\psi_U-\varphi_U}_{\LW} 
    \lesssim \norm{\psi_U-\varphi_U}_{L^2_{U^{p-1}}} + 
    \norm{\psi_U}_{L^2_{V^{p-1}}}
    = \smallo(1) \comma \\
    &\norm{\psi_V-\varphi_V}_{\LW} 
    \lesssim \norm{\psi_V-\varphi_V}_{L^2_{V^{p-1}}} 
    + \norm{\psi_V}_{L^2_{U^{p-1}}}
    = \smallo(1) \fullstop
  \end{aligned}\end{equation}
  Let us define $\psi'\defeq\psi_U+\psi_V$. Of course it holds $\psi'\in\nes$. Then, by the triangle inequality
  and \cref{eq:close_subspace_tmp}, we obtain
  \begin{align*}
    \norm{\psi-\psi'}_{\LW} 
    &\le \smallo(1) + \norm{\psi-\varphi_U-\varphi_V}_{\LW} \\
    &\le \smallo(1) + \int_{\R^n\setminus\left[B(-Re_1,R/2)\cup B(Re_1,R/2)\right]} \psi^2(U+V)^{p-1} \fullstop
  \end{align*}
  The proof is finished since also the last  integral is $\smallo(1)$ thanks to 
  \cref{lem:concentration_eigenfunctions}.
\end{proof}

\begin{remark}
  The statement of \cref{prop:realisnice} can be generalized to cover much more general situations. 
  Even if we will not need it, let us give a possible generalization.
  
  For a fixed $n\ge 5$, let $(w_k)_{k\in\N}\subseteq L^{\frac n2}(\R^n)$ be a sequence of positive 
  weights of the form $w_k=\sum_{i=1}^{\nu} U[x^{(k)}_i, \lambda_i]^{p-1}$, where $(\lambda_i)_{1\le i\le \nu}$ are fixed and 
  $(x^{(k)}_i)_{1\le i\le \nu,k\in\N}$ is a sequence of $\nu$-tuples of points in $\R^n$ such that 
  $\abs{x^{(k)}_i-x^{(k)}_j}\to\infty$ for any $i\not = j$. 
  For a fixed $\mu>0$, let $\res^{(k)}_\mu$ be the subspace of the eigenfunctions of 
  $\left(\frac{-\lapl}{w_k}\right)^{-1}$ with eigenvalue greater or equal than $\mu^{-1}$,
  and let $\nes^{(k)}_{\mu,i}$ be the subspace of the eigenfunctions of 
  $\left(\frac{-\lapl}{U[x^{(k)}_i, \lambda_i]^{p-1}}\right)^{-1}$ with eigenvalue greater or
  equal than $\mu^{-1}$.
  Then, if $\mu^{-1}$ is not an eigenvalue for $\left(\frac{-\lapl}{U[0,1]^{p-1}}\right)^{-1}$, it holds
  \begin{equation*}
    d\left(\res^{(k)}_\mu, 
    \nes^{(k)}_{\mu,1}\oplus\cdots\oplus\nes^{(k)}_{\mu,\nu}\right) 
    \to 0 \comma
  \end{equation*}
  where the distance between subspaces is induced by the $L^2_{w_k}$-norm.
  
  With some care it would be possible to obtain a similar result also for weights much more 
  general than finite sums of Talenti bubbles (in the spirit of \cref{prop:restriction_eigenfunction}). 
  We will not do that, as it is out of the scope of this note.
\end{remark}

\subsection{The norm of \texorpdfstring{$\rho$}{rho} is asymptotically larger than \texorpdfstring{$||\lapl u + u|u|^{p-1}||_{L^{(2^*)'}}$}{the discrepancy}}
Our intuition tells us that $\norm{\nabla \rho}_{L^2}$ gives a good approximation of the $H^1$-distance
of $u$ from the manifold of linear combinations of two Talenti bubbles. Thus, since our final goal is
proving that such a distance is asymptotically larger than $\norm{\lapl u + u\abs{u}^{p-1}}_{L^{(2^*)'}}$, we 
devote this section to the proof that $\norm{\nabla \rho}_{L^2}$ is asymptotically larger than 
$\norm{\lapl u + u\abs{u}^{p-1}}_{L^{(2^*)'}}$.

Our approach is very direct: we compute all the involved quantities and, in the end, compare them.

Let us emphasize that the elementary and explicit estimate
\begin{equation*}
  \norm{f}_{L^2} = \smallo(\norm{f}_{H^{-1}}) \comma
\end{equation*}
where $f=(U+V)^p-U^p-V^p$, is fundamentally equivalent to what we want to prove.
The validity of the mentioned estimate (that follows from \cref{lem:normf,lem:normf_l2} below) depends heavily
on the dimensional condition $n\ge 6$.

\begin{lemma}\label{lem:normf}
  It holds
  \begin{equation*}
    \norm{f}_{H^{-1}} \gtrsim 
    \begin{cases}
      R^{-4}\log(R)^{\frac12} \quad&\text{if $n=6$,}\\
       R^{-\frac{n+2}2} \quad&\text{if $n\ge 7$.}
    \end{cases}
  \end{equation*}
\end{lemma}
\begin{proof}
  First we deal with the easier case $n\ge 7$.
  Let us fix a smooth bump function 
  $\eta\in C_c^{\infty}(\R^n)$ such that $0\le \eta\le 1$ everywhere, $\eta\equiv 1$ in $B(1,\frac14)$, and
  $\eta\equiv 0$ in $B(1,\frac12)^{\complement}$. 
  In $B(4Re_1, R)$ the function $f$ is comparable to $R^{-n-2}$ (that coincides with the
  decay of $U[0,1]^p$), hence
  it holds
  \begin{equation*}
    R^{-2} = R^n \cdot R^{-n-2}
    \lesssim \int_{\R^n} f(x)\, \eta\left(\frac x{4R}\right)\de x
    \le \norm{f}_{H^{-1}} \norm*{\nabla \left(\eta\left(\frac{\emptyparam}{4R}\right)\right)}_{L^2}
    \lesssim R^{\frac{n-2}2}\norm{f}_{H^{-1}}
  \end{equation*}
  that gives
  \begin{equation*}
    R^{-\frac{n+2}2} \lesssim \norm{f}_{H^{-1}} \comma
  \end{equation*}
  as desired.
  
  When $n=6$, we prove the result testing $f$ against the function $f^{\frac12}$.
  Let us remark that, since $n=6$, it holds $2^*=3$, $p=2$, and thus in particular $f=2UV$.
  
  We have
  \begin{equation}\label{eq:normf_tmp1}
    \int_{\R^6} U^{\frac32}V^{\frac32} \lesssim
    \norm{f}_{H^{-1}}\left(\int_{\R^6}\abs{\nabla(f^{\frac12})}^2\right)^{\frac12} \fullstop
  \end{equation}
  We estimate independently the left-hand side and the right-hand side.
  
  Applying \cref{prop:interaction_approx} with parameters $\alpha=\beta=\frac32$ yields
  \begin{equation}\label{eq:normf_tmp2}
    \int_{\R^6} U^{\frac32}V^{\frac32} \approx R^{-6}\log(R) \fullstop
  \end{equation}
On the other hand it holds
  \begin{equation*}
    \abs{\nabla(f^{\frac12})}^2 \approx \abs{\nabla U}^2U^{-1}V + \abs{\nabla V}^2V^{-1}U
  \end{equation*}
  and, since $\abs{\nabla U} \approx U^{\frac{n-1}{n-2}}=U^{\frac54}$, we obtain
  \begin{equation*}
    \abs{\nabla(f^{\frac12})}^2 \approx U^{\frac32}V + UV^{\frac32} \fullstop
  \end{equation*}
  Thus, applying \cref{prop:integral_simple_bubbles} with $a=c=\frac32$ and $b=d=1$, we deduce
  \begin{equation}\label{eq:normf_tmp3}
    \left(\int_{\R^6}\abs{\nabla(f^{\frac12})}^2\right)^{\frac12}
    \approx
    \left(\int_{\R^6}U^{\frac32}V\right)^{\frac12}
    \approx
    \left(R^{-4}\log(R)\right)^{\frac12}
    \approx 
    R^{-2}\log(R)^{\frac12} \fullstop
  \end{equation}
  Finally, combining \cref{eq:normf_tmp1,eq:normf_tmp2,eq:normf_tmp3} we get
  \begin{equation*}
    R^{-6}\log(R) \lesssim \norm{f}_{H^{-1}}\cdot R^{-2}\log(R)^{\frac12} \comma
  \end{equation*}
  that implies the desired estimate.
\end{proof}

\begin{lemma}\label{lem:normf_l2}
  It holds
  \begin{equation*}
    \norm{f}_{L^2} \approx 
    \begin{cases}
      R^{-4} \quad&\text{if $n=6$,}\\
      R^{-5} \quad&\text{if $n=7$,}\\
      R^{-6}\log(R)^{\frac12} \quad&\text{if $n=8$,}\\
      R^{-\frac{n+4}2} \quad&\text{if $n>8$.}
    \end{cases}
  \end{equation*}
\end{lemma}
\begin{proof}
 We note that $\norm{f}^2_{L^2} = \int_{\R^n} \varphi(U, V)$ where 
  \begin{equation*}
    \varphi(x, y)=\left((x+y)^p-x^p-y^p\right)^2 \fullstop
  \end{equation*}
  The mentioned function $\varphi$ satisfies the hypotheses of \cref{prop:integral_simple_bubbles} with
  $a=2p-2,b=2,c=2,d=2p-2$. Hence $\norm{f}^2_{L^2}\approx\Phi_R(2p-2,2,2,2p-2)$, and computing the value of 
  $\Phi_R(2p-2, 2, 2, 2p-2)^{\frac12}$ yields the desired result.
\end{proof}

\begin{lemma}\label{lem:counterexample_deep}
  It holds
  \begin{equation*}
    \norm{f-\tf}_{L^{(2^*)'}} \lesssim \norm{f}_{L^2} \fullstop
  \end{equation*}
\end{lemma}
\begin{proof}
  By definition of $\tf$, we have
  \begin{align*}
    f-\tf &= (U+V)^{p-1}\left(\frac{f}{(U+V)^{p-1}}-\pi\left(\frac{f}{(U+V)^{p-1}}\right)\right) \\
    &= (U+V)^{p-1}\sum_{\psi_\res\in\rbasis} \scalprod{\frac{f}{(U+V)^{p-1}}}{\psi_\res}_{\LW}\psi_\res \\
    &= (U+V)^{p-1}\sum_{\psi_\res\in\rbasis}\scalprod{f}{\psi_\res}_{L^2}\,\psi_\res
    \fullstop
  \end{align*}
  Thus, taking the $L^{(2^*)'}$-norm and applying H\"older's inequality with exponents $\frac12+\frac1n=\frac1{(2^*)'}$, we obtain
  \begin{align*}
    \norm{f-\tf}_{L^{(2^*)'}} 
    &\le \sum_{\psi_\res\in\rbasis} \abs{\scalprod{f}{\psi_\res}_{L^2}}
    \cdot\norm{(U+V)^{p-1}\psi_\res}_{L^{(2^*)'}} \\
    &\lesssim \sum_{\psi_\res\in\rbasis} \norm{f}_{L^2}\cdot
    \norm{\psi_\res}_{L^2}\cdot
    \norm{(U+V)^{\frac{p-1}2}\psi_\res}_{L^2}\cdot
    \norm{(U+V)^{\frac{p-1}2}}_{L^n}\\
    &\lesssim \norm{f}_{L^2} \sum_{\psi_\res\in\rbasis}
    \norm{\psi_\res}_{L^2} \cdot \norm{\psi_\res}_{\LW}
  \end{align*}
  and the conclusion follows applying \cref{lem:eigenfunctions_integrability}.
\end{proof}

\begin{lemma}\label{lem:counterexample_fundamental}
  It holds
  \begin{equation*}
    \norm{f-\tf}_{L^{(2^*)'}} \lesssim \zeta_n(\norm{\nabla\rho}_{L^2})\comma
  \end{equation*}
  where $\zeta_n$ is the same function considered in \cref{thm:counterexample}.
\end{lemma}
\begin{proof}
  The estimates contained in \cref{lem:normf_l2,lem:normf} tell us that
  \begin{equation}\label{eq:counterexample_fundtmp1}
    \norm{f}_{L^2}\lesssim \zeta_n(\norm{f}_{H^{-1}}) \comma
  \end{equation}
  in particular $\norm{f}_{L^2}=\smallo(\norm{f}_{H^{-1}})$. 
  Hence, thanks to \cref{lem:counterexample_deep} and Sobolev inequality (that by duality implies the embedding $L^{(2^*)'}\embedding H^{-1}$), we have
  \begin{equation*}
    \norm{f-\tf}_{H^{-1}} \lesssim \norm{f-\tf}_{L^{(2^*)'}} 
    \lesssim \norm{f}_{L^2} = \smallo(\norm{f}_{H^{-1}})\comma
  \end{equation*}
therefore, recalling \cref{eq:normrho}, we obtain
  \begin{equation}\label{eq:counterexample_fundtmp2}
    \norm{\nabla\rho}_{L^2}\approx \norm{\tf}_{H^{-1}} \approx \norm{f}_{H^{-1}} 
    \fullstop
  \end{equation}
  Then the statement follows from \cref{eq:counterexample_fundtmp1,eq:counterexample_fundtmp2}.
\end{proof}

Thanks to all the previous results, we can now prove the main proposition of this
section.
\begin{proposition}\label{prop:counterexample_main_estimate}
  It holds 
  \begin{equation*}
    \norm{\lapl u + u\abs{u}^{p-1}}_{L^{(2^*)'}} \lesssim \zeta_n(\norm{\nabla\rho}_{L^2}) \comma
  \end{equation*}
  where $\zeta_n$ is the same function considered in \cref{thm:counterexample}.
\end{proposition}
\begin{proof}
  Thanks to \cref{eq:apriori_error_estimate,lem:counterexample_fundamental} we have
  \begin{equation}\label{eq:counterexample_maintmp3}
    \norm{\lapl u + u\abs{u}^{p-1}}_{L^{(2^*)'}} \lesssim \norm{\nabla\rho}_{L^2}^p
    + \zeta_n(\norm{\nabla\rho}_{L^2}),
  \end{equation}
  and this concludes the proof since $\norm{\nabla\rho}_{L^2}^p\ll \zeta_n(\norm{\nabla\rho}_{L^2})$.
\end{proof}

\subsection{The function \texorpdfstring{$u$}{u} is a real counterexample}
It is now time to prove that the function $u$ is the desired counterexample.
The only thing that is still missing is the fact that $\norm{\nabla\rho}_{L^2}$ is comparable to
the $H^1$-distance of $u$ from the manifold of linear combinations of two Talenti bubbles.
The rough idea is that this must be true since, thanks to \cref{prop:realisnice}, $\rho$ is almost orthogonal
to the mentioned manifold in $U+V$. However, transforming this intuition into a proof requires some care.

Let us begin with three technical lemmas. All of them are somehow related to the fact that many 
different norms are involved in our computations (i.e. $H^1, H^{-1}, \LW$) and it is 
crucial to control adequately one with the other.
\begin{lemma}\label{lem:equivalent_norms_nes}
  On the subspace $\nes$ the two norms $\norm{\emptyparam}_{\LW}$ and $\norm{\emptyparam}_{H^1}$ are 
comparable, uniformly as $R\to \infty$. Equivalently, there exist constants $C$ and $R_0$ such that, for any $R\geq R_0$, 
  \begin{equation*}
    C^{-1}\norm{\varphi}_{\LW} \le \norm{\nabla \varphi}_{L^2} \le C\norm{\varphi}_{\LW}
  \end{equation*}
  for any $\varphi\in\nes$.
\end{lemma}
\begin{proof}
  Let us recall that $\nbasis$ is a basis for $\nes$ which is asymptotically quasi-orthonormal (as $R\to \infty$) with respect to 
  both the $\LW$-scalar product and the $H^1$-scalar product (see \cref{sub:notation_counterexample}). This fact  implies that, for $R\gg 1$, the two norms are 
comparable on $\nes$ independently of $R$.
\end{proof}

\begin{lemma}\label{lem:almost_H1_orthogonality}
Let $\nes^{\perp}$ be the orthogonal complement of $\nes$ with respect to the 
  $\LW$-scalar product.
  For any $\varphi\in H^1(\R^n) \cap \nes^{\perp}$ it holds
  \begin{equation*}
    \abs*{\scalprod{\nabla\psi_\nes}{\nabla\varphi}} \le \smallo(\norm{\nabla \varphi}_{L^2})
  \end{equation*}
  for any $\psi_\nes\in\nbasis$. 
\end{lemma}
\begin{proof}
  Without loss of generality we can assume that $-\lapl \psi_\nes = \lambda\psi_\nes U^{p-1}$ with
  $\lambda\in \{1,p\}$.
  Hence, by the assumption $\varphi\in\nes^{\perp}$, applying Cauchy-Schwarz inequality we obtain
  \begin{align*}
    \scalprod{\nabla\psi_\nes}{\nabla\varphi} 
    &= \lambda \int_{\R^n} \psi_\nes U^{p-1} \varphi 
    = \lambda \int_{\R^n} \psi_\nes \left(U^{p-1} - (U+V)^{p-1}\right) \varphi \\
    &\lesssim \left(\int_{\R^n}\psi_\nes^2 
    \frac{\left((U+V)^{p-1}-U^{p-1}
    \right)^2}{(U+V)^{p-1}}\right)^{\frac12} 
    \norm{\varphi}_{\LW}
    \fullstop
  \end{align*}
  The statement now follows from the fact that the first term goes to $0$ when $R\to\infty$, while the
  second factor is bounded by $\norm{\nabla\varphi}_{L^2}$ thanks to \cref{prop:weighted_cpt_embedding}.
\end{proof}

\begin{lemma}\label{lem:U_differentiable_H1}
  Let $\mathcal U:\R\times\R^n\times\oo{0}{\infty}\to H^1(\R^n)$ be the function that maps 
  $(\alpha, z, \lambda)$ onto $\alpha U[z, \lambda]$. 
  The function $\mathcal U$ is differentiable (as a function with values in $H^1(\R^n)$) and its gradient at 
  $(\alpha, z, \lambda)$ is given by
  \begin{equation*}
    \nabla \mathcal U (\alpha, z, \lambda) = 
    \left(U[z,\lambda], \alpha\nabla_z U[z,\lambda], \alpha\partial_\lambda U[z,\lambda]\right) \fullstop
  \end{equation*}
\end{lemma}
\begin{proof}
  The statement follows from the fact that the gradients of the partial derivatives 
  $U[z,\lambda]$, $\alpha\nabla_z U[z,\lambda]$, $\alpha\partial_\lambda U[z,\lambda]$ are (locally with respect to the parameters 
  $(\alpha,z,\lambda)$) dominated by an $L^2(\R^n)$-function (in particular by a multiple of 
  $(1+\abs{x})^{1-n}$), so the result follows by dominated convergence.
\end{proof}

\begin{proposition}\label{prop:distance_is_rho}
  The $H^1$-norm of $\rho$ approximates the $H^1$-distance of $u$ from the manifold of linear combinations of
  two Talenti bubbles. More precisely, if we denote $\sigma'\defeq \alpha U[z_1,\lambda_1] + \beta U[z_2,\lambda_2]$, it holds
  \begin{equation*}
    \inf_{\sigma'}\norm{\nabla u-\nabla \sigma'}_{L^2} \gtrsim \norm{\nabla\rho}_{L^2} \comma
  \end{equation*}
  where the infimum is taken over all choices of the parameters $\alpha,\beta,\lambda_1,\lambda_2>0$ 
  and $z_1,z_2\in\R^n$.
\end{proposition}
\begin{proof}
Given $\sigma'=\alpha U[z_1,\lambda_1] + \beta U[z_2,\lambda_2]$,
  we notice that $u-\sigma'=\rho + (\sigma-\sigma')$. 
  The core idea of the proof is to show that there cannot be extreme cancellation when computing the 
  $H^1$-norm of $\rho + (\sigma-\sigma')$.
  For this, we first show that there is not extreme cancellation when computing the norm of 
  $\sigma-\sigma'$, and we then exploit that $\rho$ is \emph{almost orthogonal} to $\sigma-\sigma'$ to obtain the result.

  Let us define $U'\defeq U[z_1, \lambda_1]$ and $V'\defeq U[z_2, \lambda_2]$.
  
  If $\norm{\sigma-\sigma'}_{H^1} \ge 2\norm{\nabla\rho}_{L^2}$ then the statement trivially holds, so 
  we can assume that $\norm{\sigma-\sigma'}_{H^1} \lesssim \norm{\nabla\rho}_{L^2}$. Thus,
  since $\norm{\nabla\rho}_{L^2}=\smallo(1)$, 
the quantity $\delta=\delta(\alpha, \lambda_1, z_1, \beta, \lambda_2, z_2)$ defined as
  \begin{equation*}
    \delta \defeq \abs{\alpha-1}+\abs{\lambda_1-1}+\abs{z_1+Re_1}
    + \abs{\beta-1}+\abs{\lambda_2-1}+\abs{z_2-Re_1}
  \end{equation*}
  is also $\smallo(1)$ (note that, without loss of generality, we have assumed that
  $U'$ is close to $U$ and $V'$ is close to $V$).
  
  Our first goal is estimating $\norm{\sigma-\sigma'}_{H^1}$. We note that
  \begin{equation}\label{eq:counterexample_thm_tmp1}
    \norm{\sigma-\sigma'}_{H^1}^2 
    = \norm{U-\alpha U'}_{H^1}^2
    + \norm{V- \beta V'}_{H^1}^2 + 2 \scalprod{U-\alpha U'}{V-\beta V'}_{H^1} \fullstop
  \end{equation}
  As stated in \cref{lem:U_differentiable_H1}, the map $\mathcal U$ is differentiable and the components of
  its gradient at the point $(1, 0, 1)$ are nonzero and $H^1$-orthogonal. Hence, given that the involved quantities are 
  translation invariant, we get
  \begin{align*}
    \norm{U - \alpha U'}_{H^1} = \norm{U[-Re_1,1]-\alpha U[z_1, \lambda_1]}_{H^1} 
    &\approx \abs{\alpha-1}+\abs{\lambda_1-1}+\abs{z_1+Re_1} \comma \\
    \norm{V - \beta V'}_{H^1} = \norm{U[Re_1,1]-\beta U[z_2, \lambda_2]]}_{H^1} 
    &\approx \abs{\beta-1}+\abs{\lambda_2-1}+\abs{z_2-Re_1} \fullstop
  \end{align*}
  Using again the differentiability of $\mathcal U$, we also obtain
  \begin{align*}
    &\abs*{\scalprod{U - \alpha U'}{V - \beta V'}_{H^1}}= 
    \abs*{\scalprod{U[-Re_1,1]-\alpha U[z_1, \lambda_1]}{U[Re_1,1]-\beta U[z_2, \lambda_2]}_{H^1}}\\
    &\quad\quad\lesssim \smallo(\delta^2) 
    + \delta^2 \abs*{\scalprod{\nabla\mathcal U(1, -Re_1, 1)}{\nabla\mathcal U(1, Re_1, 1)}_{H^1}} 
    = \smallo(\delta^2) \fullstop
  \end{align*}
  The last three estimates, together with \cref{eq:counterexample_thm_tmp1}, imply that
  \begin{equation}\label{eq:norm_diff_sigma}	
   \norm{\sigma-\sigma'}_{H^1} \approx \delta \fullstop
  \end{equation}
  We now show that
  \begin{equation}
  \label{eq:last}
    \abs{\scalprod{\sigma-\sigma'}{\rho}_{H^1}} 
    = \smallo\left(\norm{\sigma-\sigma'}_{H^1}\cdot\norm{\nabla\rho}_{L^2}\right)\fullstop
  \end{equation}
  
  Let $\tilde\rho\in\nes^{\perp}$ be the orthogonal projection of $\rho$ onto $\nes^{\perp}$, with respect to the $\LW$-scalar product.
  Since $\rho\in\res^{\perp}$, thanks to \cref{prop:realisnice} we know that
  \begin{equation*}
    \norm{\rho-\tilde\rho}_{\LW} = \smallo(\norm{\rho}_{\LW}) = \smallo(\norm{\nabla\rho}_{L^2})\comma
  \end{equation*}
  and given that $\rho-\tilde\rho\in\nes$ we can apply \cref{lem:equivalent_norms_nes} and deduce that
  \begin{equation}\label{eq:counterexample_thm_norm_tilderho}
    \norm{\nabla\rho-\nabla\tilde\rho}_{L^2} = \smallo(\norm{\nabla\rho}_{L^2}) \fullstop
  \end{equation}
Thus, thanks to \cref{eq:counterexample_thm_norm_tilderho} and Cauchy-Schwarz inequality, we get
  \begin{align*}
    \abs{\scalprod{\sigma-\sigma'}{\rho}_{H^1}} 
    &\le 
    \abs{\scalprod{\sigma-\sigma'}{\tilde\rho}_{H^1}}
    +
    \norm{\sigma-\sigma'}_{H^1}\norm{\nabla\rho-\nabla\tilde\rho}_{L^2} \\
    &=\abs{\scalprod{\sigma-\sigma'}{\tilde\rho}_{H^1}} 
    + \smallo(\norm{\sigma-\sigma'}_{H^1}\norm{\nabla\rho}_{L^2}) \fullstop
  \end{align*}
  In order to estimate $\abs{\scalprod{\sigma-\sigma'}{\tilde\rho}_{H^1}}$, we split it as
  \begin{equation*}	
    \abs{\scalprod{\sigma-\sigma'}{\tilde\rho}_{H^1}} \le 
    \abs{\scalprod{U-\alpha U'}{\tilde\rho}_{H^1}}
    + \abs{\scalprod{V-\beta V'}{\tilde\rho}_{H^1}} \fullstop
  \end{equation*}
  We will focus on $\scalprod{U-\alpha U'}{\tilde\rho}_{H^1}$, as the other term can be handled analogously.
  It holds
  \begin{align*}
    \scalprod{U-\alpha U'}{\tilde\rho}_{H^1}
    &= \scalprod
    {(1-\alpha)U -\alpha(\lambda_1-1)\partial_\lambda U - \alpha (z_1+Re_1)\cdot\nabla_zU}
    {\tilde\rho}_{H^1} \\
    &\phantom{=}+ \alpha \scalprod
    {U+(\lambda_1-1)\partial_\lambda U + (z_1+Re_1)\nabla_z U - U'}
    {\tilde\rho}_{H^1} \fullstop
  \end{align*}
  Recalling \cref{lem:almost_H1_orthogonality}, the first of the two terms can be controlled as
  \begin{align*}
    &\abs*{\scalprod{(1-\alpha)U -\alpha(\lambda_1-1)\partial_\lambda U - (z_1+Re_1)\cdot\nabla_zU}
    {\tilde\rho}_{H^1}} \\
    &\quad\quad \lesssim 
    \left(\abs{1-\alpha} + \abs{\alpha}\abs{\lambda_1-1} + \abs{\alpha}\abs{z_1+Re_1}\right)
    \cdot \smallo(\norm{\nabla\tilde\rho}_{L^2}) \\
    &\quad\quad\lesssim \delta\cdot\smallo(\norm{\nabla\tilde\rho}_{L^2})
    = \smallo(\norm{\sigma-\sigma'}_{H^1}\cdot \norm{\nabla\rho}_{L^2}) \comma
  \end{align*}
  where in the last estimate we have applied \cref{eq:norm_diff_sigma,eq:counterexample_thm_norm_tilderho}.
  
  To bound the second term, we apply Cauchy-Schwarz inequality:
  \begin{align*}
    &\abs*{\scalprod{U+(\lambda_1-1)\partial_\lambda U + (z_1+Re_1)\nabla_z U - U'}{\tilde\rho}_{H^1}} \\
    &\quad\quad\le 
    \norm{U+(\lambda_1-1)\partial_\lambda U + (z_1+Re_1)\nabla_z U - U'}_{H^1}\norm{\nabla\tilde\rho}_{L^2} \\
    &\quad\quad\lesssim 
    \norm{U+(\lambda_1-1)\partial_\lambda U + (z_1+Re_1)\nabla_z U - U'}_{H^1}\norm{\nabla\rho}_{L^2} \fullstop
  \end{align*}
  Recalling \cref{eq:norm_diff_sigma}, the proof of \cref{eq:last} is finished once we note that 
  \begin{equation*}
    \norm{U+(\lambda_1-1)\partial_\lambda U + (z_1+Re_1)\nabla_z U - U'}_{H^1} 
    = \smallo(\abs{\lambda_1-1} + \abs{z_1+Re_1}) \comma
  \end{equation*}
 which is a direct consequence of \cref{lem:U_differentiable_H1}.
 
 Thanks to \cref{eq:last}, we obtain that
 \begin{equation*}
   \norm{u-\sigma'}_{H^1}=\norm{\rho+(\sigma-\sigma')}_{H^1}
   \approx \norm{\rho}_{H^1}+\norm{\sigma-\sigma'}_{H^1}\ge \norm{\rho}_{H^1},
 \end{equation*}
 concluding the proof.
\end{proof}

Thanks to \cref{prop:distance_is_rho}, we are ready to show that our family
of functions constitutes a counterexample.

\begin{proof}[Proof of \cref{thm:counterexample}]
  We prove that the parametrized family $u=u_R$ that we have built satisfies all the requirements.
  
  Since $u=U+V+\rho$, it follows by \cref{eq:normrho,lem:normf} that
  \begin{equation*}
    \norm{\nabla u - \nabla U[-Re_1,1] - \nabla U[Re_1, 1]}_{L^2} = \norm{\nabla\rho}_{L^2}\to 0 
  \end{equation*}
when $R\to\infty$.
  Hence, applying \cref{prop:counterexample_main_estimate,prop:distance_is_rho} we obtain
  \begin{equation*}
    \norm{\lapl u + u\abs{u}^{p-1}}_{H^{-1}} \lesssim \zeta_n(\norm{\nabla\rho}_{L^2})
    \lesssim \zeta_n(\norm{\nabla u-\nabla\sigma'}_{L^2})
  \end{equation*}
  and this concludes the proof.
\end{proof}

\subsection{Construction of a nonnegative counterexample}
As anticipated, we show that the positive part of the counterexample $u=u_R$ constructed in the previous sections satisfies all the requirements
of \cref{thm:counterexample_pos}.

\begin{lemma}\label{lem:norm_rhominus}
  It holds
  \begin{equation*}
    \norm{\nabla u^-}_{L^2} \lesssim \xi_n(\norm{\nabla\rho}_{L^2}) \comma
  \end{equation*}
  where $\xi_n$ is the function considered in \cref{thm:counterexample_pos}.
\end{lemma}
\begin{proof}
  Applying the divergence theorem to the vector field $u^-\nabla u$, we get
  \begin{equation}\label{eq:lemma_pos_tmp1}
    \int_{\R^n} \abs{\nabla u^-}^2 = \int_{\{u<0\}} u(-\lapl u) \fullstop
  \end{equation}
  In order to proceed let us recall that $u=U+V+\rho$ and thus, if $u<0$, then $U+V<\abs{\rho}$.
  Hence in the region $\{u<0\}$ it holds $\abs{u}\le\abs{\rho}$ and $f\lesssim \abs{\rho}^p$ (recall that $f=(U+V)^p-U^p-V^p$).
  By definition of $\rho$ (see \cref{eq:def_rho}), we have
  \begin{equation*}
    -\lapl u = U^p+V^p+p(U+V)^{p-1}\rho + \tf \comma
  \end{equation*}
  so 
  \begin{equation*}
    \abs{\lapl u} \lesssim \abs{\rho}^p + \abs{f-\tf}\quad \text{in the region $\{u<0\}$} \fullstop
  \end{equation*}
Hence, using this inequality in \cref{eq:lemma_pos_tmp1}, we obtain
  \begin{align*}
    \int_{\R^n} \abs{\nabla u^-}^2 
    &\lesssim \int_{\R^n} \abs{\rho}^{2^*} + \int_{\R^n} \abs{\rho}\abs{f-\tf}
    \le \norm{\rho}_{L^{2^*}}^{2^*} + \norm{\rho}_{L^{2^*}}\norm{f-\tf}_{L^{(2^*)'}} \\
    &\lesssim \norm{\nabla \rho}_{L^2}^{2^*} + \norm{\nabla\rho}_{L^2}\norm{f-\tf}_{L^{(2^*)'}} \comma
  \end{align*}
  where we have applied H\"older and Sobolev inequalities.
  The statement now follows thanks to \cref{lem:counterexample_fundamental}.
\end{proof}

\begin{proof}[Proof of \cref{thm:counterexample_pos}]
  We prove that $u^+=(u_R)^+$ satisfies the requirements.

  Thanks to \cref{lem:norm_rhominus}, it holds
  \begin{align*}
    \abs*{\norm{\nabla u^+-\nabla U-\nabla V}_{L^2} - \norm{\nabla u-\nabla U-\nabla V}_{L^2}} &\le
    \norm{\nabla u^-}_{L^2} \lesssim
    \xi_n(\norm{\nabla\rho}_{L^2}) \\
    &= \smallo(\norm{\nabla u-\nabla U-\nabla V}_{L^2})\fullstop
  \end{align*}
  Hence \cref{eq:R infty u+} follows directly from 
  \cref{thm:counterexample}.
  
  We now show the validity of the following two key estimates:
  \begin{align}
    &\norm{\nabla u^+-\nabla\sigma'}_{L^2} \approx 
    \norm{\nabla u-\nabla\sigma'}_{L^2} \comma \label{eq:counterexample_pos1}\\
    &\norm{\lapl u^{+}+(u^+)^p}_{H^{-1}} 
    \lesssim \norm{\lapl u + u\abs{u}^{p-1}}_{H^{-1}} + \xi_n(\norm{\nabla\rho}_{L^2}) 
    \label{eq:counterexample_pos2}
    \fullstop
  \end{align}
 Note that, combining these two estimates with \cref{prop:distance_is_rho,thm:counterexample}, we have
  \begin{align*}
    \norm{\lapl u^{+}+(u^+)^p}_{H^{-1}} 
    &\lesssim \norm{\lapl u + u\abs{u}^{p-1}}_{H^{-1}} + \xi_n(\norm{\nabla\rho}_{L^2}) \\
    &\lesssim \zeta_n(\norm{\nabla u - \nabla\sigma'}_{L^2}) + \xi_n(\norm{\nabla u - \nabla\sigma'}_{L^2}) \\
    &\lesssim \xi_n(\norm{\nabla u^+-\nabla\sigma'}_{L^2})
    \comma
  \end{align*}
  that concludes the proof of \cref{thm:counterexample_pos}.
  
  Let us begin by proving \cref{eq:counterexample_pos1}. Combining the triangle inequality with
  \cref{lem:norm_rhominus,prop:distance_is_rho}, we obtain
  \begin{align*}
    \abs*{\norm{\nabla u^+-\nabla\sigma'}_{L^2}-\norm{\nabla u-\nabla\sigma'}_{L^2}}
    &\le \norm{\nabla u^+-\nabla u}_{L^2} = \norm{\nabla u^-}_{L^2} \lesssim \xi_n(\norm{\nabla\rho}_{L^2}) \\
    &\lesssim \xi_n(\norm{\nabla u-\nabla\sigma'}_{L^2})
  \end{align*}
  and \cref{eq:counterexample_pos1} follows since $\xi_n(t)/t \to 0$ when $t\to 0$.
  
  Now we focus on \cref{eq:counterexample_pos2}. The triangle and Sobolev inequalities yield
  \begin{align*}
    \abs*{\norm{\lapl u^{+}+(u^+)^p}_{H^{-1}} - \norm{\lapl u + u\abs{u}^{p-1}}_{H^{-1}}}
    &\le \norm{\nabla u^+-\nabla u}_{L^2} + 
    \norm{(u^+)^p-u\abs{u}^{p-1}}_{H^{-1}} \\
    &\lesssim \norm{\nabla u^-}_{L^2} + \norm{(u^-)^p}_{L^{(2^*)'}} \\
    &\lesssim \norm{\nabla u^-}_{L^2} + \norm{\nabla u^-}_{L^2}^p
    \lesssim \norm{\nabla u^-}_{L^2} \fullstop
  \end{align*}
  Thanks to \cref{lem:norm_rhominus}, also \cref{eq:counterexample_pos2} follows, concluding the proof.
\end{proof}

\section{Application to convergence to equilibrium for a fast diffusion equation}\label{sec:fdi}
The goal of this section is to prove a \emph{quantitative} convergence
to the equilibrium for the fast diffusion equation
\begin{equation*}\tag{FDI}\label{eq:FDI}\begin{cases}
    u(0, \emptyparam) = u_0 \\
    \frac{\de}{\de t}u = \lapl (u^{\frac1p}) 
\end{cases}\comma\end{equation*}
where $u_0:\R^n\to\co{0}\infty$ is a nonnegative initial datum.

Given $T>0$, $z\in\R^n$ and $\lambda>0$, let us define the function 
$u_{T,z,\lambda}:\co0T\times\R^n\to\oo0\infty$ as
\begin{equation}\label{eq:barenblatt}
  u_{T,z,\lambda}(t, x) 
  \defeq \left(\frac{p-1}p\right)^{\frac{p}{p-1}}(T-t)^{\frac{p}{p-1}}U[z,\lambda](x)^p \fullstop
\end{equation}
One can note that, for any choice of the parameters, the function $u_{T,z,\lambda}$, extended
to $0$ for $t\ge T$, solves \cref{eq:FDI}.

It is well-known \cite{delpinosaez2001} that for a large class of initial data 
$u(0,\emptyparam)=u_0$, the solution $u$ of the fast diffusion equation vanishes in finite time and  the 
profile of $u$ at the vanishing time coincides with the profile of one of the special solutions 
\cref{eq:barenblatt}. 
More precisely, if $T>0$ is the vanishing time of a solution $u:\oo0\infty\times\R^n\to\co0\infty$ of 
\cref{eq:FDI}, there exist $z\in\R^n$ and $\lambda>0$ such that
\begin{equation*}
  \norm*{\frac{u(t,\emptyparam)}{u_{T,z,\lambda}}-1}_{L^{\infty}} \to 0 
  \quad\text{as $t\to T^{-}$} \fullstop
\end{equation*}

We prove a quantitative version of the convergence.
The proof partially overlaps with the proof of \cite[Theorem 1.3]{cirfigmag2018}. 
However, since the latter contains an error\footnote{The first identity of \cite[Equation (3.14)]{cirfigmag2018} does 
not hold. That identity is then used to establish \cite[Equation (3.15)]{cirfigmag2018} and this generates
the error.}, 
and for the sake completeness, we report the full proof.
\begin{theorem}
  For any $n\ge 3$, let $u:\co{0}{\infty}\times\R^n\to\co{0}{\infty}$ be a solution of \cref{eq:FDI}
  with nonnegative initial datum $u_0\in L^{(2^*)'}(\R^n)$.
  Then the solution $u$ vanishes in finite time and, if $0<T=T(u_0)<\infty$ is the vanishing time,
  there exist $z\in\R^n$ and $\lambda>0$ such that
  \begin{equation*}
    \norm*{\frac{u(t,\emptyparam)}
    {u_{T,z,\lambda}(t,\emptyparam)}
    -1}_{L^{\infty}} \le A\cdot(T-t)^{\kappa(n)} \quad \forall\, 0<t<T \comma
  \end{equation*}
  where $\kappa=\kappa(n)>0$ is a dimensional constant, and $A=A(n,u_0)$ is a constant that depends
  on $n$ and the initial datum.
\end{theorem}
\begin{proof}
  With $C(n)$ we will denote any constant that depends only on the dimension $n$. 
  The value of $C(n)$ can change from line to line.
  On the contrary, the constants $A_1,A_2,\dots$ are allowed to depend also on the initial datum $u_0$.

  Applying \cite[Chapter 7, Theorem 7.10]{vazquez2006}, we can assume without loss of generality that
  \begin{equation}\label{eq:delpino}
    \frac{u(t,\emptyparam)}{u_{T,0,1}(t,\emptyparam)} 
    \stackrel{L^{\infty}(\R^n)}{\longrightarrow} 
    1 
    \quad\text{as $t\to T^{-}$} \fullstop
  \end{equation}
    Following \cite{delpinosaez2001}, let us define $w:\co0\infty\times\R^n\to\co0\infty$ as
  \begin{equation*}
    w(s, x) \defeq \left(\frac{u(t, x)}
    {\left(\frac{p-1}p\right)^{\frac{p}{p-1}}(T-t)^{\frac{p}{p-1}}}\right)^{\frac1p}\comma
  \end{equation*}
  where $t=T\left(1-\exp\bigl(-\frac{(p-1)s}p\bigr)\right)$. 
  By definition, it holds
  \begin{equation*}
    w(s,x) = \left(\frac{u(t,x)}{u_{T,0,1}(t,x)}\right)^{\frac1p}U[0,1](x)
  \end{equation*}
  and thus \cref{eq:delpino} implies
  \begin{equation}\label{eq:delpino_clean}
    w(s) \stackrel{L^{2^*}(\R^n)}{\longrightarrow} U[0,1]
    \quad\text{as $s\to\infty$}\fullstop
  \end{equation}
  Moreover, it can be shown that \cref{eq:FDI} implies that $w$ satisfies the equation
  \begin{equation}\label{eq:pde_for_w}
    \frac{\de}{\de s}(w^p) = \lapl w + w^p \fullstop
  \end{equation}
  Let us compute the time-derivative of the $L^{2^*}$-norm and $H^1$-norm of $w$:
  \begin{align}
    \frac{\de}{\de s} \int_{\R^n} w^{2^*} 
    &= \frac{2^*}p\int_{\R^n} \left(w^{2^*}-\abs{\nabla w}^2\right) \comma \label{eq:der_twostar}\\
    \frac{\de}{\de s} \int_{\R^n} \abs{\nabla w}^2 
    &= \frac2p \int_{\R^n} \left(\abs{\nabla w}^2 - \frac{(\lapl w)^2}{w^{p-1}}\right)
    = \frac{2}p\int_{\R^n}\left( w^{2^*}-\abs{\nabla w}^2 - \frac{(\lapl w + w^p)^2}{w^{p-1}} \right)
    \fullstop \label{eq:der_hone}
  \end{align}
  Hence, defining the functional $J$ as
  \begin{equation*}
    J(w) \defeq \frac12\int_{\R^n} \abs{\nabla w}^2 - \frac1{2^*}\int_{\R^n} w^{2^*} \comma
  \end{equation*}
\cref{eq:der_twostar,eq:der_hone} give
  \begin{equation*}
    \frac{\de}{\de s} J(w) = -\frac1p\int_{\R^n} \frac{(\lapl w + w^p)^2}{w^{p-1}} \fullstop
  \end{equation*}
  In particular, the quantity $s\mapsto J(w(s))$ is decreasing.
  Let us remark that, for any choice of the parameters $z\in\R^n$ and $\lambda>0$, the Talenti 
  bubble $U[z,\lambda]$ is a critical point for $J$ and it holds
  \begin{equation*}
    J(U[z, \lambda]) = S^n\left(\frac12-\frac1{2^*}\right) \fullstop
  \end{equation*}
  We now estimate $\frac{\de}{\de s} J(w)$.
  Applying H\"older's inequality we get
  \begin{equation*}
    \left(\int_{\R^n} \left(
      \frac{(\lapl w + w^p)^{\frac{2^*}p}}
      {w^{\frac{2^*(p-1)}{2p}}}
    \right)^{\frac{2p}{2^*}}\right)^{\frac{2^*}{2p}}
    \left(\int_{\R^n} \left(
      w^{\frac{2^*(p-1)}{2p}}
    \right)^{\frac{2p}{p-1}}\right)^{\frac{p-1}{2p}} 
    \ge \int_{\R^n} (\lapl w + w^p)^{\frac{2^*}{p}} \comma
  \end{equation*}
  that is equivalent to
  \begin{equation*}
    \int_{\R^n} \frac{(\lapl w + w^p)^2}{w^{p-1}}
    \ge 
    \norm{\lapl w + w^p}_{L^{(2^*)'}}^2
    \left(\int_{\R^n} w^{2^*}\right)^{-\frac2n} \fullstop
  \end{equation*}
  Hence we deduce that
  \begin{equation*}
    \frac{\de}{\de s} J(w) 
    \le -\frac1p\norm{\lapl w + w^p}_{L^{(2^*)'}}^2\left(\int_{\R^n} w^{2^*}\right)^{-\frac2n} \fullstop
  \end{equation*}
Recalling \cref{eq:delpino_clean} and defining 
  $\delta(w)\defeq\norm{\lapl w + w^p}_{L^{(2^*)'}}$, the latter inequality can be simplified to
  \begin{equation}\label{eq:time_derivative_estimate}
    \frac{\de}{\de s} J(w) 
    \le  - C(n)\delta(w)^2
  \end{equation}
  for any large enough time $s$.

  We want to show that, for any sufficiently large time $s>0$, the
  function $w(s)$ satisfies the bound on the energy necessary to apply 
  \cref{cor:single_bubble}.
  
  Thanks to \cref{eq:delpino_clean}, the quantity $J(w)$ is bounded from below at all times.
  Therefore \cref{eq:time_derivative_estimate} implies the existence of a sequence of times
  $(s_i)_{i\in\N}\subseteq\co0\infty$ such that $s_i\nearrow\infty$ and 
  $\delta(w(s_i))\to 0$ as $i\to\infty$. Moreover 
  $\norm{\nabla w(s_i)}_{L^2}$ is uniformly bounded because $J(w(s_i))$ is decreasing.
  Thus we can apply \cite[Prop. 2.1]{struwe1984} to deduce that the convergence 
  \cref{eq:delpino_clean} can be upgraded to the $H^1$-convergence
  \begin{equation*}
    w(s_i) \stackrel{H^1(\R^n)}{\longrightarrow} U[0,1] 
    \quad\text{as $i\to\infty$}\fullstop
  \end{equation*}
  Since $J(w(s))$ is decreasing, we have
  \begin{equation}\label{eq:functional_limit}
    \lim_{s\to\infty}J(w(s))
    = \lim_{i\to\infty}J(w(s_i)) = J(U[0,1]) \comma
  \end{equation}
  that together with \cref{eq:delpino_clean} implies  that  
  \begin{equation}\label{eq:energy_estimate}
    \int_{\R^n}\abs{\nabla w(s,x)}^2 \de x\to \int_{\R^n} \abs{\nabla U[0,1]}^2 = S^n
    \quad\text{as $s\to\infty$}\fullstop
  \end{equation}
  
  Our next goal is showing that $J(w)-J(U[0,1])$ is bounded from above by $\delta(w)^2$.
  Thanks to \cref{eq:energy_estimate}, we can apply \cref{cor:single_bubble}  to deduce that, for any 
  sufficiently large $s>0$, there is a decomposition $w=W + \rho$, where $W$ is a Talenti 
  bubble (that depends on the time $s$) and $\rho$ satisfies
  \begin{equation}\label{eq:fde_tmp1}
    \int_{\R^n}\abs{\nabla\rho}^2 \le C(n)\delta(w)^2 \fullstop
  \end{equation}
  Substituting $w=W + \rho$ in $J(w)$ and noticing that $(W+\rho)^{2^*}-W^{2^*}-2^*W^p\rho\geq 0$ (as a consequence of the positivity of $w$ and Bernoulli's inequality)\footnote{The fact that $\int [(W+\rho)^{2^*}-W^{2^*}-2^*W^p\rho] \ge 0$ is crucial to fix the issue in the proof of \cite[Theorem 1.3]{cirfigmag2018}.},  
using \cref{eq:fde_tmp1} we get
  \begin{equation}\label{eq:talenti_criticality}\begin{aligned}
    J(w) 
    &= J(W) 
    + \frac12\int_{\R^n}\abs{\nabla\rho}^2+\int_{\R^n}\nabla\rho\cdot\nabla W
    -\frac1{2^*}\int_{\R^n}\left( (W+\rho)^{2^*}-W^{2^*}\right) \\
    &= J(U[0,1]) + \frac12\int_{\R^n}\abs{\nabla\rho}^2
    -\frac1{2^*}\int_{\R^n}\left((W+\rho)^{2^*}-W^{2^*}-2^*W^p\rho \right)\\
    &\le J(U[0,1])+C(n)\delta(w)^2 \fullstop
  \end{aligned}\end{equation}
  Let us emphasize that this latter inequality is the central point of the whole proof, as
  it encodes the criticality of Talenti bubbles for the functional $J$. 
  Joining the inequalities \cref{eq:time_derivative_estimate,eq:talenti_criticality}, we obtain
  \begin{equation}\label{eq:main_inequality}
    \frac{\de}{\de s}\left( J(w) - J(U[0,1])\right) \le -C(n)\delta(w)^2 \le
    -C(n)\left(J(w) - J(U[0,1])\right)
  \end{equation}
  for any sufficiently large $s$.
  
Since $J(w(s))$ is decreasing in $s$, \cref{eq:functional_limit} tells us that $J(w(s))-J(U[0,1])\ge 0$ for all $s\ge 0$.
  Hence, \cref{eq:main_inequality} implies the existence of a constant $A_1>0$ such that, for any $s>0$,
  \begin{equation*}
    0 \le J(w(s))-J(U[0,1]) \le A_1 e^{-C(n)s} \fullstop
  \end{equation*}
  This exponential decay together with \cref{eq:main_inequality} implies the fundamental bound
  \begin{equation*}
    \int_s^{\infty}\delta(w)^2 \le A_2 e^{-C(n)s}
  \end{equation*}
  and from this, splitting the integral on fixed-length intervals and applying Cauchy-Schwarz inequality, 
  we can deduce
  \begin{align*}
    \int_s^{\infty}\delta(w)=\sum_{k=0}^\infty \int_{s+k}^{s+k+1}\delta(w)
    \leq \sum_{k=0}^\infty \left(\int_{s+k}^{s+k+1}\delta(w)^2\right)^{1/2}
    \leq \sum_{k=0}^\infty\left(A_2e^{-C(n)[s+k]}\right)^{1/2}
    \le A_3 e^{-C(n)s} \fullstop
  \end{align*}
  Recalling \cref{eq:pde_for_w}, for any $t>s>0$ we can write pointwise
  \begin{equation*}
    w(t)^p-w(s)^p = \int_s^t \left(\lapl w + w^p\right)\comma 
  \end{equation*}
  and taking the $L^{(2^*)'}$-norm of both sides we get
  \begin{align*}
    \norm{w(t)-w(s)}_{L^{2^*}} \le A_4\norm{w(t)^p-w(s)^p}_{L^{(2^*)'}} 
    \le
    A_4\int_s^t\norm{\lapl w + w^p}_{L^{(2^*)'}} 
    \le A_4\int_s^\infty \delta(w) 
    \le A_5e^{-C(n)s} \fullstop
  \end{align*}
Letting $t\to \infty$, this implies that the convergence stated in \cref{eq:delpino_clean} is exponential:
  \begin{equation}\label{eq:fdi_final}
    \norm{w(s)-U[0,1]}_{L^{2^*}} \le A_6e^{-C(n)s} \fullstop
  \end{equation}
  Finally, let $F:\R^n\to\S^n$ be the stereographic projection 
  $F(x)\defeq\left(\frac{2x}{1+\abs{x}^2},\frac{\abs{x}^2-1}{1+\abs{x}^2}\right)$ and let 
  $v:\co0\infty\times\S^n\to\co0\infty$ be the function defined as
  \begin{equation*}
    v(s, F(x)) \defeq \frac{w(s,x)}{U[0,1](x)} \fullstop
  \end{equation*}
  Since the Jacobian of $F$ satisfies
  \begin{equation*}
    \det(\de F)(x) = \left(\frac{2}{1+x^2}\right)^n = C(n)\cdot U[0,1](x)^{2^*}\comma
  \end{equation*}
  the estimates \cref{eq:fdi_final} becomes
  \begin{equation}\label{eq:fdi_final2}
    \norm{v(s)-1}_{L^{2^*}(\S^n)} \le A_7e^{-C(n)s} \fullstop
  \end{equation}
  Noticing that
  \begin{equation*}
    v(s, F(x)) = \left(\frac{u(t,x)}{u_{T,0,1}(t,x)}\right)^{\frac1p}\comma
  \end{equation*}
  thanks to \cref{eq:fdi_final2} the proof would be concluded if we were able to show that
  \begin{equation*}
    \norm{v(s)-1}_{L^{\infty}(\S^n)} \le A_8 \norm{v(s)-1}_{L^{2^*}(\S^n)}^{C(n)}
  \end{equation*}
  for any sufficiently large $s$.
  This latter inequality follows directly from interpolation inequalities, noticing that $v(s)$ is uniformly 
  Lipschitz for $s \gg1$ as a consequence of \cite[(4.2) and Proposition 5.1]{delpinosaez2001}.
\end{proof}

\appendix

\section{Spectral properties of the weighted Laplacian}\label{app:spectrum}
Let $\w\in L^{\frac n2}(\R^n)$ be a positive weight. 
Our goal is to study the properties of the spectrum of the operator $\frac{-\lapl}\w$ and of its eigenfunctions.
The spectrum of such a weighted Laplacian is thoroughly studied in literature, see for example \cite{allegretto1992}. An exhaustive list of related references is contained 
in the introduction of \cite{szulkin1998}.

After some first basic statements that tell us that the spectrum of the operator is discrete, 
we move our attention to the properties of the eigenfunctions. 
For general weights we prove that \emph{almost}-eigenfunctions are close to true eigenfunctions and that
the eigenfunctions obey a concentration property. 

Finally we focus on weights that decay at infinity as $(1+\abs{x})^{-4}$. In this situation we are able to
show some finer integrability properties of the eigenfunctions, and deduce that the restriction of an 
eigenfunction is close to an eigenfunction for the restriction of the weight.

\subsection{Results valid for any \texorpdfstring{$\w\in L^{\frac n2}(\R^n)$}{n/2-integrable weight}}
Let us begin with a technical, albeit important, proposition that gives us a compact embedding (in the style
of the classical Sobolev embedding) from a homogeneous Sobolev space into a weighted space.
\begin{proposition}[Compact embedding in weighted space]\label{prop:weighted_cpt_embedding}
  Given a positive integer $n\in\N$, let $1\le p < n$ and $q < p^*$ be two real numbers. 
  For any positive weight $\w\in L^{(\frac{p^*}q)'}(\R^n)$, the following compact embedding holds:
  \begin{equation*}
    \dot W^{1,p}(\R^n) \embedding[cpt] L^q_\w(\R^n) \fullstop
  \end{equation*}
\end{proposition}
\begin{proof}
  Let us fix a real number $R$, and denote by $B_R=B(0,R)$ the ball of radius $R$ centered at the origin.
    Thanks to the chain of embeddings
  \begin{equation*}
    \dot W^{1,p}(\R^n) \embedding L^{p^*}(\R^n) \embedding L^{p^*}(B_R) \embedding L^p(B_R) \comma
  \end{equation*}
  it follows that
  \begin{equation*}
    \dot W^{1,p}(\R^n) \embedding W^{1,p}(B_R)
  \end{equation*}
  and therefore, applying the Rellich-Kondrakov theorem, it holds
  \begin{equation}\label{eq:weight_embedding1}
    \dot W^{1,p}(\R^n) \embedding[cpt] L^q(B_R) \fullstop
  \end{equation}
  Let us define $\w_R:\R^n\to\R$ as
  \begin{equation*}
    \w_R(x) \defeq
    \begin{cases}
      \w(x) \quad&\text{if $\w(x) \le R$,} \\
      0    \quad&\text{otherwise} \fullstop
    \end{cases}
  \end{equation*}
  Since $\w_R\in L^\infty(\R^n)$, it holds
  \begin{equation*}
    L^q(B_R) \embedding L^q_{\w_R}(B_R) \comma
  \end{equation*}
  and therefore \cref{eq:weight_embedding1} implies
  \begin{equation}\label{eq:weight_embedding2}
    \dot W^{1,p}(\R^n) \embedding[cpt] L^q_{\w_R}(B_R) \fullstop
  \end{equation}
  Let us remark that H\"older and Sobolev inequalities  imply that, for any 
  $g\in \dot W^{1,p}(\R^n)$ and any Borel set $E\subseteq \R^n$, it holds
  \begin{equation}\label{eq:weight_embedding3}
    \norm{g}_{L^q_\w(E)} \le C\norm{\nabla g}_{L^p(\R^n)}^q\norm{\w}_{L^\alpha(E)}
  \end{equation}
  where $\alpha \defeq (\frac{p^*}{q})'$ and $C=C(n, p, q)$ is a constant. 
  
  Let us now fix a bounded sequence $(f_k)_{k\in\N}\subseteq \dot W^{1,p}(\R^n)$.
  Up to extracting a subsequence, thanks to \cref{eq:weight_embedding2}, by a diagonal argument we can find a 
  function $f\in \dot W^{1, p}(\R^n)$ such that for any $R > 0$ it holds $f_k\to f$ in the 
  $L^q_{\w_R}(B_R)$-norm.
  We want to prove that $f_k\to f$ in the stronger $L^q_\w(\R^n)$-norm.
  
  For a fixed $R>0$, recalling \cref{eq:weight_embedding3}, we have
  \begin{align*}
    \limsup_{k\to\infty}&\norm{f_k-f}^q_{L^q_\w(\R^n)} 
    = 
    \limsup_{k\to\infty}
    \norm{f_k-f}^q_{L^q_{\w_R}(B_R)} + \norm{f_k-f}^q_{L^q_{\w_R}(B_R^\complement)}
    + \norm{f_k-f}^q_{L^q_\w(\{\w>R\})} \\
    &\le 
    \limsup_{k\to\infty}
    C\cdot\norm{\nabla f_k-\nabla f}_{L^p(\R^n)}^q \left(
    \norm{\w}_{L^\alpha(B_R^\complement)}
    +\norm{\w}_{L^\alpha(\{\w>R\})}\right) \\
    &\le 
    2^qC\left(\norm{\nabla f}_{L^p(\R^n)}^q + \sup_{k\in\N}\norm{\nabla f_k}_{L^p(\R^n)}^q\right)
    \left(\norm{\w}_{L^\alpha(B_R^\complement)} +\norm{\w}_{L^\alpha(\{\w>R\})}\right) \fullstop
  \end{align*}
  The desired convergence now follows sending $R$ to infinity.
\end{proof}

We can now state and prove the main theorem of this section.
\begin{theorem}\label{thm:inverse_well_defined}
  For any $n\ge 3$ and any positive weight $\w\in L^{\frac n2}(\R^n)$, the inverse operator 
  $\left(\frac{-\lapl}{\w}\right)^{-1}$ is well-defined and continuous from $L^2_\w(\R^n)$
  into $H^1(\R^n)$. 
  Hence, thanks to \cref{prop:weighted_cpt_embedding}, it is a compact self-adjoint operator 
  from $L^2_\w(\R^n)$ into itself.
\end{theorem}
\begin{proof}
  Let $\varphi\in H^1(\R^n)$ and $f\in L^2_\w(\R^n)$. Applying H\"older and Sobolev inequalities, we obtain
  \begin{align*}
    \scalprod{f}{\varphi}_{L^2_\w} = \int_{\R^n}f\varphi \,\w
    \le \left(\int_{\R^n} f^2\w\right)^{\frac12}
    \left(\int_{\R^n} \w^{\frac n2}\right)^{\frac1n}
    \left(\int_{\R^n} \abs{\varphi}^{2^*}\right)^{\frac1{2^*}}
    \lesssim \norm{f}_{L^2_\w}\norm{\w}_{L^{\frac n2}}^{\frac12}\norm{\varphi}_{H^1} \fullstop
  \end{align*}
  As a consequence, the map 
  \begin{equation*}
    L^2_\w(\R^n)\ni f \mapsto \scalprod{f}{\emptyparam}_{L^2_\w}\in (H^1)'
  \end{equation*}
  is continuous and injective. 
  Applying Riesz Theorem, it follows that there exists a unique continuous linear map 
  $T:L^2_\w(\R^n)\to H^1(\R^n)$ such that for any $f\in L^2_\w(\R^n)$ and any $g\in H^1(\R^n)$ 
  it holds
  \begin{equation*}
    \int_{\R^n} fg\,\w = \int_{\R^n} \nabla T(f)\cdot \nabla g 
    = \int_{\R^n} -\lapl T(f)\, g \implies -\lapl T(f) = f\w \fullstop
  \end{equation*}
  Thus $T=\left(\frac{-\lapl}\w\right)^{-1}$ and the statement is proven.
\end{proof}
\begin{remark}
  From now on we will use implicitly the following useful identity:
  \begin{equation*}
    \left(\frac{-\lapl}{\w}\right)^{-1}\Bigl(\frac f\w\Bigr)=(-\lapl)^{-1}f \fullstop
  \end{equation*}
\end{remark}

Since we have shown that $\left(\frac{-\lapl}{\w}\right)^{-1}$ is compact and self-adjoint, we know that its
spectrum is discrete. From now on we move our attention to the structure of its eigenfunctions.
We begin by showing that if a function is \emph{almost} an eigenfunction, than it close to a true 
eigenfunction.
\begin{lemma}[Approximate eigenfunction]\label{lem:approximation_eigenfunctions}
  Let us fix $n\ge 3$ and a positive weight $\w\in L^{\frac n2}(\R^n)$.
  Let $\psi\in L^2_\w(\R^n)$ be such that $-\lapl \psi - \lambda \w \psi = f$ for some $\lambda > 0$ and 
  $f:\R^n\to\R$. If $\psi = \sum \alpha_k\psi_k$ where $(\lambda_k^{-1},\psi_k)_{k\in\N}$ is the 
  sequence of eigenvalues and normalized eigenfunctions for $\left(\frac{-\lapl}{\w}\right)^{-1}$, 
  then it holds
  \begin{equation*}
    \sum_k \alpha_k^2\left(1-\frac{\lambda}{\lambda_k}\right)^2 = \norm{\lapl^{-1} f}^2_{L^2_\w} \fullstop
  \end{equation*}
\end{lemma}
\begin{proof}
  Substituting $\psi = \sum \alpha_k\psi_k$ into $-\lapl \psi - \lambda \w \psi = f$ yields
  \begin{equation*}
    \sum_k \alpha_k(\lambda_k-\lambda)\psi_k = \frac{f}{\w} \fullstop
  \end{equation*}
  Applying $\left(\frac{-\lapl}{\w}\right)^{-1}$ to both sides yields
  \begin{equation*}
    \sum_k \alpha_k\left(1-\frac{\lambda}{\lambda_k}\right)\psi_k = (-\lapl)^{-1}f \fullstop
  \end{equation*}
  and the desired inequality follows taking the $L^2_\w$-norm.
\end{proof}

The following lemma ensures that, as soon as we assume natural conditions on the spectral decomposition of $f$, 
if $-\lapl u-\lambda \w u=f$ then we can control the $H^1$-norm of $u$ with the $H^{-1}$-norm of $f$.
\begin{lemma}\label{lem:H1estimate}
  Let us fix $n\ge 3$ and a positive weight $\w\in L^{\frac n2}(\R^n)$.
  Let $u\in L^2_\w(\R^n)$ and $f\in H^{-1}(\R^n)$ be such that $-\lapl u-\lambda \w u = f$ for some 
  $\lambda > 0$.
  Let $u=\sum \alpha_k \psi_k$, where $(\lambda_k^{-1}, \psi_k)$ is the sequence
  of eigenvalues and normalized eigenfunctions of $\left(\frac{-\lapl}{\w}\right)^{-1}$, and assume that 
  whenever $\alpha_k\not=0$ it holds $\lambda_k\ge \lambda(1-\eps)^{-1}$ for some $\eps \in \oo01$.
Then
  \begin{equation*}
    \eps\norm{\nabla u}_{L^2} \le \norm{f}_{H^{-1}} \le \norm{\nabla u}_{L^2} \fullstop
  \end{equation*}
\end{lemma}
\begin{proof}
  First of all, let us check that the family of functions $\lambda_k^{-\frac12}\psi_k$ is a complete 
  orthonormal basis of $H^1(\R^n)$. For any $i,j\in\N$, it holds
  \begin{equation*}
    \scalprod{\psi_i}{\psi_j}_{H^1} = \int_{\R^n} \nabla\psi_i \cdot \nabla\psi_j 
    = \int_{\R^n} -\lapl\psi_i\,\psi_j
    =\lambda_i\int_{\R^n} \psi_i\psi_j\, \w = \scalprod{\psi_i}{\psi_j}_{L^2_\w}\comma
  \end{equation*}
  and thus the desired $H^1$-orthonormality follows from the orthonormality of $(\psi_k)$ with respect to the
  $L^2_\w$-scalar product. A similar computation shows that this basis is also complete in $H^1(\R^n)$.
  
  As shown in the proof of \cref{lem:approximation_eigenfunctions}, it holds
  \begin{equation*}
    (-\lapl)^{-1} f = \sum_k \alpha_k\left(1-\frac{\lambda}{\lambda_k}\right)\psi_k
  \end{equation*}
  and thus, since $\lambda_k^{-\frac12}\psi_k$ is an orthonormal basis of $H^1(\R^n)$, we deduce that
  \begin{equation*}
    \norm{f}^2_{H^{-1}} = \norm{\nabla \lapl^{-1} f}^2_{L^2} 
    = \sum_k \alpha_k^2\left(1-\frac{\lambda}{\lambda_k}\right)^2\lambda_k \fullstop
  \end{equation*}
  Similarly, it holds
  \begin{equation*}
    \norm{\nabla u}^2_{L^2} = \sum_k \alpha_k^2\lambda_k
  \end{equation*}
  and thus the desired two-sided estimate follows directly from the assumption 
  that $\lambda_k\ge \lambda(1-\eps)^{-1}$ whenever $\alpha_k \neq 0$.
\end{proof}

It is natural to expect that the eigenfunctions are concentrated in the zone where the weight itself is
concentrated. The following lemma shows exactly this.
\begin{lemma}[Concentration of eigenfunctions]\label{lem:concentration_eigenfunctions}
  Let us fix $n\ge 3$ and a positive weight $\w\in L^{\frac n2}(\R^n)$.
  Let $\lambda>0$ and $\psi\in L^2_\w(\R^n)$ be such that $-\lapl\psi = \lambda \w\psi$.
  For any measurable set $E\subseteq\R^n$ it holds
  \begin{equation*}
    \int_E \psi^2\w \le C(n)\lambda \norm{\w}_{L^{\frac n2}(E)} \int_{\R^n} \psi^2 \w \fullstop
  \end{equation*}
\end{lemma}
\begin{proof}
  H\"older's inequality gives us
  \begin{equation*}
    \int_E \psi^2\w \le \norm{\psi}_{L^{2^*}(E)}^2\norm{\w}_{L^{\frac n2}(E)} \fullstop
  \end{equation*}
 Also, the fact that $\psi$ is an eigenfunction relative to $\lambda$ together with the Sobolev 
  inequality tell us
  \begin{equation*}
    \norm{\psi}_{L^{2^*}(E)}^2 \le C\int_{\R^n}\abs{\nabla \psi}^2 = \lambda C\int_{\R^n} \psi^2\w \comma
  \end{equation*}
  where $C=C(n)$ is a constant that depends only on the dimension.
  
  Joining the two inequalities yields the desired estimate.
\end{proof}

\subsection{Further results when \texorpdfstring{$\w\approx(1+\lvert x\rvert)^{-4}$}{w decays like the inverse of the 4th power of the distance}}
If we assume more on the weight, namely that $\w\approx(1+\abs{x})^{-4}$, we can obtain some better 
integrability of the eigenfunctions and 
we can prove that, in a certain sense, the restriction of eigenfunctions for the weight $\w$ are 
eigenfunctions for the restriction of the weight.

In this section it is crucial that $\w\approx(1+\abs{x})^{-4}$, as we rely on 
\cref{thm:rellich_inequality}. 
This result is known in literature as Rellich's inequality (see \cite[Chapter 6]{balinsky2015}).
A short proof of the mentioned inequality with the sharp constant can be found for example in 
\cite{machihara2017}.
\begin{theorem}[Rellich's Inequality]\label{thm:rellich_inequality}
  If $n\ge 5$ it holds
  \begin{equation*}
    \int_{\R^n} f^2 \abs{x}^{-4} \de x \le C_n \int_{\R^n} \abs{\lapl u}^2 \de x
  \end{equation*}
  for any $f\in H^2(\R^n)$, where $H^2(\R^n)$ is the space of functions with Laplacian in 
  $L^2(\R^n)$.
\end{theorem}
\begin{corollary}\label{cor:rellich2}
  Let $\w\in L^{\frac n2}(\R^n)$ be a positive weight such that $\w\le c\sum (1+\abs{x-x_i})^{-4}$ for a 
  constant $c>0$ and some points $x_1,\dots,x_k\in \R^n$. If $n\ge 5$, it holds
  \begin{equation*}
    \norm{\lapl^{-1}f}_{L^2_\w} \lesssim \sqrt{ck} \norm{f}_{L^2}
  \end{equation*}
  for any function $f\in L^2(\R^n)$.
\end{corollary}
\begin{proof}
  The statement is a direct consequence of \cref{thm:rellich_inequality}.
\end{proof}

We can now prove that the eigenfunctions of $\frac{-\lapl}{\w}$ are in $L^2(\R^n)$. 
A priori we know only that they belong to $L^2_\w(\R^n)$.
The proof is achieved by duality applying \cref{cor:rellich2}.
\begin{lemma}\label{lem:eigenfunctions_integrability}
  Let $\w\in L^{\frac n2}(\R^n)$ be a positive weight such that $\w\le c\sum (1+\abs{x-x_i})^{-4}$ for a
  constant $c>0$ and some points $x_1,\dots, x_k\in \R^n$. If $n\ge 5$ and $\psi\in L^2_\w(\R^n)$ is such that
  $-\lapl\psi = \lambda \w\psi$ for some $\lambda>0$, then
  \begin{equation*}
    \norm{\psi}_{L^2} \le \lambda \sqrt{ck}\norm{\psi}_{L^2_\w} \fullstop
  \end{equation*}
\end{lemma}
\begin{proof}
  Let us fix a test function $\varphi\in C^{\infty}_c(\R^n)$. It holds
  \begin{align*}
    \int_{\R^n}\psi\varphi = \int_{\R^n}(-\lapl\psi)((-\lapl)^{-1}\varphi)
    = \lambda\int_{\R^n} \psi \w (-\lapl)^{-1}\varphi 
    \le \lambda\norm{\psi}_{L^2_\w}\norm{(-\lapl)^{-1}\varphi}_{L^2_\w} 
    \lesssim \lambda \norm{\psi}_{L^2_\w}\sqrt{ck}\norm{\varphi}_{L^2} \comma
  \end{align*}
  where in the last step we applied \cref{cor:rellich2}.
  The statement follows by taking the supremum over all functions $\varphi$ with $\norm{\varphi}_{L^2}\leq 1$.
\end{proof}

Let us conclude our study of the spectral properties of $\left(\frac{-\lapl}{\w}\right)^{-1}$ with the 
following intuitive proposition. 
It states that, under suitable assumptions, if we restrict an eigenfunction relative to the weight 
$\w$ to a zone where $\w$ is almost the same as $\w_1$, then the restriction is almost an eigenfunction 
for the weight $\w_1$.
\begin{proposition}\label{prop:restriction_eigenfunction}
  For a fixed $n\ge 5$, let $\w_1, \w\in L^{\frac n2}(\R^n)$ be two positive weights such that 
  $c^{-1}(1+\abs{x-\bar x}^{-4})\le \w$ and $\w_1 \le c\abs{x-\bar x}^{-4}$ for some constant 
  $c>0$ and $\bar x\in\R^n$.
  Let $\psi\in L^2_\w(\R^n)$ and $\lambda>0$ be  such that $-\lapl\psi-\lambda \w\psi = 0$ and 
  $\int_{\R^n} \psi^2\w=1$.
  
  Let us fix an arbitrary smooth function $\eta:\R\to\cc01$ such that $\eta(t)=1$ if $t\le 1$ and $\eta(t)=0$
  if $t\ge 2$. We will consider $\abs{\eta'}_\infty,\abs{\eta''}_\infty$ universal constants (in particular 
  they can be hidden in the $\lesssim$ notation).
  
  For a fixed radius $R>1$, denote $\varphi(x)\defeq \psi(x)\,\eta\left(\frac{\abs{x-\bar x}}{R}\right)$, and 
  write $\varphi=\sum \alpha_k\varphi_k$ where $(\lambda_k^{-1},\varphi_k)$ is the sequence of eigenvalues for
  $\left(\frac{-\lapl}{\w_1}\right)^{-1}$. Then it holds
  \begin{equation*}
    \sum_k \alpha_k^2\left(\frac{\lambda}{\lambda_k}-1\right)^2 \lesssim 
    c\lambda \left(R^{-2}+c\norm{\w}_{L^{\frac n2}(B_{2R}(\bar x)\setminus B_R(\bar x))} 
    + \lambda\sup_{B_{2R}(\bar x)} \frac{\abs{\w_1-\w}^2}{\w}\right)\fullstop
  \end{equation*}
\end{proposition}
\begin{proof}
  Without loss of generality we can assume $\bar x = 0$.
  For notational simplicity, we denote $\eta_R(x) \defeq \eta\left(\frac{\abs{x}}{R}\right)$.
  
  Applying \cref{lem:approximation_eigenfunctions,cor:rellich2} we deduce
  \begin{equation}\label{eq:tech_prop1}
    \sum_k \alpha_k^2\left(\frac{\lambda}{\lambda_k}-1\right)^2 
    \lesssim c\norm{-\lapl \varphi-\lambda \w_1\varphi}^2_{L^2}\fullstop
  \end{equation}
  Let us expand $-\lapl \varphi-\lambda \w_1\varphi$ as follows:
  \begin{equation}\label{eq:tech_prop2}\begin{aligned}
    -\lapl \varphi-\lambda \w_1\varphi 
    &= (-\lapl\psi-\lambda \w\psi)\eta_R 
    - 2\nabla\psi\cdot\nabla\eta_R - \psi\lapl\eta_R - \lambda(\w_1-\w)\psi\eta_R \\
    &= - 2\nabla\psi\cdot\nabla\eta_R - \psi\lapl\eta_R - \lambda(\w_1-\w)\psi\eta_R\fullstop
  \end{aligned}\end{equation}
  Since $\abs{\eta'}_\infty\lesssim R^{-1}$ and $\abs{\eta''}_\infty\lesssim R^{-2}$,
  \cref{eq:tech_prop1,eq:tech_prop2} imply
  \begin{equation}\label{eq:tech_prop3}
    \sum_k \alpha_k^2\left(\frac{\lambda}{\lambda_k}-1\right)^2 
    \lesssim
    c\left(
    R^{-2}\int_{\R^n}\abs{\nabla\psi}^2 
    + R^{-4}\int_{B_{2R}\setminus B_R}\psi^2
    +\lambda^2\int_{B_{2R}}\abs{\w_1-\w}^2\psi^2 
    \right)\fullstop
  \end{equation}
  Given that $\psi$ is an eigenfunction with unit $L^2_\w$-norm, it holds $\int_{\R^n}\abs{\nabla\psi}^2=\lambda$.
  Recalling that $c^{-1}(1+\abs{x})^{-4}\le \w$, a direct application of \cref{lem:concentration_eigenfunctions} yields
  \begin{equation*}
    R^{-4}\int_{B_{2R}\setminus B_R}\psi^2 \lesssim c\int_{B_{2R}\setminus B_R}\psi^2 \w 
    \lesssim c\lambda\norm{\w}_{L^{\frac n2}(B_{2R}\setminus B_R)} \fullstop
  \end{equation*}
  Finally the last term is estimated as
  \begin{equation*}
    \int_{B_{2R}}\abs{\w_1-\w}^2\psi^2 
    \le \sup_{B_{2R}}\frac{\abs{\w_1-\w}^2}{\w} \int_{B_{2R}}\psi^2\w
    \le \sup_{B_{2R}}\frac{\abs{\w_1-\w}^2}{\w} \fullstop
  \end{equation*}
  Substituting all the mentioned estimates into \cref{eq:tech_prop3} finishes the proof.
\end{proof}

\section{Integrals involving two Talenti bubbles}\label{app:computations}
This appendix is devoted to the computations of integral quantities involving two Talenti bubbles.
First we deal with the case of two general Talenti bubbles $U$ and $V$, and integrals of the form 
$\int_{\R^n} U^\alpha V^\beta$ with $\alpha+\beta=2^*$.
Then we consider the special case of two bubbles $U=U[-Re_1,1]$ and $V=U[Re_1,1]$ for a large $R>1$ and we 
obtain a simple formula to compute very general integrals of functions that depend only on $U$ and $V$.

Even though the estimate \cref{prop:interaction_approx} shares the same spirit of the estimates 
\cite[F7--F21]{bahri1989}, our estimate is, as far as we can tell, not implied by the inequalities stated 
in \cite{bahri1989}.

All the proofs in this appendix exploit the same simple strategy: splitting the involved integrals in regions
where the integrand has a power-like behavior and then computing the integrals explicitly.

\begin{lemma}\label{lem:interaction_for_simple_bubbles}
  Given $n\ge 3$, let us fix $\alpha+\beta=2^*$ with $\alpha,\beta\geq 0$, $\lambda \in \oc01$, and $z\in\R^n$.
  Set $D\defeq \abs{z}$.
  
  If $\abs{\alpha-\beta}\ge\eps$ for some $\eps>0$, then
  \begin{equation*}
    \int_{\R^n} U[0,1]^\alpha U[z,\lambda]^\beta \approx_{n,\eps} 
    \begin{cases}
      \left(\frac{1}{\lambda D^2}\right)^{\frac{(n-2)\min(\alpha,\beta)}2} &\text{ if } D\lambda\ge 1 \comma\\
      \lambda^{\frac{(n-2)\min(\alpha,\beta)}2} &\text{ if } D\lambda\le 1 \fullstop\\
    \end{cases}
  \end{equation*}
  If instead $\alpha=\beta=\frac{2^*}2$, then
  \begin{equation*}
    \int_{\R^n} U[0,1]^\alpha U[z,\lambda]^\beta \approx_{n} 
    \begin{cases}
      \left(\frac{1}{\lambda D^2}\right)^{\frac n2}\log(\lambda^2 D) &\text{ if } D\lambda\ge 1 \comma\\
      \lambda^{\frac n2}\log(\lambda^{-1}) &\text{ if } D\lambda\le 1 \fullstop\\
    \end{cases}
  \end{equation*}
\end{lemma}
\begin{proof}
  We split the proof in two cases.
  
  \vspace{0.4em}
  \noindent\textit{The case $D\le \lambda^{-1}$.}
  \begin{figure}[htb]
    \centering
    \footnotesize

\begin{tikzpicture}[scale=0.4]

\pgfdeclarepatternformonly{north east lines wide2}%
   {\pgfqpoint{-1pt}{-1pt}}%
   {\pgfqpoint{20pt}{20pt}}%
   {\pgfqpoint{19pt}{19pt}}%
   {
     \pgfsetlinewidth{0.2pt}
     \pgfpathmoveto{\pgfqpoint{0pt}{0pt}}
     \pgfpathlineto{\pgfqpoint{19.1pt}{19.1pt}}
     \pgfusepath{stroke}
    }

\pgfmathsetmacro{\Zx}{1.7}
\pgfmathsetmacro{\Zy}{1}
\coordinate (O) at (0, 0);
\coordinate (Z) at (\Zx, \Zy);

\pgfmathsetmacro{\D}{{veclen(\Zx, \Zy)}}
\pgfmathsetmacro{\L}{3.5}

\fill[pattern=north east lines wide2] (-2*\L,-2*\L) rectangle ({2*\L},{2*\L});
\draw[fill=white!95!black] (O) circle ({2*\L});
\draw (O) circle (1);
\draw (Z) circle (\L);

\draw[fill=black] (O) circle [radius=1pt] node[above]{$0$};
\draw[fill=black] (Z) circle [radius=1pt] node[above]{$z$};

\pgfmathsetmacro{\T}{160}
\pgfmathsetmacro{\Tx}{cos(\T)*2*\L}
\pgfmathsetmacro{\Ty}{sin(\T)*2*\L}
\draw[-,dashed] (O) -- (\Tx,\Ty) node[midway,above] {$2\lambda^{-1}$};

\pgfmathsetmacro{\T}{250}
\pgfmathsetmacro{\Tx}{cos(\T)}
\pgfmathsetmacro{\Ty}{sin(\T)}
\draw[-,dashed] (O) -- (\Tx,\Ty) node[midway,right] {$1$};

\pgfmathsetmacro{\T}{280}
\pgfmathsetmacro{\Tx}{\Zx + cos(\T)*\L}
\pgfmathsetmacro{\Ty}{\Zy + sin(\T)*\L}
\draw[-,dashed] (Z) -- (\Tx,\Ty) node[midway,right] {$\lambda^{-1}$};

\node[draw,right,fill=white] at (-2*\L, 2*\L) {$D\le \lambda^{-1}$};

\end{tikzpicture}

    \caption{The balls $B(0,1)$, $B(z,\lambda^{-1})$, and $B(0,2\lambda^{-1})$, involved in the proof of the case $D\le\lambda^{-1}$.}
  \end{figure}
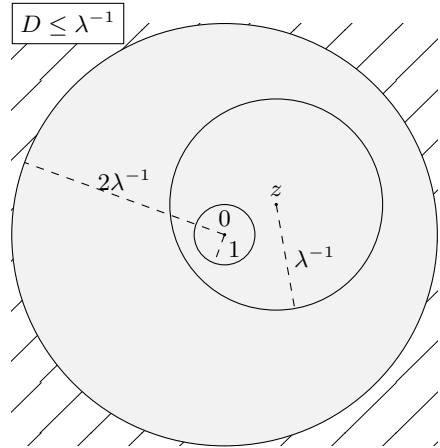
  In $B(0,2\lambda^{-1})$ it holds $U[z,\lambda]\approx U[z,\lambda](z) \approx \lambda^{\frac{n-2}2}$,
  while in $B(0,2\lambda^{-1})^{\complement}$ it holds $U[0,1]\approx \abs{x}^{-(n-2)}$ and 
  $U[z,\lambda]\approx \lambda^{-\frac{n-2}2}\abs{x}^{-(n-2)}$. Thus, recalling that $\alpha+\beta=2^*$, we get
  \begin{align*}
    &\int_{\R^n}U[0,1]^\alpha U[z, \lambda]^\beta = 
    \int_{B(0,2\lambda^{-1})}U[0,1]^\alpha U[z, \lambda]^\beta 
    + \int_{B(0,2\lambda^{-1})^{\complement}}U[0,1]^\alpha U[z, \lambda]^\beta \\
    &\quad\quad\quad\approx
    \int_0^{2\lambda^{-1}} (1+t^2)^{-\alpha\frac{n-2}2}\lambda^{\beta\frac{n-2}2} t^{n-1}\de t
    +
    \int_{2\lambda^{-1}}^{\infty} t^{-\alpha(n-2)}\lambda^{-\beta\frac{n-2}2}t^{-\beta(n-2)} t^{n-1}\de t \\
    &\quad\quad\quad\approx
    \lambda^{\beta\frac{n-2}2}\int_1^{2\lambda^{-1}} t^{n-1-\alpha(n-2)}\de t
    +
    \lambda^{\alpha\frac{n-2}2}
    \fullstop
  \end{align*}
  Let us approximate the value of such expression under a further assumption on the relation between $\alpha$
  and $\beta$.
  \begin{itemize}
   \item If $\alpha\ge\beta+\eps$, the expression becomes
comparable to $\lambda^{\frac{n-2}2\beta}$.
   \item If $\beta\ge\alpha+\eps$, the expression becomes comparable to $\lambda^{\frac{n-2}2\alpha}$. 
   \item If $\alpha=\beta$, the expression becomes comparable to $\lambda^{\frac{n}2}\log(\lambda^{-1})$.
  \end{itemize}
  
  \vspace{0.7em}
  \noindent\textit{The case $D\ge \lambda^{-1}$.}
  \begin{figure}[htb]
    \centering
    \footnotesize

\begin{tikzpicture}[scale=0.5]

\pgfmathsetmacro{\Zx}{3}
\pgfmathsetmacro{\Zy}{1}
\coordinate (O) at (0, 0);
\coordinate (Z) at (\Zx, \Zy);

\pgfmathsetmacro{\D}{{veclen(\Zx, \Zy)}}
\pgfmathsetmacro{\L}{1.2}

\pgfdeclarepatternformonly{north east lines wide1}%
   {\pgfqpoint{-1pt}{-1pt}}%
   {\pgfqpoint{30pt}{30pt}}%
   {\pgfqpoint{29pt}{29pt}}%
   {
     \pgfsetlinewidth{0.2pt}
     \pgfpathmoveto{\pgfqpoint{0pt}{0pt}}
     \pgfpathlineto{\pgfqpoint{29.1pt}{29.1pt}}
     \pgfusepath{stroke}
    }

\draw[pattern=north east lines wide1] (O) circle ({2*\D});
\draw[fill=white!95!black] (O) circle ({\D/2});
\draw[fill=white!95!black] (Z) circle ({\D/2});
\draw (O) circle (1);
\draw (Z) circle (\L);

\draw [fill=black] (O) circle [radius=1pt] node[above]{$0$};
\draw [fill=black] (Z) circle [radius=1pt] node[above]{$z$};

\pgfmathsetmacro{\T}{160}
\pgfmathsetmacro{\Tx}{cos(\T)*2*\D}
\pgfmathsetmacro{\Ty}{sin(\T)*2*\D}
\draw[-,dashed] (O) -- (\Tx,\Ty) node[midway,above] {$2D$};

\pgfmathsetmacro{\T}{40}
\pgfmathsetmacro{\Tx}{cos(\T)*\D/2}
\pgfmathsetmacro{\Ty}{sin(\T)*\D/2}

\pgfmathsetmacro{\T}{250}
\pgfmathsetmacro{\Tx}{cos(\T)}
\pgfmathsetmacro{\Ty}{sin(\T)}
\draw[-,dashed] (O) -- (\Tx,\Ty) node[midway,right] {$1$};

\pgfmathsetmacro{\T}{130}
\pgfmathsetmacro{\Tx}{\Zx + cos(\T)*\D/2}
\pgfmathsetmacro{\Ty}{\Zy + sin(\T)*\D/2}

\pgfmathsetmacro{\T}{280}
\pgfmathsetmacro{\Tx}{\Zx + cos(\T)*\L}
\pgfmathsetmacro{\Ty}{\Zy + sin(\T)*\L}
\draw[-,dashed] (Z) -- (\Tx,\Ty) node[midway,right] {$\lambda^{-1}$};

\node[draw,right] at (-2*\D, 2*\D) {$D\ge \lambda^{-1}$};
\node at (\D*0.30,\D*0.75) {$\frac D2$};
\draw[->] (\D*0.20,\D*0.70) -- (\D*0.10,\D*0.50);
\draw[->] (\D*0.40,\D*0.70) -- (\D*0.56,\D*0.65);

\end{tikzpicture}

    \caption{The balls $B(0,1)$, $B(0,\frac{D}2)$, $B(z,\lambda^{-1})$, $B(z,\frac{D}2)$, and $B(0,2D)$, involved in the proof of the case $D\ge\lambda^{-1}$.}
  \end{figure}
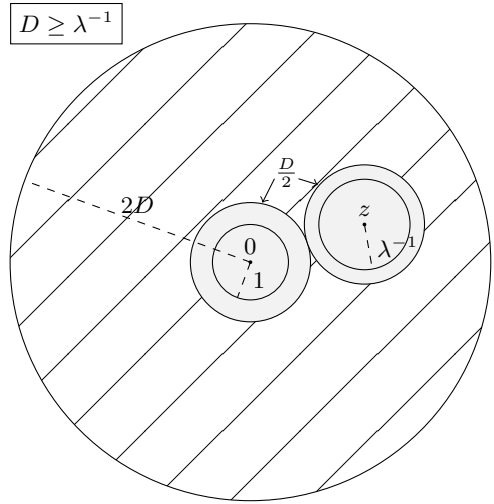
  Since we know a priori that $\lambda\le 1$, in this case we have also $D\ge 1$.
  Then, in $B(0,\frac D2)$ we have
  \begin{equation*}
    U[z,\lambda]\approx U[z,\lambda](0) \approx \left(\frac{1}{\lambda D^2}\right)^{\frac{n-2}2}\comma
  \end{equation*}
  and similarly in $B(z, \frac D2)$ we have
  \begin{equation*}
    U[0,1]\approx U[0,1](z) \approx \frac 1{D^{n-2}} \fullstop
  \end{equation*}
  We note that actually both approximations still hold inside 
 $B(0,2D)\setminus \left(B(0,\frac{D}2)\cup B(z,\frac{D}{2})\right)$, whereas in
  $B(0,2D)^{\complement}$ we have
  \begin{equation*}
    U[0,1] \approx \abs{x}^{-(n-2)} 
    \quad\text{and}\quad 
    U[z,\lambda]\approx \lambda^{-\frac{n-2}2}D^{-(n-2)} \fullstop
  \end{equation*}
  Thanks to these approximations, we can compute
  \begingroup
  \allowdisplaybreaks
  \begin{align*}
    \int_{\R^n}U[0,1]^\alpha U[z, \lambda]^\beta 
    &= \int_{B(0,\frac D2)}U[0,1]^\alpha U[z, \lambda]^\beta
    + \int_{B(z,\frac D2)}U[0,1]^\alpha U[z, \lambda]^\beta \\
    &\quad + \int_{B(0,2D)\setminus\left(B(0,\frac D2)\cup B(z,\frac D2)\right)}U[0,1]^\alpha U[z, \lambda]^\beta
    + \int_{B(0,2D)^{\complement}}U[0,1]^\alpha U[z, \lambda]^\beta \\
    &\approx
    \int_0^{\frac D2}\left(\frac1{(1+t^2)^{\frac{n-2}2}}\right)^{\alpha}
    \left(\frac{1}{\lambda^{\frac{n-2}2} D^{n-2}}\right)^{\beta}t^{n-1}\de t \\
    &\quad +
    \int_0^{\frac D2}\left(\frac1{D^{n-2}}\right)^{\alpha}
    \left(\frac{\lambda}{1+\lambda^2 t^2}\right)^{\frac{n-2}2\beta}t^{n-1}\de t \\
    &\quad + D^n\left(\frac1{D^{n-2}}\right)^{\alpha}\left(\frac{1}{\lambda^{\frac{n-2}2} D^{n-2}}\right)^{\beta}
    +\int_{2D}^{\infty}\left(\frac1{t^{n-2}}\right)^{\alpha}
    \left(\frac1{\lambda^{\frac{n-2}2}t^{n-2}}\right)^{\beta}t^{n-1}\de t \\
    &= \lambda^{-\beta\frac{n-2}2}D^{-\beta(n-2)}
    \int_0^{\frac D2}\frac{t^{n-1}}{(1+t^2)^{\alpha\frac{n-2}2}}\de t \\
    &\quad + \lambda^{-\alpha\frac{n-2}2}D^{-\alpha(n-2)}
    \int_0^{\frac{D\lambda}2}\frac{s^{n-1}}{(1+s^2)^{\beta\frac{n-2}2}}\de s \\
    &\quad + D^{-n}\lambda^{-\beta\frac{n-2}2} 
    + \lambda^{-\beta\frac{n-2}2}\int_{2D}^{\infty}\frac{1}{t^{n+1}}\de t \\
    &\approx
    \lambda^{-\beta\frac{n-2}2}D^{-\beta(n-2)}
    \int_{\frac14}^{\frac D2}t^{n-1-\alpha(n-2)}\de t \\
    &\quad + \lambda^{-\alpha\frac{n-2}2}D^{-\alpha(n-2)}
    \int_{\frac14}^{\frac{D\lambda}2}s^{n-1-\beta(n-2)}\de s \\
    &\quad + D^{-n}\lambda^{-\beta\frac{n-2}2} \fullstop
  \end{align*}
  \endgroup
  
  Once again, we approximate the value of such expression under a further assumption on the relation 
  between $\alpha$ and $\beta$.
  \begin{itemize}
   \item If $\alpha \ge \beta+\eps$, the expression becomes comparable to $(\lambda D^2)^{-\frac{(n-2)\beta}2}$.
   \item If $\beta \ge \alpha+\eps$, the expression becomes comparable to $(\lambda D^2)^{-\frac{(n-2)\alpha}2}$.
   \item If $\alpha = \beta$, since $D\lambda\ge 1$ and $\lambda\le 1$
   the expression becomes comparable to $\log(D)\,(\lambda D^2)^{-\frac{n}2}\approx \log(\lambda D^2)\,(\lambda D^2)^{-\frac{n}2}$.
  \end{itemize}
  As we have covered all possible cases, the statement is proven.
\end{proof}

\begin{proposition}\label{prop:interaction_approx}
  Given $n\ge 3$, let $U=U[z_1, \lambda_1]$ and $V=U[z_2, \lambda_2]$ be two bubbles such that 
  $\lambda_1\ge\lambda_2$.
  Let us define the quantity $Q=Q(U, V)=Q(z_1, \lambda_1, z_2, \lambda_2)$ as
  \begin{equation*}
    Q \defeq 
    \min\left(\frac{\lambda_2}{\lambda_1}, \frac1{\lambda_1\lambda_2 \abs{z_1-z_2}^2}\right) \fullstop
  \end{equation*}
  Then, for any fixed $\eps>0$ and any nonnegative exponents such that $\alpha+\beta=2^*$, it holds
  \begin{equation*}
    \int_{\R^n} U^\alpha V^\beta \approx_{n,\eps} 
    \begin{cases}
      Q^{\frac{(n-2)\min(\alpha, \beta)}2} &\text{ if }\ \abs{\alpha-\beta}\ge \eps\comma \\
      Q^{\frac n2}\log(\frac1Q) &\text{ if }\ \alpha=\beta \fullstop
    \end{cases}
  \end{equation*}
\end{proposition}
\begin{proof}
  Since the integral $\int_{\R^n} U^\alpha V^\beta$ is invariant under the transformations described in
  \cref{subsec:symmetries}, this result follows directly from \cref{lem:interaction_for_simple_bubbles}.
\end{proof}
\begin{remark}
 The behavior of $\int_{\R^n} U^\alpha V^\beta$ when $\alpha$ is close, but not equal, to $\beta$ cannot be
 captured by simple formulas as the ones in the statement of \cref{prop:interaction_approx}. Indeed
 in such a range of exponents there is the transition between a pure power behavior and a \emph{power+logarithm} behavior.
\end{remark}

\begin{corollary}\label{cor:interaction_integral_localized}
  Given $n\ge 3$ and two bubbles $U_1=U[z_1, \lambda_1]$ and $U_2=U[z_2, \lambda_2]$ with 
  $\lambda_1\ge\lambda_2$, it holds
  \begin{equation*}
    \int_{\R^n} U_1^p U_2 \approx \int_{B(z_1, \lambda_1^{-1})} U_1^p U_2 \fullstop
  \end{equation*}
\end{corollary}
\begin{proof}
  Thanks to the symmetries described in \cref{subsec:symmetries}, without loss of generality we can assume 
  $z_1=0$, $\lambda_1=1$. For clarity, let us denote $\lambda\defeq\lambda_2$ and $z\defeq z_2$.
  
  Since $\lambda\le 1$, it holds $U_2(0)\approx U_2(x)$ for any $x\in B(0,1)$. Hence we have
  \begin{equation*}
    \int_{B(0,1)}U_1^pU_2 \approx U_2(0) 
    \approx \left(\frac{\lambda}{1+\lambda^2\abs{z}^2}\right)^{\frac{n-2}2} \fullstop
  \end{equation*}
  Thanks to \cref{prop:interaction_approx} we know that
  \begin{equation*}
    \int_{\R^n} U_1^p U_2 \approx \min\left(\lambda, \frac1{\lambda\abs{z}^2}\right)^{\frac{n-2}2}
  \end{equation*}
 so the statement is proven since
  \begin{equation*}
    \min\left(\lambda, \frac1{\lambda\abs{z}^2}\right)^{\frac{n-2}2}
    \approx \left(\frac{\lambda}{1+\lambda^2\abs{z}^2}\right)^{\frac{n-2}2} \fullstop
  \end{equation*}
\end{proof}

\begin{proposition}\label{prop:integral_simple_bubbles}
  Given $n\ge 3$, let $\varphi:\oo{0}\infty^2\to\oo{0}\infty$ be a function such that
  \begin{itemize}
   \item $\varphi(x, y)\approx x^ay^b$ if $y\le 2x$,
   \item $\varphi(x, y)\approx x^cy^d$ if $x\le 2y$,
  \end{itemize}
  where $a,b,c,d\ge 0$ are nonnegative exponents satisfying $a+b=c+d>\frac n{n-2}$.
  If we denote $U\defeq U[-Re_1, 1]$ and $V\defeq U[Re_1, 1]$ for some $R\gg 1$, then it holds
  \begin{align*}
    \int_{\R^n} \varphi(U, V) 
    \approx \Phi_R(a, b, c, d) &\defeq R^{-b(n-2)}\int_1^R t^{n-1-a(n-2)}\de t \\
    &+ R^{-c(n-2)}\int_1^R t^{n-1-d(n-2)}\de t \\
    &+ R^{n-(a+b)(n-2)}
    \comma
  \end{align*}
  where the hidden constants do not depend on $R$ (but are allowed to depend on the dimension $n$ and the 
  function $f$). As a consequence, it holds
  \begin{equation*}
    \Phi_R(a, b, c, d) \approx
    \begin{cases}
      R^{-\min(b, c)(n-2)} \quad&\text{if $\max(a,d)>\frac{n}{n-2}$,}\\
      R^{-\min(b, c)(n-2)}\log(R) \quad&\text{if $\max(a,d)=\frac{n}{n-2}$,}\\
      R^{n-(a+b)(n-2)} \quad&\text{if $\max(a,d)<\frac{n}{n-2}$.}
    \end{cases}
  \end{equation*}
\end{proposition}
\begin{proof}
  The only properties of $U[0,1]$ we are going to use are that $U[0,1]\approx 1$ in $B(0,1)$ and 
  $U[0,1]\approx \abs{x}^{2-n}$ elsewhere.
  
  We split the desired integral in four zones
  \begin{align*}
    \int_{\R^n} \varphi(U, V) =
    \int_{B_U} \varphi(U, V) 
    +\int_{B_V} \varphi(U, V) 
    +\int_{B(0, 2R)\setminus(B_U\cup B_V)} \varphi(U, V)
    +\int_{B(0, 2R)^\complement} \varphi(U, V) \comma
  \end{align*}
  where $B_U = B(-Re_1, \frac R2)$ and $B_V=B(R e_1, \frac R2)$.
  Let us compute the four terms separately.
  
  We start with the integral on $B_U$:
  \begin{align*}
    \int_{B_U} \varphi(U, V) 
    &\approx \int_{B_U} U^a V^b 
    \approx R^{-b(n-2)}\int_{B(0,\frac R2)} U[0,1]^a\\
   & \approx R^{-b(n-2)}\left(1+\int_1^R \frac{t^{n-1}}{t^{a(n-2)}}\de t\right) 
    \approx R^{-b(n-2)}\int_1^R t^{n-1-a(n-2)}\de t \fullstop
  \end{align*}
  Of course, we can perform an analogous computation also on the ball $B_V$.
  
  In $B(0, 2R)\setminus(B_U\cup B_V)$ it holds $U\approx V\approx R^{-(n-2)}$ and thus 
  $\phi(U,V)\approx R^{-(a+b)(n-2)}$. Therefore
  \begin{equation*}
    \int_{B(0, 2R)\setminus(B_U\cup B_V)} \varphi(U, V) \approx R^{n-(a+b)(n-2)} \fullstop
  \end{equation*}
Finally, outside 
  $B(0, 2R)$ it holds
  \begin{equation*}
    \int_{B(0, 2R)^\complement} \varphi(U, V) 
    \approx \int_{B(0, 2R)^\complement} \abs{x}^{-(a+b)(n-2)}\de x 
    \approx R^{n-(a+b)(n-2)}
    \fullstop
  \end{equation*}
  Combining all these estimates, the result follows.
\end{proof}

\bigskip

\noindent
\textit{Acknowledgments:} both authors are funded by the European Research Council under the Grant 
Agreement No. 721675 ``Regularity and Stability in Partial Differential Equations (RSPDE)''.

\bigskip

\printbibliography

\end{document}